\documentclass[12pt,a4paper,nosumlimits]{amsart}
\usepackage{amssymb,amsxtra,eucal}

\calclayout

\newcommand{\+}{\nobreakdash-}
\renewcommand{\:}{\colon}
\renewcommand{\;}{,\medspace}
\renewcommand{\.}{\text{$\mskip .5\thinmuskip$}}

\renewcommand{\le}{\leqslant}
\renewcommand{\ge}{\geqslant}

\newcommand{\rarrow}{\longrightarrow}
\newcommand{\larrow}{\longleftarrow}
\newcommand{\birarrow}{\rightrightarrows}
\newcommand{\ot}{\otimes}
\newcommand{\ocn}{\odot}
\newcommand{\st}{\star}

\newcommand{\bu}{{\text{\smaller\smaller$\scriptstyle\bullet$}}}
\newcommand{\lrarrow}{\.\relbar\joinrel\relbar\joinrel\rightarrow\.}

\newcommand{\B}{{\mathcal B}}
\newcommand{\C}{{\mathcal C}}
\newcommand{\D}{{\mathcal D}}
\renewcommand{\O}{{\mathcal O}}
\newcommand{\M}{{\mathcal M}}
\newcommand{\J}{{\mathcal J}}

\newcommand{\fR}{{\mathfrak R}}

\newcommand{\boR}{{\mathbb R}}
\newcommand{\boL}{{\mathbb L}}
\newcommand{\Z}{{\mathbb Z}}

\newcommand{\s}{{\mathbf s}}

\renewcommand{\b}{{\mathsf{b}}}
\newcommand{\co}{{\mathsf{co}}}
\newcommand{\ctr}{{\mathsf{ctr}}}
\newcommand{\abs}{{\mathsf{abs}}}
\newcommand{\inj}{{\mathsf{inj}}}
\newcommand{\proj}{{\mathsf{proj}}}

\newcommand{\sA}{{\mathsf A}}
\newcommand{\sB}{{\mathsf B}}
\newcommand{\sC}{{\mathsf C}}
\newcommand{\sD}{{\mathsf D}}
\newcommand{\sE}{{\mathsf E}}
\newcommand{\sJ}{{\mathsf J}}
\newcommand{\sP}{{\mathsf P}}

\newcommand{\Hot}{{\mathsf{Hot}}}

\newcommand{\modl}{{\operatorname{\mathsf{--mod}}}}
\newcommand{\comodl}{{\operatorname{\mathsf{--comod}}}}
\newcommand{\discr}{{\operatorname{\mathsf{--discr}}}}
\newcommand{\contra}{{\operatorname{\mathsf{--contra}}}}
\newcommand{\qcoh}{{\operatorname{\mathsf{--qcoh}}}}
\newcommand{\ctrh}{{\operatorname{\mathsf{--ctrh}}}}
\newcommand{\tors}{{\operatorname{\mathsf{-tors}}}}
\newcommand{\ctra}{{\operatorname{\mathsf{-ctra}}}}
\newcommand{\com}{{\operatorname{\mathsf{-com}}}}
\newcommand{\adj}{{\operatorname{\mathsf{-adj}}}}

\DeclareMathOperator{\Spec}{Spec}
\DeclareMathOperator{\Hom}{Hom}
\DeclareMathOperator{\Ext}{Ext}
\DeclareMathOperator{\Tor}{Tor}
\DeclareMathOperator{\Ctrtor}{Ctrtor}
\DeclareMathOperator{\Tel}{Tel}
\DeclareMathOperator{\cone}{cone}
\DeclareMathOperator{\cocone}{cocone}
\DeclareMathOperator{\im}{im}
\DeclareMathOperator{\id}{id}
\DeclareMathOperator{\Id}{id}
\DeclareMathOperator{\pd}{pd}
\DeclareMathOperator{\fd}{fd}
\DeclareMathOperator{\cfd}{cfd}

\DeclareMathOperator{\fHom}{\mathfrak{Hom}}

\newcommand{\Section}[1]{\bigskip\section{#1}\medskip}
\theoremstyle{plain}
\newtheorem{thm}{Theorem}[section]
\newtheorem{lem}[thm]{Lemma}
\newtheorem{prop}[thm]{Proposition}
\newtheorem{cor}[thm]{Corollary}
\theoremstyle{definition}
\newtheorem{ex}[thm]{Example}

\newtheorem{rem}[thm]{Remark}

\begin{document}

\title{Dedualizing complexes and MGM duality}

\author{Leonid Positselski}

\address{Department of Mathematics, Faculty of Natural Sciences,
University of Haifa, Mount Carmel, Haifa 31905, Israel; and
\newline\indent Laboratory of Algebraic Geometry, National Research
University Higher School of Economics, Moscow 117312; and
\newline\indent Sector of Algebra and Number Theory, Institute for
Information Transmission Problems, Moscow 127051, Russia}

\email{posic@mccme.ru}

\begin{abstract}
 We show that various derived categories of torsion modules
and contramodules over the adic completion of a commutative ring
by a weakly proregular ideal are full subcategories of the related
derived categories of modules.
 By the work of Dwyer--Greenlees and Porta--Shaul--Yekutieli, this
implies an equivalence between the (bounded or unbounded)
conventional derived categories of the abelian categories of torsion
modules and contramodules.
 Over the adic completion of a commutative ring by an arbitrary
finitely generated ideal, we obtain an equivalence between
the derived categories of complexes of modules with torsion and
contramodule cohomology modules.
 We also define two versions of the notion of
a \emph{dedualizing complex} over the adic completion of
a commutative ring, one for an ideal with an Artinian quotient ring
and the other one for a weakly proregular ideal, and use these
to construct equivalences between the conventional as well as
certain exotic derived categories of the abelian categories
of torsion modules and contramodules.
 The philosophy of derived co-contra correspondence is discussed 
in the introduction.
\end{abstract}

\maketitle

\tableofcontents

\setcounter{section}{-1}
\section{Introduction}
\medskip

\setcounter{subsection}{-1}
\subsection{{}}
 In its simplest and purest form, the \emph{comodule-contramodule
correspondence} is a natural equivalence between the additive
categories of injective left comodules and projective left
contramodules over the same coassociative coalgebra $\C$ over
a field~$k$.
 The equivalence of additive categories induces an equivalence
between the homotopy categories of complexes of injective comodules
and projective contramodules.
 One would like to view this equivalence as an equivalence between
the derived categories of left comodules and left contramodules
over $\C$, using complexes of injective comodules and projective
contramodules as resolutions.
 However, the former are right resolutions and the latter are left
ones, so one has to work with doubly unbounded complexes.
 The categories of unbounded complexes of injective or projective
objects are well-known to differ from the unbounded derived
categories as the latter are conventionally defined.
 The reason is that an unbounded complex of injectives or projectives
may be acyclic, yet not contractible.

 Hence the relevance of the concepts of the \emph{derived categories
of the second kind}, or the \emph{coderived} and \emph{contraderived}
categories, in the comodule-contramodule correspondence constructions.
 In well-behaved situations, the coderived category $\sD^\co(\sA)$
of an abelian or exact category $\sA$ is equivalent to the homotopy
category $\Hot(\sA_\inj)$ of complexes of injective objects in $\sA$,
while the contraderived category $\sD^\ctr(\sB)$ of a category $\sB$
is equivalent to the homotopy category $\Hot(\sB_\proj)$ of complexes
of projective objects in~$\sB$.
 Thus the \emph{derived co-contra correspondence} over a coalgebra
$\C$ takes the form of an equivalence between the coderived category
of the abelian category of left $\C$\+comodules $\C\comodl$ and
the contraderived category of the abelian category of left
$\C$\+contramodules $\C\contra$,
$$
 \sD^\co(\C\comodl)\simeq\sD^\ctr(\C\contra).
$$
 Certain acyclic complexes of comodules and contramodules survive
in the coderived and contraderived categories, and an acyclic
complex is sometimes assigned to an irreducible object by
the derived co-contra correspondence~\cite[Sections~0.2.2, 0.2.5,
and~0.2.6\+-7]{Psemi} (see also~\cite[Sections~4.4 and~5.2]{Pkoszul}).

\subsection{{}}
 In more complicated relative situations mixing algebra and coalgebra
features, derived co-contra correspondence theories can be often
developed using the following guiding principles~\cite{Psli}:
\begin{enumerate}
\renewcommand{\theenumi}{\roman{enumi}}
\item depending on whether one's abelian/exact category is a category
of comodule-like or contramodule-like objects, one takes the coderived
or the contraderived category along the coalgebra variables;
\item one takes the conventional unbounded derived category along
the ring or algebra variables;
\item over a ring or coalgebra of finite homological dimension, there
is no difference between the derived, coderived, and contraderived
categories, so one does not have to pay attention to the distinction.
\end{enumerate}
 
 The rules (i\+iii) are sufficient to build most of the derived
the comodule-contra\-module and semimodule-semicontramodule
correspondence theorems of the book~\cite{Psemi}
(see~\cite[Sections~0.3.7, 5.4\+-5.5, and~6.3]{Psemi}).
 However, there are several simple and important situations which they
do not cover.

\subsection{{}}
 In the papers~\cite{Jor,Kra}, the homotopy categories of unbounded
complexes of projective and injective modules over certain coherent
or Noetherian rings were studied; and in~\cite{IK}, an equivalence
between the homotopy categories of complexes of projective and
injective modules over a \emph{ring} (or two rings, in
the noncommutative case) \emph{with a dualizing complex} was obtained.
 In the paper~\cite{Neem} and the dissertation~\cite{Murf}, these
results were extended to complexes of flat and injective quasi-coherent
sheaves on a Noetherian scheme; and in
the paper~\cite[Section~2.5]{EP}, to matrix factorizations.
 Finally, in~\cite[Section~5.7]{Pcosh} this duality is formulated as
a commutative diagram of equivalences between \emph{four} exotic derived
categories of quasi-coherent sheaves and contraherent cosheaves on
a semi-separated Noetherian scheme.

 The related piece of philosophy appears to look as follows.
 A coalgebra has been ``dualized'' already, being a dual thing to
an algebra or a ring; so a coalgebra $\C$ is a \emph{dualizing
complex over itself}.
 The conventional derived category of modules over a ring is
tautologically equivalent to itself, but constructing an equivalence
between the coderived and contraderived categories of modules over
a (say, commutative Noetherian) ring requires a dualizing complex. 
 Having a dualizing complex over a ring makes it ``more like
a coalgebra''.
 Therefore, 
\begin{enumerate}
\renewcommand{\theenumi}{\roman{enumi}}
\setcounter{enumi}{3}
\item given a dualizing complex for a set of ring or algebra variables,
one can use the coderived and the contraderived category along these
variables on the two sides of the derived co-contra correspondence.
\end{enumerate}

\subsection{{}} \label{naive-quasi-contrah}
 The opposite situation, when one considers the conventional unbounded
derived categories of comodules and contramodules, is not as
well-studied.
 The relation between the conventional derived categories of
DG\+comodules and DG\+contramodules over an arbitrary DG\+coalgebra
over a field is discussed in (particularly, the postpublication
\texttt{arXiv} version of) \cite[Section~5.5]{Pkoszul}, but it does not
have the familiar form of an equivalence of triangulated categories.

 In~\cite[Section~4.6]{Pcosh}, we construct an equivalence between
the conventional (bounded or unbounded; and also absolute, etc.)\
derived categories of the abelian category of quasi-coherent sheaves
and the exact category of contraherent cosheaves on a quasi-compact
semi-separated scheme~$X$,
$$
 \sD^\st(X\qcoh)\simeq\sD^\st(X\ctrh).
$$
 An explanation is that, generally, a scheme ``mixes the ring and
coalgebra variables'' in such a way that the ring(s) are responsible
for the pieces being glued, while the coalgebra (or, rather, coring)
governs the gluing procedure~\cite{KR}.
 In a quasi-compact quasi-separated scheme, the gluing procedure
has finite homological dimension (the quasi-coherent sheaf cohomology
being a functor of finite homological dimension); so the rule~(iii)
``along the gluing variables'' applies.
 Still, one additionally has to explain how to pass from the sheaves to
the cosheaves and back.
 Over a scheme, one takes the contraherent $\fHom_X$ from or
the contratensor product $\ocn_X$ with the structure sheaf $\O_X$ of
the scheme~$X$ \cite[Sections~2.5\+-6]{Pcosh}.

\subsection{{}}
 On the other hand, there is a stream of literature discussing torsion,
completion, and duality theories for modules or sheaves over formal
schemes.
 From the point of view elaborated in the above discussion, there seem
to be \emph{two} such duality theories that need to be property
distinguished from each other.

 One of them is concerned with dualizing complexes on formal
schemes~\cite[Section~5]{Yek0} and, being formulated in the generality
of unbounded complexes of infinitely generated modules, leads to
a covariant equivalence between derived categories of the second kind,
extending the results of~\cite{IK,Murf,EP} and~\cite[Section~5.7]{Pcosh}
to the realm of formal schemes.
 For affine Noetherian formal schemes, this
``covariant Serre--Grothendieck duality'' is formulated
in~\cite[Section~C.1]{Pcosh} (see~\cite[Section~C.5]{Pcosh} for
a noncommutative version); and for ind-affine ind-Noetherian
ind-schemes (the ind-spectra of pro-Noetherian topological rings),
in~\cite[Section~D.2]{Pcosh}.

 To demonstrate a precise assertion here, given a dualizing complex
$\D^\bu$ for a projective system $R_0\larrow R_1\larrow R_2\larrow\dotsb$
of Noetherian commutative rings and surjective morphisms between them
with the projective limit $\fR=\varprojlim_n R_n$, \
\cite[Theorem~D.2.7]{Pcosh}~claims that there is a natural equivalence
between the coderived category of discrete modules and the contraderived
category of contramodules over the topological ring~$\fR$,
$$
 \sD^\co(\fR\discr)\simeq\sD^\ctr(\fR\contra).
$$

 The other duality theory over adically complete rings and formal
schemes is known as the \emph{Matlis--Greenlees--May
duality}~\cite{Mat,DG,PSY}.
 The aim of the present paper is to formulate it explicitly as
an equivalence between the conventional derived categories of
the abelian categories of torsion modules and contramodules over
the adic completions of certain commutative rings.
 Thus the MGM duality is viewed as a species of the ``na\"\i ve
derived co-contra correspondence'' of~\cite[Section~4.6]{Pcosh}.

\subsection{{}}  \label{geometric-subcategories-equivalence}
 Let us devote a few paragraphs to a more substantive discussion of
the issues involved.
 Let $X$ be a scheme and $Z\subset X$ be a closed subscheme.
 Denote by $U=X\setminus Z$ the open complement to $Z$ in~$X$.
 Then the formal completion of $X$ along $Z$ can be viewed heuristically
as the complement to $U$ in $X$,
$$
 X_Z\sphat = X\setminus U.
$$
 In particular, the quasi-coherent torsion sheaves on $X_Z\sphat$ are,
almost by the definition, those quasi-coherent sheaves on $X$ whose
restrictions vanish on~$U$.

 Passing to the triangulated categories and denoting by $k\:U\rarrow X$
the open embedding morphism, we notice that under weak assumptions
on~$k$ the inverse image functor $k^*\:\sD(X\qcoh)\rarrow\sD(U\qcoh)$
between the derived categories of quasi-coherent sheaves on $X$ and $U$
has a right adjoint functor of derived direct image $\boR k_*\:
\sD(U\qcoh)\rarrow\sD(X\qcoh)$.
 The composition $k^*\circ\boR k_*$ is isomorphic to the identity
functor on $\sD(U\qcoh)$, so $k^*$~is a Verdier quotient functor and
the functor $\boR k_*$~is fully faithful.
 To be more precise, one would probably want to have the scheme $U$
quasi-compact and the scheme $X$ semi-separated, or otherwise
the scheme $X$ locally Noetherian, for the functor $\boR k_*$ to be
well-behaved; and under slightly stronger assumptions one can
prove existence of a Neeman extraordinary inverse image functor
$k^!\:\sD(X\qcoh)\rarrow\sD(U\qcoh)$ right adjoint to~$\boR k_*$
\cite{Neem0}.

 It follows immediately that the kernels of the functors~$k^*$
and~$k^!$ are two equivalent subcategories in $\sD(X\qcoh)$
(a purely algebraic description of this picture can be found
in~\cite[Section~8]{PSY}).
 Furthermore, the functor of inverse image of contraherent cosheaves
$k^!\:\sD(X\ctrh)\rarrow\sD(U\ctrh)$ has a left adjoint functor of
derived direct image $\boL k_!\:\sD(U\ctrh)\rarrow\sD(X\ctrh)$.
 The equivalences of derived categories $\sD(X\qcoh)\simeq\sD(X\ctrh)$
and $\sD(U\qcoh)\simeq\sD(U\ctrh)$ from
Section~\ref{naive-quasi-contrah} transform the functor $\boR k_*$ into
the functor $\boL k_!$, identifying the two functors~$k^!$ and providing
a rather explicit construction of Neeman's inverse
image~\cite[Section~4.8]{Pcosh}.

 This is not yet a promised equivalence between the derived categories
of a pair of abelian or exact categories, however.
  E.~g., one still has to prove that the kernel of the functor~$k^*$
is equivalent to the derived category $\sD(X\qcoh_{Z\tors})$ of
quasi-coherent torsion sheaves on $X_Z\sphat$, or that the triangulated
functor $\sD(X\qcoh_{Z\tors})\rarrow\sD(X\qcoh)$ induced by the embedding
of abelian categories $X\qcoh_{Z\tors}\rarrow X\qcoh$ is fully faithful.
 For bounded derived categories $\sD^\b$ or $\sD^+$ in place of $\sD$
and a Noetherian scheme $X$ this is not difficult, as one can use
the Artin--Rees lemma to show that injectives in $X\qcoh_{Z\tors}$ are
also injective in $X\qcoh$; for the unbounded derived categories,
the question appears to be more involved.
 One also has to identify the kernel of the functor~$k^!$ with
the derived category of the abelian or exact category of contramodules
or contraherent cosheaves of contramodules on~$X_Z\sphat$.

\subsection{{}}
 Let us now continue the discussion from a different angle.
 When the scheme $X$ is quasi-compact and quasi-separated, and its
open subscheme $U\subset X$ is quasi-compact, the above heuristics
suggest that the formal scheme $X_Z\sphat$, being the result of
``subtracting'' $U$ from $X$, has the homological dimension of its
``coalgebra variables'' bounded, approximately, by the sum of
the cardinality parameters of the open coverings of $X$ and $U$,
and consequently finite.
 To use another visual metaphor, one can say that the formal completion
of $X$ along $Z$ is obtained by cutting from $X$ a small tubular
neighborhood around~$Z$.
 The observation is that the scissors used to perform the cut tend
to have finite homological dimension.

 Once again, the Artin--Rees lemma, claiming that the completion functor
is exact, would seem to provide an even stronger assertion for finitely
generated modules, but we are interested in infinitely generated ones
(even over a Noetherian ring).
 The precise conditions needed to prove such a result are a more
technical issue.
 In the case of an affine scheme $X=\Spec R$ our heuristics seem to
suggest that the open subscheme $U=X\setminus Z$ should be
quasi-compact, i.~e., the ideal $I\subset R$ defining the closed
subscheme $Z\subset X$ should be finitely generated.
 In assumptions, admittedly, stronger than this, still weaker than
the Noetherianity of $R$, a pair of such homological
finite-dimensionality assertions on the comodule and contramodule
sides is indeed proven in~\cite[Corollaries~4.28 and~5.27]{PSY}.

\subsection{{}}
 No homological dimension condition, though, can automatically solve
the problem of constructing an equivalence between the derived
categories of two different abelian categories, even if these are
the categories of comodules and contramodules over the same
coalgebra-like algebraic structure.
 The piece of data that is missing here is what we call
a \emph{dedualizing complex}. 

 The coderived category of left comodules and the contraderived
category of left contramodules over a coalgebra over a field are
always equivalent to each other; but constructing an equivalence
between the conventional derived categories of comodules and
contramodules over a (co-Noetherian or cocoherent) coalgebra
requires a dedualizing complex.
 A ring is a dedualizing complex over itself; having a dedualizing
complex for a coalgebra makes it more like a ring.
 In other words,
\begin{enumerate}
\renewcommand{\theenumi}{\roman{enumi}}
\setcounter{enumi}{4}
\item given a dedualizing complex for a set of coalgebra variables,
one can have a ``naive'' derived co-contra correspondence with
the conventional derived categories of comodules and contramodules
along these variables on the two sides of a triangulated equivalence.
\end{enumerate}

 The definition of a dedualizing complex $\B^\bu$ is, approximately,
dual to that of a dualizing one.
 It has to be a finite complex of comodules or bicomodules.
 There being, generally speaking, no projective objects in comodule
categories, one cannot ask $\B^\bu$ to be a finite complex
\emph{of projective objects}, but it has to have \emph{finite
projective dimension} as a complex in the bounded derived category.
 It has to satisfy a unitality condition imposed on the graded ring
of its endomorphisms in the derived category. 
 Finally, there should be a finiteness condition similar or dual to
the coherence condition on the cohomology sheaves of a dualizing
complex on a Noetherian scheme.
 The latter one seems to be the hardest to formulate, and
the approaches may vary.

 Let us emphasize that the dualizing and the dedualizing complexes play
rather different roles in the respective theories.
 The purpose of the dualizing complex is to mitigate the infinite
homological dimension problem, bridging the gap between the coderived
and the contraderived category.
 The dedualizing complex is there to bridge the gap between
the abelian categories of comodules and contramodules.
 In both cases, the problem to be solved by the choice of a specific
complex appears due to a mismatch between the kind of abelian
categories being considered and the kind of derived category
constructions being applied.
 Between the coderived category of comodules and the contraderived
category of contramodules (say, over a coalgebra over a field),
there is no gap to be bridged.

\subsection{{}}
 We have yet to explain how a dedualizing complex on a formal scheme
is to be obtained.
 For dualizing complexes on algebraic varieties $X$ over a field~$k$,
the classical prescription is to consider the structure morphism
$p\:X\rarrow\Spec k$ and set $\D_X^\bu=p^+\O_{\Spec k}$, where
$p^+$~denotes the Deligne extraordinary inverse image functor
(i.~e., $f^+$~is equal to~$f^!$ for proper morphisms and to~$f^*$
for open embeddings).

 A similar rule works for formal completions of algebraic varieties:
given a scheme $X$ with a dualizing complex $\D_X^\bu$ and a closed
subscheme $Z\subset X$, a dualizing complex on the formal scheme
$X_Z\sphat$ can be constructed as the derived subcomplex with
set-theoretic supports $\D_{X_Z\sphat}^\bu=\boR i^!\D_X^\bu$,
where $i$~denotes the embedding $X_Z\sphat\rarrow X$.
 As to the \emph{dedualizing} complex on $X_Z\sphat$, it is produced
by applying the functor $\boR i^!$ to the structure sheaf
of the scheme~$X$,
$$
 \B_{X_Z\sphat}^\bu=\boR i^!\O_X,
 \qquad i\:X_Z\sphat\rarrow X.
$$
 Note that the formal completion $X_Z\sphat\subset X$ of a closed
subscheme $Z$ in a scheme $X$ can be considered both as a kind of
``closed subscheme'' or ``open subscheme'' in $X$, depending on
a point of view.
 The above construction of the dedualizing complex can be explained
by saying that one would ``prefer'' to use the $*$\+restriction of
quasi-coherent sheaves in order to obtain a dedualizing complex
on $X_Z\sphat$ from a dedualizing complex on $X$; but given that
a dedualizing complex on a formal scheme must be a complex of
quasi-coherent torsion sheaves, one is ``forced'' to use
the $!$\+restriction, which is a ``satisfactory substitution''
inasmuch as the natural morphism $X_Z\sphat\rarrow X$ can be viewed
as a species of open embedding.

 The complex $\B_{X_Z\sphat}^\bu$ for the formal completion of an affine
scheme $X=\Spec R$ along its closed subscheme $Z\subset X$
defined by a finitely generated ideal $I\subset R$ plays a crucial
role in the MGM duality papers~\cite{DG,PSY}.
 Traditionally, it is viewed as an object of the derived category of
$R$\+modules represented by an explicit Koszul/telescope complex of
infinitely generated free $R$\+modules $\Tel^\bu(R,\s)$ concentrated
in the cohomological degrees from~$0$ to~$m$, where
$\s=\{s_1,\dotsc,s_m\}$ is a chosen set of generators of the ideal~$I$,
and having $I$\+torsion cohomology modules.
 With the perspective of extending the MGM duality to nonaffine
schemes in mind, certainly it is preferable to work with torsion
modules and contramodules over the $I$\+adic completion of $R$ only,
avoiding the use of any other $R$\+modules such as the infinitely
generated free ones.
 Dependence on a fixed set of generators~$\s$ of the ideal $I$,
which may not exist globally along a closed subscheme $Z$ in
a nonaffine scheme $X$, is also undesirable.

\subsection{{}}  \label{description-of-results}
 Let us now describe the contents of this paper and our main results
in some detail.
 For a commutative ring $R$ with a \emph{weakly proregular} finitely
generated ideal $I\subset R$, \ Porta, Shaul, and Yekutieli construct
an equivalence between (what they call) the full subcategories of
\emph{cohomologically $I$\+torsion complexes}
$\sD(R\modl)_{I\tors}\subset\sD(R\modl)$ and
\emph{cohomologically $I$\+adically complete complexes}
$\sD(R\modl)_{I\com}\subset\sD(R\modl)$ in the unbounded derived
category of $R$\+modules $\sD(R\modl)$,
\begin{equation}  \label{subcategories-equivalent}
 \sD(R\modl)_{I\tors}\simeq\sD(R\modl)_{I\com}.
\end{equation}
\cite[Theorems~1.1 and~7.11]{PSY}.
 Moreover, they consider the full abelian subcategory
of $I$\+torsion $R$\+modules $R\modl_{I\tors}\subset R\modl$
in the abelian category of $R$\+modules, and show
that the full subcategory of cohomologically $I$\+torsion
complexes coincides with the full subcategory of complexes
with $I$\+torsion cohomology modules in $\sD(R\modl)$,
$$
 \sD(R\modl)_{I\tors}=\sD_{I\tors}(R\modl)\subset\sD(R\modl)
$$
\cite[Corollary~4.32]{PSY}; a similar result was earlier
obtained in~\cite[Proposition~6.12]{DG}.

 Our first result in Section~\ref{derived-torsion-modules} is that,
for a weakly proregular finitely generated ideal $I$ in a commutative
ring $R$, the derived category of the abelian category of $I$\+torsion
$R$\+modules is a full subcategory in $\sD(R\modl)$ coinciding
with the full subcategory of complexes with $I$\+torsion
cohomology modules,
$$
 \sD_{I\tors}(R\modl)=\sD(R\modl_{I\tors})\subset\sD(R\modl).
$$
 Furthermore, in Section~\ref{derived-contramodules} we define
the full abelian subcategory of \emph{$I$\+contramodules} (or
\emph{$I$\+contramodule $R$\+modules}) $R\modl_{I\ctra}\subset
R\modl$ and show that the derived category of $I$\+contramodule
$R$\+modules is a full subcategory of the derived category of
$R$\+modules,
$$
 \sD(R\modl_{I\ctra})\subset\sD(R\modl).
$$
 Moreover, the full subcategory of cohomologically $I$\+adically
complete complexes coincides with the full subcategory of complexes
with $I$\+contramodule cohomology modules in $\sD(R\modl)$ and with
(the image of) the derived category of the abelian category
of $I$\+contramodule $R$\+modules,
$$
 \sD(R\modl)_{I\com}=\sD_{I\ctra}(R\modl)=\sD(R\modl_{I\ctra}).
$$
(concerning the first equality, cf.~\cite[Proposition~6.15]{DG}).

 Combining these results, we obtain a natural equivalence between
the (bounded or unbounded) conventional derived categories of
the abelian categories of $I$\+torsion and $I$\+contramodule
$R$\+modules,
\begin{equation} \label{derived-categories-equivalent}
 \sD^\st(R\modl_{I\tors})\simeq\sD^\st(R\modl_{I\ctra}).
\end{equation}

 Let us point out that the equivalence of triangulated
subcategories~\eqref{subcategories-equivalent} in the derived
category $\sD(R\modl)$ can be actually established for any
finitely generated ideal $I$ in a commutative ring $R$, as it
follows from the above discussion in
Section~\ref{geometric-subcategories-equivalence}.
 This was shown already by Dwyer and Greenlees in~\cite[Theorem~2.1,
Section~4.1, and Section~6]{DG}, and we explain this anew in our
Section~\ref{duality-theorem-secn}.
 In fact, we prove that for any finitely generated ideal $I\subset R$
the full subcategories of complexes with torsion and contramodule
cohomology modules are naturally equivalent,
\begin{equation}
 \sD^\st_{I\tors}(R\modl)\simeq\sD^\st_{I\ctra}(R\modl).
\end{equation}
 However, the weak proregularity condition appears to be essential
for constructing an equivalence between the derived categories of
the abelian categories $R\modl_{I\tors}$ and $R\modl_{I\ctra}$ in
the formula~\eqref{derived-categories-equivalent}.

\subsection{{}}
 Notice that the formal completions of algebraic varieties over
a field~$k$ at their closed points defined over~$k$ are precisely
the ind-spectra of finitely cogenerated conilpotent cocommutative
coalgebras over~$k$.
 (See~\cite[Section~I.6]{Dem} for a discussion of what we would
now call the equivalence between the categories of ind-finite
ind-schemes over a field and cocommutative coalgebras over
this field.)
 Thus coalgebras occur in connection with closed points on
algebraic varieties, while the adic completions with respect to
ideals correspond to arbitrary closed subsets.

 As the above discussion in this introduction suggests, the notion of
a dedualizing complex in the case of (not even necessarily
cocommutative, but coassociative) coalgebras over a field is easier
to approach than for adic completions of commutative rings.
 In fact, thinking about the coalgebra case was an important
stepping stone for the development of the author's ideas about
the subject.

 Still, from the point of view of many a reader the necessity to
delve into an additional, largely unrelated background and recall
the basics of coalgebras on the way to the desired definition
from commutative algebra may feel like an unnecessary burden.
 Therefore, we decided to move our treatment of the MGM duality
for coalgebras to a separate paper~\cite{Pmc}, while restricting
ourselves here to a brief mention of this theory and its main
result in this section of the introduction.

 The definition of a dedualizing complex as a complex of bicomodules
over two coalgebras satisfying a list of conditions is presented
in~\cite{Pmc}.
 Given a dedualizing complex $\B^\bu$ for a pair of cocoherent
coalgebras $\C$ and $\D$, an equivalence between the bounded or
unbounded, conventional or absolute derived categories of left
comodules over $\C$ and left contramodules over $\D$ is obtained, 
$$
 \sD^\st(\C\comodl)\simeq\sD^\st(\D\contra).
$$

 Instead of the coalgebra story, we have included in this paper
an intermediate Section~\ref{dedualizing-artinian-quotient}, where
a theory very similar to the MGM duality for coalgebras is developed
in the case of a Noetherian ring $R$ with an ideal $I$ such that
the quotient ring $R/I$ is Artinian (cf.~\cite[Section~1.1]{Pweak}).
 This assumption, even though very restrictive, allows to give
a conceptually clear definition of a dedualizing complex, which
can serve as a prototype for the more clumsy condition applicable
in the general case.

 In the final Section~\ref{dedualizing-weakly-proregular}, we work out
the dedualizing complexes-based approach to MGM duality theory
over the adic completions of commutative rings by arbitrary
weakly proregular finitely generated ideals.
 Even though our definition of a dedualizing complex in this setting
is not as neat as in the cases of two coalgebras or an ideal with
the Artinian quotient ring, it allows to obtain a second proof of
the equivalence of derived
categories~\eqref{derived-categories-equivalent}.
 Moreover, while the proof based on the arguments of
Section~\ref{duality-theorem-secn} applies to the conventional
derived categories only, the approach in
Section~\ref{dedualizing-weakly-proregular} produces an equivalence
of the absolute derived categories as well.

 A brief summary of the definitions of exotic derived categories
used in the main body of the paper is included in
Appendix~\ref{exotic-derived}.
 A formalism of derived functors needed for the constructions of
the last two sections is developed in
Appendix~\ref{derived-finite-homol-dim}.

 For a general discussion of the abelian categories of contramodules,
we refer to the survey paper~\cite{Prev}.
 In the role of an introduction to the coderived, contraderived and
absolute derived categories the reader can use~\cite[Sections~0.2
and~2.1]{Psemi} or, better yet, \cite[Sections~3\+-5]{Pkoszul};
for more advanced results, one can look
into~\cite[Sections~1.3\+-1.6]{EP} or~\cite[Appendix~A]{Pcosh}.
 An overview exposition on the derived comodule-contramodule
correspondence phenomenon can be found in the presentation~\cite{Psli}.

\subsection*{Acknowlegdement}
 The author is grateful to all the people who brought the MGM duality
to his attention over the several recent years, including
Sergey Arkhipov, Joseph Bernstein, and Amnon Yekutieli.
 This work was started when I~was visiting Ben Gurion University of
the Negev in Be'er Sheva in the Summer of 2014 and continued when
I~was supported by a fellowship from the Lady Davis Foundation at
the Technion in October 2014--March 2015.
 I~also wish to thank the anonymous referee for several helpful
suggestions.

\Section{Derived Category of Torsion Modules}
\label{derived-torsion-modules}

 Let $R$ be a commutative ring and $I\subset R$ be an ideal.
 An $R$\+module $M$ is said to be \emph{$I$\+torsion} if for every
pair of elements $s\in I$ and $x\in M$ there exists an integer
$n\ge1$ such that $s^nx=0$ in~$M$.
 The full subcategory of $I$\+torsion submodules in $R\modl$ is
denoted by $R\modl_{I\tors}$.
 Clearly, the subcategory $R\modl_{I\tors}$ is closed under
the operations of passage to submodules, quotient modules,
extensions, and infinite direct sums in $R\modl$.
 Therefore, $R\modl_{I\tors}$ is an abelian category with exact functors
of infinite direct sum and its embedding $R\modl_{I\tors}\rarrow
R\modl$ is an exact functor preserving the infinite direct sums.

 Following~\cite{PSY}, we denote by $\Gamma_I(M)$ the maximal
$I$\+torsion submodule in an arbitrary $R$\+module~$M$.
 The functor $\Gamma_I\:R\modl\rarrow R\modl_{I\tors}$ is left 
exact; it is the right adjoint functor to the embedding functor
$R\modl_{I\tors}\rarrow R\modl$.

 Assume that the ideal $I$ is finitely generated.
 Then an $R$\+module $M$ is $I$\+torsion if and only if for every
element $x\in M$ there exists $n\ge1$ such that $I^nx=0$ in~$M$.

 Let $s_1$,~\dots,~$s_m\in R$ be a finite sequence of elements in $R$
generating the ideal~$I$.
 For simplicity of notation, we will denote the sequence
$s_1$,~\dots,~$s_m$ by a single letter~$\s$.
 For any $R$\+module $M$, consider the following \v Cech complex
$C_\s^\bu(M)$
$$\textstyle
 \bigoplus_{j=1}^m M[s_j^{-1}]\lrarrow\bigoplus_{j'<j''}
 M[s_{j'}^{-1},s_{j''}^{-1}]\lrarrow\dotsb\lrarrow
 M[s_1^{-1},\dotsc,s_m^{-1}].
$$
 We place the first term $\bigoplus_j M[s_j^{-1}]$ of the complex
$C_\s^\bu(M)$ in the cohomological degree~$0$, so that there is
a natural morphism of complexes (coaugmentation or unit)
$k\:M\rarrow C_\s^\bu(M)$, and denote by $C_\s^\bu(M)\sptilde$
the cocone of this morphism,
$$\textstyle
 M\lrarrow\bigoplus_{j=1}^m M[s_j^{-1}]\lrarrow\bigoplus_{j'<j''}
 M[s_{j'}^{-1},s_{j''}^{-1}]\lrarrow\dotsb\lrarrow
 M[s_1^{-1},\dotsc,s_m^{-1}].
$$
 So the complex $C_\s^\bu(M)\sptilde$ also has its first term~$M$
at the cohomological degree~$0$.
 Obviously, there are natural isomorphisms of complexes
$C_\s^\bu(M)\simeq C_\s^\bu(R)\ot_R M$ and
$C_\s^\bu(M)\sptilde\simeq C_\s^\bu(R)\sptilde\ot_R M$.
 Furthermore, the complex $C_\s^\bu(R)\sptilde$ is isomorphic to
the tensor product of the similar complexes corresponding to
one-element sequences $\{s_j\}$ over all the elements of
the sequence $s_1$,~\dots,~$s_m$
$$
 C_\s^\bu(R)\sptilde\simeq C_{\{s_1\}}^\bu(R)\sptilde\ot_R\dotsb
 \ot_R C_{\{s_m\}}^\bu(R)\sptilde.
$$

\begin{lem} \label{cech-complex}
\textup{(a)} For any $R$\+module $M$, all the cohomology modules
$H^*C_\s^\bu(M)\sptilde$ of the complex $C_\s^\bu(M)\sptilde$ are
$I$\+torsion $R$\+modules. \par
\textup{(b)} For any $R$\+module $M$, the natural morphism of
complexes $C_\s^\bu(M)\sptilde\rarrow M$ induces an isomorphism
of $R$\+modules
$$
 H^0C_\s^\bu(M)\sptilde\simeq\Gamma_I(M).
$$ \par
\textup{(c)} For any $I$\+torsion $R$\+module $M$, the morphism
of complexes $C_\s^\bu(M)\sptilde \rarrow M$ is an isomorphism.
\end{lem}

\begin{proof}
 To prove part~(a), notice that the complex
$C_\s^\bu(M)\sptilde[s_j^{-1}]$ is contractible for every~$j$,
because $C_{\{s_j\}}^\bu(R)\sptilde[s_j^{-1}]$ is a contractible
two-term complex $(R[s_j^{-1}]\to R[s_j^{-1}])$.
 To obtain part~(b), recall that for any given element $s\in R$
the kernel of the map $M\rarrow M[s^{-1}]$ consists precisely of
all the elements $x\in M$ for which there exists an integer $n\ge1$
such that $s^nx=0$.
 Part~(c) is obvious, as the complex $C_\s^\bu(M)$ vanishes entirely
for any $I$\+torsion $R$\+module~$M$.
\end{proof}

 Denote by $X$ the affine scheme $\Spec R$, by $Z=\Spec R/I\subset X$
the corresponding closed subscheme in $X$, by $U=X\setminus Z$ its
open complement, and by $U_j$ the principal affine open subschemes
$\Spec R[s_j^{-1}]\subset\Spec R$.
 The open subschemes $U_j\subset X$ form an affine open covering of
the open subscheme $U\subset X$.
 Let $\M$ denote the quasi-coherent sheaf on $X$ corresponding to
the $R$\+module~$M$; then the \v Cech complex $C_\s^\bu(M)$ computes
the sheaf cohomology $H^*(U,\M|_U)$.

 In particular, when $R$ is a Noetherian ring and $J$ is an injective
$R$\+module, the related quasi-coherent sheaf $\J$ on $X$ is flasque
and $\J|_U$ is an injective quasi-coherent sheaf on~$U$
\cite[Lemma~II.7.16 and Theorem~II.7.18]{Har}.
 Hence it follows that the complex $C_\s^\bu(J)$ is quasi-isomorphic
to its zero cohomology module $H^0C_\s^\bu(J)\simeq\J(U)$ and
the complex $C_\s^\bu(J)\sptilde$ is quasi-isomorphic to the two-term
complex $\J(X)\rarrow\J(U)$, which also has its only cohomology group
in degree~$0$.
 We have shown that $H^nC_\s^\bu(J)\sptilde=0$ for $n>0$.
 A direct algebraic proof of this assertion (based on the same
Matlis' classification of injective modules over Noetherian
rings~\cite{Mat0} as Hartshorne's argument in~\cite{Har}) can be
found in~\cite[Theorem~4.34]{PSY}.

 According to one of the equivalent definitions, a finite sequence
of elements~$\s$ in a commutative ring $R$ is said to be
\emph{weakly proregular} if the cohomology groups of the complex
$C_\s^\bu(J)\sptilde$ vanish in the positive cohomological degrees
for all injective $R$\+modules~$J$ (see~\cite[Theorem~3.2(ii)]{Sch}
or~\cite[Theorem~4.24]{PSY}).
 According to~\cite[Corollary~6.2]{PSY}
(see also~\cite[Corollary~3.3]{Sch}), the weak proregularity
property of a finite sequence~$\s$ in a commutative ring $R$ only
depends on the ideal $I$ that this sequence generates, and in fact,
even only on the ideal $\sqrt{I}\subset R$.
 We have seen that any ideal $I$ in a Noetherian ring $R$ is
weakly proregular.
 Another equivalent definition of weak proregularity will appear in
Section~\ref{derived-contramodules}.

\begin{lem}  \label{gamma-finite-homol-dim}
\textup{(a)} For any weakly proregular finite sequence of
elements\/~$\s$ in a commutative ring~$R$, the complex
$C_\s^\bu(M)\sptilde$ assigned to an $R$\+module $M$ computes the right
derived functor\/ $\boR^*\Gamma_I(M)$ of the left exact functor\/
$\Gamma_I$, viewed as taking values in the category $R\modl$. \par
\textup{(b)} For any weakly proregular finitely generated ideal $I$ in
a commutative ring $R$, (the right derived functor\/ $\boR^*\Gamma_I$
of) the left exact functor\/ $\Gamma_I\:R\modl\rarrow R\modl_{I\tors}$
has finite homological dimension not exceeding the minimal number
of generators of the ideal~$I$.
\end{lem}

\begin{proof}
 This is~\cite[Theorem~3.2(iii)]{Sch} or~\cite[Corollaries~4.26
and~4.28]{PSY}.
 Part~(a): given an $R$\+module $M$, consider its right injective
resolution~$J^\bu$.
 Then the total complex of the bicomplex $C_\s^\bu(J^\bu)\sptilde$
is quasi-isomorphic both to the complex $\Gamma_I(J^\bu)$,
since the morphism $\Gamma_I(J^n)\rarrow C_\s^\bu(J^n)\sptilde$
is a quasi-isomorphism for every~$n$, and to the complex
$C_\s^\bu(M)\sptilde$, because the functor $N\longmapsto
C_\s^\bu(N)\sptilde$ takes exact sequences of $R$\+modules $N$
to exact sequences of complexes $C_\s^\bu({-})\sptilde$.
 As the embedding functor $R\modl_{I\tors}\rarrow R\modl$ is exact
and takes nonzero objects to nonzero objects, part~(b)
immediately follows.
\end{proof}

 For the most part of
Sections~\ref{derived-torsion-modules}\+-\ref{derived-contramodules},
we will consider simultaneously and almost on equal footing a number
of conventional and exotic derived category constructions introduced
in~\cite{Psemi,Pkoszul}, and~\cite[Appendix~A]{Pcosh}.
 The definitions of these derived categories are briefly recalled
in Appendix~\ref{exotic-derived} to this paper.
 A reader unfamiliar with derived categories of the second kind will
not loose much restricting his attention to the conventional
derived categories $\sD^\b$, $\sD^+$, $\sD^-$, and $\sD$, and
forgetting about the rest.
 In Section~\ref{duality-theorem-secn}, we will only work with
the four conventional derived categories; and then return to a list
of symbols including both the conventional and the absolute derived
categories in the final Sections~\ref{dedualizing-artinian-quotient}\+-%
\ref{dedualizing-weakly-proregular} and
Appendix~\ref{derived-finite-homol-dim}.

\begin{thm}  \label{torsion-fully-faithful}
 Let $R$ be a commutative ring and $I\subset R$ be a weakly proregular
finitely generated ideal.
 Then for any symbol\/ $\st=\b$, $+$, $-$, $\varnothing$, $\abs+$,
$\abs-$, $\co$, or\/~$\abs$, the triangulated functor\/
$\sD^\st(R\modl_{I\tors})\rarrow\sD^\st(R\modl)$ induced by the embedding
of abelian categories $R\modl_{I\tors}\rarrow R\modl$ is fully faithful.
\end{thm}

\begin{proof}
 In the case of the derived categories $\sD^\b$ or $\sD^+$ over
a Noetherian ring $R$, it suffices to notice that, according to
the Artin--Rees lemma, the functor $R\modl_{I\tors}\rarrow R\modl$
preserves injectivity of objects (cf.~\cite[Proposition~2.2.2(a)]{Prev}).
 For the other derived categories in our list,
Lemma~\ref{gamma-finite-homol-dim}(b) essentially says that
the question ``has finite homological dimension'' and therefore
``reduces to finite complexes''.
 A formal proof in the general case proceeds as follows.

 Denote by $R\modl_{I\tors\adj}$ the full subcategory of objects adjusted
to $\Gamma_I$ in $R\modl$; in other words, the subcategory
$R\modl_{I\tors\adj}\subset R\modl$ consists of all the $R$\+modules $M$
for which $\boR^n\Gamma_I(M)=0$ for $n>0$, or equivalently,
$H^nC_\s^\bu(M)\sptilde=0$ for $n>0$.
 The full subcategory $R\modl_{I\tors\adj}$ is closed under extensions,
cokernels of injective morphisms, and infinite direct sums
in $R\modl$.
 Hence the subcategory $R\modl_{I\tors\adj}$ inherits an exact category
structure of the abelian category $R\modl$.

 Moreover, by Lemma~\ref{gamma-finite-homol-dim}(b) any $R$\+module has
a finite right resolution of uniformly bounded length by objects of
the subcategory $R\modl_{I\tors\adj}$.
  It follows that for any symbol~$\st$ on our list the triangulated
functor $\sD^\st(R\modl_{I\tors\adj}) \rarrow\sD^\st(R\modl)$ induced by
the exact embedding $R\modl_{I\tors\adj}\rarrow R\modl$ is
an equivalence of triangulated
categories~\cite[Proposition~A.5.6]{Pcosh},
$$
 \sD^\st(R\modl_{I\tors\adj})\simeq\sD^\st(R\modl).
$$

 Obviously, the restriction of functor $\Gamma_I$ to the full exact
subcategory $R\modl_{I\tors\adj}\allowbreak\subset R\modl$ is an exact
functor $\Gamma_I\:R\modl_{I\tors\adj}\rarrow R\modl_{I\tors}$.
 Applying the functor $\Gamma_I$ to complexes of modules from
the category $R\modl_{I\tors\adj}$ termwise, we obtain the right
derived functor
$$
 \boR\Gamma_I\:\sD^\st(R\modl)\rarrow\sD^\st(R\modl_{I\tors}).
$$
 As the passage to derived functors in the sense of
Deligne~\cite[n$^{\mathrm{os}}$\,1.2.1\+-2]{Del} generally preserves
adjoint pairs of functors, the derived functor $\boR\Gamma_I$ is
right adjoint to the triangulated functor $\sD^\st(R\modl_{I\tors})
\rarrow\sD^\st(R\modl)$ induced by the embedding of abelian
categories $R\modl_{I\tors}\rarrow R\modl$
(see also~\cite[Lemma~8.3]{Psemi}).

 Furthermore, by Lemma~\ref{cech-complex}(c) the abelian subcategory
$R\modl_{I\tors}\subset R\modl$ is contained in the exact subcategory
$R\modl_{I\tors\adj}$,
$$
 R\modl_{I\tors}\subset R\modl_{I\tors\adj}.
$$
 Therefore, the composition of adjoint functors
$$
 \sD^\st(R\modl_{I\tors})\rarrow\sD^\st(R\modl)\rarrow
 \sD^\st(R\modl_{I\tors})
$$
is isomorphic to the identity functor on $\sD^\st(R\modl_{I\tors})$.
 It follows immediately that the functor $\sD^\st(R\modl_{I\tors})
\rarrow\sD^\st(R\modl)$ is fully faithful, while the functor
$\boR\Gamma_I\:\sD^\st(R\modl)\rarrow\sD^\st(R\modl_{I\tors})$ is
a Verdier quotient functor.
\end{proof}

\begin{cor} \label{complexes-with-torsion-cohomology}
 Let $R$ be a commutative ring and $I\subset R$ be a weakly
proregular finitely generated ideal.
 Then for any symbol\/ $\st=\b$, $+$, $-$, or\/~$\varnothing$,
the triangulated functor\/ $\sD^\st(R\modl_{I\tors})\rarrow
\sD^*(R\modl)$ identifies the derived category\/
$\sD^\st(R\modl_{I\tors})$ with the full subcategory\/
$\sD^\st_{I\tors}(R\modl)\subset\sD^\st(R\modl)$ consisting of
all the complexes with the cohomology modules belonging
to $R\modl_{I\tors}$.
\end{cor}

\begin{proof}
 This is~\cite[Corollary~4.32]{PSY}; see
also~\cite[Proposition~6.12]{DG}.
 It is obvious that the image of the functor $\sD^\st(R\modl_{I\tors})
\rarrow\sD^\st(R\modl)$ is contained in $\sD^\st_{I\tors}(R\modl)$.
 To prove the converse inclusion, one notices that, according to
the proof of Theorem~\ref{torsion-fully-faithful}, the image of
the fully faithful functor $\sD^\st(R\modl_{I\tors})\rarrow
\sD^\st(R\modl)$ consists precisely of all the complexes fixed
by the composition of adjoint functors $\sD^\st(R\modl)\rarrow
\sD^\st(R\modl_{I\tors})\rarrow\sD^\st(R\modl)$.
 Since the functor $\boR\Gamma_I$ has finite homological dimension,
a complex viewed as an object of the conventional derived category is
fixed by this composition whenever its cohomology modules are.
\end{proof}

\Section{Derived Category of Contramodules}
\label{derived-contramodules}

 An abelian group $P$ with an additive operator $s\:P\rarrow P$ is
said to be an \emph{$s$\+contramodule} if for every sequence of
elements $p_0$, $p_1$, $p_2$,~\dots~$\in P$ there is a unique
sequence of elements $q_0$, $q_1$, $q_2$,~\dots~$\in P$ satisfying
the infinite system of nonhomogeneous linear equations
\begin{equation} \label{p-q-s-linear-system}
 q_n=sq_{n+1}+p_n \qquad \text{for all $n\ge0$.}
\end{equation}
 The \emph{infinite summation operation with $s$\+power coefficients}
in an $s$\+contramodule $P$ is defined by the rule
$$
 \sum_{n=0}^\infty s^n p_n=q_0\in P.
$$
 Conversely, given an additive, associative, and unital infinite
summation operation
$$
 (p_n)_{n=0}^\infty\longmapsto\sum_{n=0}^\infty s^np_n
$$
in an abelian group $P$ one can uniquely solve the system of
equations~\eqref{p-q-s-linear-system} by setting
$$
 q_n=\sum_{i=0}^\infty s^ip_{n+i}.
$$
(see~\cite[Section~1.6]{Prev} and~\cite[Section~B.5]{Pweak};
cf.~\cite[Section~C.2]{Pcosh}).

 A module $P$ over a ring $R$ with a chosen element $s\in R$ is
said to be an \emph{$s$\+contra\-module} if it is a contramodule with
respect to the operator of multiplication with~$s$.
 When $s$~is a central element in $R$, this condition is equivalent
to the Ext group vanishing $\Ext_R^i(R[s^{-1}],P)=0$ for $i=0$
and~$1$ (notice that the $R$\+module $R[s^{-1}]$ has projective
dimension at most~$1$, so such Ext groups with $i\ge2$ always vanish)
\cite[Lemma~B.7.1]{Pweak}.

 Let $R$ be a commutative ring, $I$ be an ideal in $R$, and $s_j\in R$
be a set of generators of the ideal~$I$.
 Then the property of an $R$\+module $P$ to be a contramodule with
respect to all the elements~$s_j$ depends only on the ideal $I\subset R$
and not on the choice of a particular set of its generators.
 Indeed, it suffices to check this assertion for rings finitely
generated over the ring of integers~$\Z$, and for Noetherian rings
it follows from~\cite[Theorem~B.1.1]{Pweak}
(see also~\cite[Section~2.2]{Prev}).
 Another proof of this assertion will be obtained as a byproduct of
the arguments below in this section.

 So an $R$\+module $P$ is called an \emph{$I$\+contramodule}, or
an \emph{$I$\+contramodule $R$\+module}, if it is a contramodule
with respect to some set of generators of the ideal~$I$.
 Making use of the definition of contramodules in terms of Ext
vanishing, one easily deduces the assertion that the property of
an $R$\+module to be an $I$\+contramodule only depends on
the radical $\sqrt{I}\subset R$ of the ideal~$I$.

 The full subcategory of $I$\+contramodule $R$\+modules
$R\modl_{I\ctra}\subset R\modl$ is closed under the kernels and
cokernels of morphisms, extensions, and infinite products in $R\modl$.
 Hence $R\modl_{I\ctra}$ is an abelian category with exact functors
of infinite product and its embedding $R\modl_{I\ctra}\rarrow R\modl$
into the category of $R$\+modules is an exact functor preserving
the infinite products.

 Let $M$ and $P$ be two $R$\+modules.
 Then whenever \emph{either} $M$ is $I$\+torsion, \emph{or} $P$ is
an $I$\+contramodule, the $R$\+module $\Hom_R(M,P)$ is
an $I$\+contramodule.
 Indeed, for any element $s\in I$ the infinite summation operations
with $s$\+power coefficients in $\Hom_R(M,P)$ are provided by
the formula
$$
 \Big(\sum_{n=0}^\infty s^nf_n\Big)(x)=\sum_{n=0}^\infty f_n(s^nx),
$$
where the sum in the right-hand side only has a finite number of
nonvanishing summands, in the former case, and
$$
 \Big(\sum_{n=0}^\infty s^nf_n\Big)(x)=\sum_{n=0}^\infty s^nf_n(x),
$$
where the sum in the right-hand side is the $s$\+power infinite
summation operation in the $s$\+contramodule $P$, in the latter one.
 Here $f_n$~are arbitrary elements of the $R$\+module $\Hom_R(M,P)$
and $x$ is an arbitrary element of~$M$.

 Assume that the ideal $I\subset R$ is finitely generated.
 In any $R$\+module that is separated and complete in the $I$\+adic
topology, i.~e., an $R$\+module $P$ for which the natural map
$P\rarrow\varprojlim_n P/I^nP$ is an isomorphism, the infinite
summation operation with $s$\+power coefficients can be defined
for any $s\in I$ as the $I$\+adic limit of finite partial sums.
 So all the $I$\+adically separated and complete $R$\+modules are
$I$\+contramodules.
 Conversely, all the $I$\+contramodule $R$\+modules are $I$\+adically
complete, but they do not have to be $I$\+adically separated
(see~\cite[Example~4.33]{PSY} or~\cite[Section~1.5]{Prev}
and the references therein).
 Still, one has $P/IP\ne0$ for any
$I$\+contramodule $R$\+module $P\ne0$
(see~\cite[Corollary~2.5]{PSY2}, \cite[Lemma~1.3.1]{Pweak},
or~\cite[Lemma~2.1(b)]{Prev}).

 The theory of ``cohomologically $I$\+adically complete complexes
of $R$\+modules'' is developed in~\cite{PSY} using the $I$\+adic
completion functor $\Lambda_I\:M\longmapsto \varprojlim_n M/I^nM$,
which is neither left, nor right exact (and generally seems to be
only well-behaved for flat $R$\+modules~$M$).
 We prefer to construct and use the functor left adjoint to
the embedding functor of the full subcategory of $I$\+contramodule
$R$\+modules.

\begin{prop}  \label{functor-delta}
 Assume that the ideal $I\subset R$ is finitely generated.
 Then the embedding functor $R\modl_{I\ctra}\rarrow R\modl$ has
a left adjoint functor
$$
 \Delta_I\:R\modl\rarrow R\modl_{I\ctra}.
$$
\end{prop}

\begin{proof}
 We will present an explicit construction of the functor~$\Delta_I$.
 Let us first consider the case of an ideal $I$ generated by
a single element $s\in R$.
 To any $R$\+module $M$ we assign the $R$\+module morphism
$$\textstyle
 \phi_M^s\:\prod_{n=1}^\infty M\lrarrow\prod_{n=0}^\infty M
$$
taking every sequence of elements $r_1$, $r_2$,~\dots~$\in M$
to the sequence of elements
$$
 m_0=-sr_1, \ \ m_1=r_1-sr_2, \ \ m_2=r_2-sr_3, \ \dotsc 
$$
 Denote by $\Delta_s(M)$ the cokernel of the morphism~$\phi_M^s$.
 
 Let us show, first of all, that the $R$\+module $\Delta_s(M)$ is
an $s$\+contramodule.
 The assertion only depends on the $\Z[s]$\+module structure
on $M$, so one can assume that $R=\Z[s]$.
 For any abelian group $A$, consider the $\Z[s]$\+module $A[s]$ of
polynomials in~$s$ with the coefficients in~$A$.
 It is straightforward to compute that $\Delta_s(A[s])=A[[s]]$ is
the $\Z[s]$\+module of formal Taylor power series in~$s$ with
the coefficients in~$A$; the isomorphism is provided by the map
$\prod_{n=0}^\infty A[s]\rarrow A[[s]]$ taking any sequence of
polynomials $m_0$, $m_1$, $m_2$,~\dots\ to the power series
$\sum_{n=0}^\infty s^nm_n(s)$.
 Clearly, any $\Z[s]$\+module is the cokernel of a morphism of
modules of the form $A[s]$, the functor $\Delta_s$ preserves
cokernels, and the $\Z[s]$\+modules $A[[s]]$ are $s$\+contramodules.
 Since the class of $s$\+contramodules is closed under cokernels,
the assertion is proven.

 Now let us show that the group of $R$\+module morphisms
$\Hom_R(M,P)$ is naturally isomorphic to $\Hom_R(\Delta_s(M),P)$
whenever an $R$\+module $P$ is an $s$\+contramodule.
 The embedding $M\rarrow\prod_{n=0}^\infty M$ taking every element
$m\in M$ to the sequence $m_0=m$, $m_1=m_2=\dotsb=0$, induces
a natural morphism $M\rarrow\Delta_s(M)$ for any $R$\+module~$M$.
 For any $s$\+contramodule $P$, any $R$\+linear map $f\:M\rarrow P$
is extended to an $R$\+linear map $g\:\Delta_s(M)\rarrow P$ by the rule
$$
 g(m_0,m_1,m_2,\dotsc) = \sum_{n=0}^\infty s^n f(m_n),
$$
where the summation sign in the right-hand side stands for
the infinite summation operation in an $s$\+contramodule~$P$.

 It remains to show that there exists no other $R$\+linear map
$\Delta_s(M)\rarrow P$ whose composition with the natural map
$M\rarrow\Delta_s(M)$ is equal to the given $R$\+linear map
$f\:M\rarrow P$.
 Indeed, let $h\:\Delta_s(M)\rarrow P$ be such a map; we will denote
its composition with the natural surjection $\prod_{n=0}^\infty M
\rarrow\Delta_s(M)$ also by~$h$.  Set
$$
 q_n=h(m_n,m_{n+1},m_{n+2},\dotsc)\in P\qquad \text{for every $n\ge0$.}
$$
 The map~$\phi_M^s$ takes the sequence of elements $(r_1,r_2,\dotsc)
=(m_{n+1},m_{n+2},\dotsc)$ to the sequence
$$
 \phi(m_{n+1},m_{n+2},\dotsc)=
 (0,m_{n+1},m_{n+2},\dotsc)-s(m_{n+1},m_{n+2},\dotsc),
$$
so the elements $q_n$ and $f(m_n)\in P$ satisfy the system of
linear equations
$$
 q_n-f(m_n)=sq_{n+1} \qquad\text{for all $n\ge0$.}
$$
 By the definition of an $s$\+contramodule, the sequence of
elements~$q_n$ is uniquely determined by the sequence~$f(m_n)$.

 We have constructed the functor $\Delta_I=\Delta_s$ for an ideal
$I=(s)$ generated by a single element $s\in R$.
 To produce the functor $\Delta_I$ for an ideal $I=(s_j)_{j=1}^m$
generated by a finite sequence of elements $s_j\in R$, notice that for
any two elements $s$ and $t\in R$ the functor $\Delta_s\:R\modl
\rarrow R\modl_{(s)\ctra}\subset R\modl$ takes the full subcategory
$R\modl_{(t)\ctra}\subset R\modl$ into itself.
 Indeed, the full subcategory $R\modl_{(t)\ctra}$ is closed under
the infinite products and cokernels in $R\modl$.

 Hence it is clear that the restriction of the functor $\Delta_s$
to the full subcategory $R\modl_{(t)\ctra}\subset R\modl$ provides
a functor $R\modl_{(t)\ctra}\rarrow R\modl_{(s,t)\ctra}$ left adjoint
to the embedding of full subcategory $R\modl_{(s,t)\ctra}\rarrow
R\modl_{(t)\ctra}$.
 Composing the functors $\Delta_{s_j}$ over all the elements in
a chosen finite set of generators of the ideal $I\subset R$, we
obtain the desired functor $\Delta_I=\Delta_{s_1}\dotsm\Delta_{s_m}$.
\end{proof}

 In the paper~\cite[Definition~5.1]{PSY}, the notation
$\Tel^\bu(R,s)$ is used for the following two-term complex of
infinitely generated free $R$\+modules sitting in the cohomological
degrees~$0$ and~$1$
$$\textstyle
 \bigoplus_{n=0}^\infty R\delta_n\lrarrow
 \bigoplus_{n=0}^\infty R\delta_n
$$
with the differential $d(\delta_0)=\delta_0$, \,$d(\delta_n)=
\delta_{n-1}-s\delta_n$ for $n\ge 1$.
 Essentially the same complex can be found in~\cite[equation~(6.7)]{DG}.
 Passing to the quotient modules of both the terms of this complex
by the submodules $R\delta_0$ and redenoting $\epsilon_n=\delta_{n+1}$
in the left-hand side, we obtain a homotopy equivalent complex
$T^\bu(R,s)$ of the form
$$\textstyle
 \bigoplus_{n=0}^\infty R\epsilon_n\lrarrow
 \bigoplus_{n=1}^\infty R\delta_n
$$
with the differential $d(\epsilon_0)=-s\delta_1$, \,$d(\epsilon_n)=
\delta_n-s\delta_{n+1}$ for $n\ge 1$.

 The telescope complex $\Tel^\bu(R,s)$ is interesting because
it is quasi-isomorphic to the two-term \v Cech complex
$C^\bu_{\{s\}}(R)\sptilde=(R\to R[s^{-1}])$ from
Section~\ref{derived-torsion-modules}; the quasi-isomorphism
$\Tel^\bu(R,s)\rarrow C^\bu_{\{s\}}(R)\sptilde$ is given by
the rules $\delta_0\longmapsto 1$, \,$\delta_n\longmapsto0$
in degree~$0$ and $\delta_n\longmapsto s^{-n}$ in degree~$1$.
 Setting $\Tel^\bu(R,\s)=\Tel^\bu(R,s_1)\ot_R\dotsb\ot_R\Tel^\bu(R,s_m)$
and $T^\bu(R,\s)=T^\bu(R,s_1)\ot_R\dotsb\ot_R T^\bu(R,s_m)$ for a finite
sequence $s_1$,~\dots, $s_m\in R$, we have a quasi-isomorphism
of finite complexes of flat $R$\+modules
$$
 \Tel^\bu(R,\s)\lrarrow C^\bu_\s(R)\sptilde
$$
and a homotopy equivalence of finite complexes of free $R$\+modules
$$
 \Tel^\bu(R,\s)\lrarrow T^\bu(R,\s).
$$

\begin{lem} \label{telescope-complex}
\textup{(a)} For any $R$\+module $M$, all the homology modules
of the complex\/ $\Hom_R(\Tel^\bu(R,\s),M)$ are $I$\+contramodule
$R$\+modules. \par
\textup{(b)} For any $R$\+module $M$, there is a natural isomorphism
of $R$\+modules
$$
 H_0\Hom_R(\Tel^\bu(R,\s),M)\simeq\Delta_I(M).
$$ \par
\textup{(c)} For any $I$\+contramodule $R$\+module $P$, the morphism
of complexes\/ $\Tel^\bu(R,\s)\rarrow R$ induces a quasi-isomorphism
of complexes of $R$\+modules
$$
 P\lrarrow\Hom_R(\Tel^\bu(R,\s),P).
$$
\end{lem}

\begin{proof}
 Part~(a): for any element $s\in R$, denote by $T^\bu(R,s)'$
the subcomplex of the two-term complex $\Tel^\bu(R,s)$ spanned by all
the generators~$\delta_n$ with the exception of
the generator~$\delta_0$ in degree~$0$.
 The complex $T^\bu(R,s)'$ a free $R$\+module resolution of
the $R$\+module $R[s^{-1}]$.
 Consequently, an $R$\+module $P$ is an $s$\+contramodule if and only
if the two-term complex $\Hom_R(T^\bu(R,s)',P)$ is acyclic, that is
the $R$\+module morphism $\Hom_R(T^1(R,s)',P)\rarrow\Hom_R(T^0(R,s)',P)$
is an isomorphism.
 It follows that a complex of $R$\+modules $P^\bu$ has $s$\+contramodule
cohomology modules if and only if the morphism of complexes
$\Hom_R(T^1(R,s)',P^\bu)\rarrow\Hom_R(T^0(R,s)',P^\bu)$ is
a quasi-isomorphism, that is the complex $\Hom_R(T^\bu(R,s)',P^\bu)$
is acyclic.
 Now for any $R$\+module $M$ and any $1\le j\le m$ one has
$\Hom_R(T^\bu(R,s_j)',\Hom_R(\Tel^\bu(R,\s),M))\simeq
\Hom_R(\Tel^\bu(R,\s)\ot_R T^\bu(R,s_j)'\;M)$, and the complex
$\Tel^\bu(R,\s)\ot_R T^\bu(R,s_j)'$, being a complex of free
$R$\+modules quasi-isomorphic to the contractible complex
$C_\s^\bu(R)\sptilde[s_j^{-1}]$ (see the proof of
Lemma~\ref{cech-complex}(a)), is also contractible.

 Part~(b): the functor $\Delta_s$ from the proof of
Proposition~\ref{functor-delta} is isomorphic to the functor
$M\longmapsto H_0\Hom_R(T^\bu(R,s),M)$ by construction, as
the map~$\phi_M^s$ is precisely the differential in
the two-term complex $\Hom_R(T^\bu(R,s),M)$.
 Consequently, the functor $\Delta_I$ is isomorphic to
the functor $M\longmapsto H_0\Hom_R(T^\bu(R,\s),M)$ when the ideal $I$
is generated by a sequence of elements $s_1$,~\dots, $s_m\in R$.
 The latter functor is isomorphic to the functor
$M\longmapsto H_0\Hom_R(\Tel^\bu(R,\s),M)$, because the complexes
$\Tel^\bu(R,\s)$ and $T^\bu(R,\s)$ are homotopy equivalent.

 Part~(c): all the terms of the complexes on the both sides of
the quasi-isomorphism $\Tel^\bu(R,\s)\rarrow C_\s^\bu(R)\sptilde$ have
the property that all the modules $\Ext_R^n$ from them into any
$I$\+contramodule $R$\+module $P$ in the positive degrees $n>0$
vanish, so the complexes $\Hom_R(\Tel^\bu(R,\s),P)$ and
$\Hom_R(C_\s^\bu(R)\sptilde,P)$ are quasi-isomorphic.
 Furthermore, all the terms of the complex $C_\s^\bu(R)$ in the kernel
of the natural morphism of complexes $C_\s^\bu(R)\sptilde\rarrow R$
have the property that all the modules $\Ext_R^n$ from them into
any $I$\+contramodule $R$\+module $P$ in all the degrees $n\ge0$
vanish, so the complex $\Hom_R(C_\s^\bu(R)\sptilde,P)$ is
quasi-isomorphic to~$P$.
\end{proof}

 In particular, we have obtained another proof of the fact that
the property of an $R$\+module to be a contramodule with respect
to each element~$s_j$ in a finite sequence $s_1$,~\dots, $s_m\in R$
only depends on the radical $\sqrt{I}\subset R$ of the ideal $I$
generated by the elements $s_1$,~\dots,~$s_m$ in~$R$.
 Indeed, according to Lemma~\ref{telescope-complex}(b) together with
the proof of Proposition~\ref{functor-delta}, an $R$\+module $P$ is
a contramodule with respect to $s_1$,~\dots, $s_m$ if and only if it has
the form $H_0\Hom_R(\Tel^\bu(R,\s),M)$ for some $R$\+module~$M$.
 Since the complexes $\Tel^\bu(R,\s)$ for various sequences~$\s$
corresponding to the same ideal $\sqrt{I}$ are homotopy
equivalent~\cite[Theorem~6.1]{PSY}, it follows that the contramodule
property only depends on the ideal~$\sqrt{I}$.
 (Alternatively, one can use Lemma~\ref{telescope-complex}(a) and~(c).)

 A projective system of $R$\+modules $M_1\larrow M_2\larrow M_3\larrow
\dotsb$ indexed by the positive integers is said to be \emph{pro-zero}
if for every $l\ge1$ there exists $n>l$ such that the composition
of maps $M_l\larrow\dotsb\larrow M_n$ vanishes.
 Both the countable projective limit functor and its (first and only)
derived functor vanish on pro-zero projective systems,
$\varprojlim_n M_n=0=\varprojlim^1_nM_n$.

 The two-term complex $\Tel^\bu(R,s)$ is the union of its subcomplexes
of free $R$\+modules $\Tel_n^\bu(R,s)$ generated by the symbols
$\delta_0$,~\dots, $\delta_n$ in the degrees~$0$ and~$1$.
 Accordingly, the complex $\Tel^\bu(R,\s)$ is the union of its
subcomplexes of finitely generated free $R$\+modules $\Tel_n^\bu(R,s_1)
\ot_R\dotsb\ot_R\Tel_n^\bu(R,s_m)$.
 The dual complexes of finitely generated free $R$\+modules
$\Hom_R(\Tel_n^\bu(R,\s),R)$, concentrated in the homological degrees
from~$0$ to~$m$, form a projective system.
 A finite sequence of elements~$\s$ in a commutative ring $R$ is called
\emph{weakly proregular} if the projective system of homology modules
$H_i\Hom_R(\Tel_n^\bu(R,\s),R)$ is pro-zero for every $i>0$
(see~\cite[Definition~2.3]{Sch} or~\cite[Definition~4.21 and
Lemma~5.7]{PSY}.
 The following lemma shows that this definition is equivalent to
the one that we used in Section~\ref{derived-torsion-modules}.

\begin{lem}
 Let $M^1_\bu\larrow M^2_\bu\larrow M^3_\bu\larrow\dotsb$ be
a projective system of complexes of $R$\+modules.
 Then the projective system of homology modules $H_i(M^1_\bu)\larrow
H_i(M^2_\bu)\larrow H_i(M^3_\bu)\larrow\dotsb$ is pro-zero
if and only if for any injective $R$\+module $J$ one has\/
$\varinjlim_n H^i\Hom_R(M^n_\bu,J)=0$.
\end{lem}

\begin{proof}
 This is essentially~\cite[Lemma~2.4]{Sch} or~\cite[Theorem~4.24]{PSY}.
 One has $H^i\Hom_R(M^n_\bu,J)\simeq\Hom_R(H_i(M^n_\bu),J)$,
so the ``only if'' assertion is obvious.
 To prove the ``if'', pick an injective $R$\+module morphism from
the $R$\+module $H_i(M^l_\bu)$ to some injective $R$\+module~$J$.
 This morphism represents an element of $H^i\Hom_R(M^l_\bu,J)$,
which has to die in the inductive limit; so there exists $n>k$
such that the composition $H_i(M^n_\bu)\rarrow H_i(M^l_\bu)
\rarrow J$ is a zero morphism.
 It follows that the morphism $H_i(M^n_\bu)\rarrow H_i(M^l_\bu)$
is zero. \hbadness=1050
\end{proof}

\begin{ex}
 It was already mentioned in Section~\ref{derived-torsion-modules}
that by~\cite[Corollary~6.2]{PSY} the weak proregularity is
a property of the radical $\sqrt{I}\subset R$ of the ideal $I$ that
a finite sequence $s_1$,~\dots, $s_m\in R$ generates and not of
the sequence itself.
 We have also seen that any ideal in a Noetherian commutative ring
is weakly proregular.
 Here is a simple observation providing a class of examples of weakly
proregular ideals in non-Noetherian rings~\cite[Remark~1.14]{Sha}
(cf.~\cite[Example~4.35 and Theorem~6.5]{PSY}).

 Let $R\rarrow T$ be a morphism of commutative rings such that $T$
is a flat module over~$R$.
 Let $\s$~be a finite sequence of elements in the ring~$R$, and
let $I$ be the ideal generated by~$\s$ in~$R$.
 Then the ideal $TI$ in the ring $T$ is weakly proregular whenever
the ideal $I$ in the ring $R$ is.
 In other words, the image of a finite sequence~$\s$ is weakly
proregular in $T$ if a sequence~$\s$ is weakly proregular in~$R$.
 Indeed, one has $H_i\Hom_T(\Tel_n^\bu(T,\s),T)\simeq
T\ot_R H_i\Hom_R(\Tel_n^\bu(R,\s),R)$, so for any given~$i$
the projective system $H_i\Hom_T(\Tel_n^\bu(T,\s),T)$ is pro-zero
whenever the projective system $H_i\Hom_R(\Tel_n^\bu(R,\s),R)$ is.
\end{ex}

\begin{lem} \label{agree-on-flats}
 For any weakly proregular finitely generated ideal $I$ in a commutative
ring $R$, the restriction of the functor $\Delta_I$ to the full
subcategory of flat $R$\+modules $F$ in $R\modl$ is an exact functor
isomorphic to the functor of $I$\+adic completion $\Lambda_I\:
F\longmapsto\varprojlim_n F/I^nF$.
\end{lem}

\begin{proof}
 In view of Lemma~\ref{telescope-complex}(b), this
is~\cite[computations~(2\+-3) in the proof of Theorem~5.21]{PSY}.
 The complex $\Tel^\bu_n(R,\s)$ is homotopy equivalent to the tensor
product over~$j$ of the two-term complexes of free $R$\+modules
$(R\to s_j^{-n}R)$.
 Therefore, one has $H_0\Hom_R(\Tel^\bu_n(R,\s),M)\simeq M/(\s^n)M$
for any $R$\+module $M$, where $(\s^n)$~denotes the ideal generated
by the sequence of elements $s_1^n$,~\dots, $s_m^n$ in~$R$.

 Clearly, $\varprojlim_n M/(\s^nM)\simeq\varprojlim_n M/I^nM$ for
any $R$\+module $M$; it remains to check the isomorphism
$$\textstyle
 H_0\varprojlim_n\Hom_R(\Tel_n^\bu(R,\s),F)\simeq
 \varprojlim_nH_0\Hom_R(\Tel_n^\bu(R,\s),F)
$$
for a flat $R$\+module~$F$.
 Since the terms of the subcomplexes $\Tel_n^\bu(R,\s)$ are direct
summands of the respective terms of the ambient complex
$\Tel^\bu(R,\s)$, the transition maps in the projective system
of complexes $\Hom_R(\Tel_n^\bu(R,\s),M)$ are surjective for
any $R$\+module~$M$.
 Finally, if the projective system $H_i\Hom_R(\Tel_n^\bu(R,\s),R)$
is pro-zero for every $i>0$, then the projective systems
$$
 H_i\Hom_R(\Tel_n^\bu(R,\s),F)\simeq
 H_i\Hom_R(\Tel_n^\bu(R,\s),R)\ot_RF
$$
are also pro-zero for any flat $R$\+module $F$ and every~$i>0$.
 Consequently, $\varprojlim^1_n H_i\Hom_R(\Tel_n^\bu(R,\s),F)=0$
for all~$i\in\Z$.
\end{proof}

\begin{ex}
 It is clear from the above computation that for any ideal $I$
generated by a finite sequence of elements~$\s$ in $R$ and any
$R$\+module $M$ there is a natural surjective morphism of
$R$\+modules $\Delta_I(M)\rarrow\Lambda_I(M)$ with the kernel
isomorphic to $\varprojlim^1_n H_1\Hom_R(\Tel_n^\bu(R,\s),M)$.
 The following example shows that this kernel may be nontrivial
even for the $R$\+module $M=R$ when the weak proregularity is not
assumed.

 Let the sequence $\s$ consist of a single element $s\in R$.
 Then the homology group $H_1\Hom_R(\Tel_n^\bu(R,s),M)$ is isomorphic
to the submodule of elements annihilated by~$s^n$ in~$M$; for any
$n>l$, the map $H_1\Hom_R(\Tel_n^\bu(R,s),M)\rarrow
H_1\Hom_R(\Tel_l^\bu(R,s),M)$ is the multiplication with~$s^{n-l}$.
 Now let $R$ be the commutative algebra over a field~$k$ generated
by an infinite sequence of elements $x_1$, $x_2$, $x_3$,~\dots\
and an additional generator~$s$ with the relations $x_ix_j=0$ and
$s^ix_i=0$ for all $i$, $j\ge1$.
 Then the submodule of elements annihilated by~$s^n$ in $R$ is
generated by the elements $x_1$,~\dots, $x_n$, $sx_{n+1}$,
$s^2x_{n+2}$,~\dots, and the map of multiplication with $s^{n-l}$
never annihilates this submodule entirely for $l\ge1$, so
the projective system $H_1\Hom_R(\Tel_n^\bu(R,s),R)$ is not pro-zero
and the one-element sequence $s\in R$ is not weakly proregular.

 Furthermore, one can compute the $R$\+module $\varprojlim^1_n
H_1\Hom_R(\Tel_n^\bu(R,s),R)$ as being isomorphic to the module
of formal power series $W[[s]]$ in the variable~$s$ with
the coefficient space $W$ equal to the derived projective limit
$\prod_i kx_i/\bigoplus_i kx_i$ of the projective system of
vector spaces $kx_n\oplus kx_{n+1}\oplus\dotsb$ and their obvious
embeddings.
 The generators $x_i\in R$ act by zero in the $R$\+module $W[[s]]$,
while the generator $s\in R$ acts in the power series in~$s$ in
the standard way.
\end{ex}

\begin{lem} \label{delta-finite-homol-dim}
\textup{(a)} For any weakly proregular finite sequence of
elements\/~$\s$ in a commutative ring $R$, the complex\/
$\Hom_R(\Tel^\bu(R,\s),M)$ assigned to an $R$\+module $M$ computes
the left derived functor\/ $\boL_*\Delta_I(M)$ of the right exact
functor\/ $\Delta_I$, viewed as taking values in the category
$R\modl$. \par
\textup{(b)} For any weakly proregular finitely generated ideal~$I$
in a commutative ring~$R$, (the left derived functor\/
$\boL_*\Delta_I$ of) the right exact functor $\Delta_I\:R\modl\rarrow
R\modl_{I\ctra}$ has finite homological dimension not exceeding
the minimal number of generators of the ideal~$I$.
\end{lem}

\begin{proof}
 Part~(a): according to~\cite[Theorem~5.21]{PSY}, one has
$H_i\Hom_R(\Tel^\bu(R,\s),F)=0$ for any flat $R$\+module $F$ and all
$i>0$ (we have essentially explained as much already in the proof of
Lemma~\ref{agree-on-flats}).
 Taking also into account Lemma~\ref{telescope-complex}(b),
the argument proceeds from this point in the (standard) way dual to
the proof of Lemma~\ref{gamma-finite-homol-dim}(a) above.
 As the embedding $R\modl_{I\ctra}\rarrow R\modl$ is an exact functor
taking nonzero objects to nonzero objects, part~(b) follows.
 (Cf.~\cite[Corollaries~5.25 and~5.27]{PSY}.)
\end{proof}

\begin{lem} \label{contramodules-adjusted}
 For any weakly proregular finitely generated ideal~$I$ in a commutative
ring $R$, the derived functor of the functor $\Delta_I$ takes
$I$\+contramodules to themselves, i.~e., there is a natural isomorphism
$\Delta_I(P)\simeq P$ and one has\/ $\boL_n\Delta_I(P)=0$ for all $n>0$
and any $I$\+contramodule $R$\+module~$P$.
\end{lem}

\begin{proof}
 This is Lemma~\ref{telescope-complex}(c) together with
Lemma~\ref{delta-finite-homol-dim}(a).
 Notice that the isomorphism $\Delta_I(P)\simeq P$ follows already
from the fact that the functor $\Delta_I$ is left adjoint to
the fully faithful embedding functor $R\modl_{I\ctra}\rarrow R\modl$.
 It is only the second assertion of the lemma that depends on
the weak proregularity assumption.
\end{proof}

 Now we are in the position to prove the main result of this section.

\begin{thm} \label{contra-fully-faithful}
 Let $R$ be a commutative ring and $I\subset R$ be a weakly proregular
finitely generated ideal.
 Then for any symbol\/ $\st=b$, $+$, $-$, $\varnothing$, $\abs+$,
$\abs-$, $\ctr$, or\/~$\abs$, the triangulated functor\/
$\sD^\st(R\modl_{I\ctra})\rarrow\sD^\st(R\modl)$ induced by
the embedding of abelian categories $R\modl_{I\ctra}\rarrow R\modl$
is fully faithful.
\end{thm}

\begin{proof}
 The argument is similar to the proof of
Theorem~\ref{torsion-fully-faithful}.
 In the case of a Noetherian ring $R$ and the derived categories
$\sD^\b$ or $\sD^-$, the assertion follows from~\cite[Propositions~B.9.1
and~B.10.1]{Pweak} (see also~\cite[Proposition~2.2.2(b)]{Prev}).
 A proof of the general case proceeds as follows.

 Consider the full subcategory $R\modl_{I\ctra\adj}\subset R\modl$
consisting of all the $R$\+modules $M$ for which $\boL_n\Delta_I(M)=0$
for all $n>0$.
 The full subcategory $R\modl_{I\ctra\adj}$ is closed under extensions,
kernels of surjective morphisms, and infinite products in $R\modl$;
and, according to Lemma~\ref{delta-finite-homol-dim}(b),
every $R$\+module has a finite left resolution of uniformly bounded
length by objects from $R\modl_{I\ctra\adj}$.
 Hence the subcategory $R\modl_{I\ctra\adj}$ inherits the exact category
structure of the abelian category $R\modl$, and for any symbol~$\st$
on our list the triangulated functor $\sD^\st(R\modl_{I\ctra\adj})
\rarrow\sD^\st(R\modl)$ is an equivalence of triangulated
categories~\cite[Proposition~A.5.6]{Pcosh},
$$
 \sD^\st(R\modl_{I\ctra\adj})\simeq\sD^\st(R\modl).
$$

 Obviously, the restriction of the functor $\Delta_I$ to the full
exact subcategory $R\modl_{I\ctra\adj}\subset R\modl$ is an exact
functor $\Delta_I\:R\modl_{I\ctra\adj}\rarrow R\modl_{I\ctra}$.
 Applying the functor $\Delta_I$ to complexes of modules from
the category $R\modl_{I\ctra\adj}$ termwise, we obtain the left
derived functor
$$
 \boL\Delta_I\:\sD^\st(R\modl)\rarrow\sD^\st(R\modl_{I\ctra}).
$$
 The derived functor $\boL\Delta_I$ is left adjoint to the triangulated
functor $\sD^\st(R\modl_{I\ctra})\allowbreak \rarrow\sD^\st(R\modl)$
induced by the embedding of abelian categories $R\modl_{I\ctra}\rarrow
R\modl$ \cite[Lemma~8.3]{Psemi}.

 Furthermore, the abelian subcategory $R\modl_{I\ctra}\subset R\modl$
is contained in the exact subcategory $R\modl_{I\ctra\adj}$,
$$
 R\modl_{I\ctra}\subset R\modl_{I\ctra\adj}
$$
by Lemma~\ref{contramodules-adjusted}, and the functor $\Delta_I$ takes
$I$\+contramodule $R$\+modules to themselves.
 Therefore, the composition of adjoint functors
$$
 \sD^\st(R\modl_{I\ctra})\rarrow\sD^\st(R\modl)\rarrow
 \sD^\st(R\modl_{I\ctra})
$$
is isomorphic to the identity functor on $\sD^\st(R\modl_{I\ctra})$.
 It follows immediately that the functor $\sD^\st(R\modl_{I\ctra})
\rarrow\sD^\st(R\modl)$ is fully faithful, while the functor
$\boL\Delta_I\:\sD^\st(R\modl)\rarrow\sD^\st(R\modl_{I\ctra})$
is a Verdier quotient functor.
\end{proof}

\begin{cor} \label{complexes-with-contramodule-cohomology}
 Let $R$ be a commutative ring and $I\subset R$ be a weakly proregular
finitely generated ideal.
 Then for any symbol\/ $\st=\b$, $+$, $-$, or\/~$\varnothing$,
the triangulated functor\/ $\sD^\st(R\modl_{I\ctra})\rarrow
\sD^\st(R\modl)$ identifies the derived category\/
$\sD^\st(R\modl_{I\ctra})$ with the full subcategory\/
$\sD^\st_{I\ctra}(R\modl)\subset\sD^\st(R\modl)$ consisting of all
the complexes with the cohomology modules belonging to $R\modl_{I\ctra}$.
\end{cor}

\begin{proof}
 The proof is similar to that of
Corollary~\ref{complexes-with-torsion-cohomology}.
 It is obvious that the image of the functor $\sD^\st(R\modl_{I\ctra})
\rarrow\sD^\st(R\modl)$ is contained in $\sD^\st_{I\ctra}(R\modl)$.
 To prove the converse inclusion, notice that, according to the proof
of Theorem~\ref{contra-fully-faithful}, the image of the fully
faithful functor $\sD^\st(R\modl_{I\ctra})\rarrow\sD^\st(R\modl)$
consists precisely of all the complexes fixed by the composition of
adjoint functors $\sD^\st(R\modl)\rarrow\sD^\st(R\modl_{I\ctra})
\rarrow\sD^\st(R\modl)$.
 Since the functor $\boL\Delta_I$ has finite homological dimension,
by the way-out functor argument of~\cite[Proposition~I.7.1]{Har}
a complex viewed as an object of the conventional derived category
is fixed by this composition whenever its cohomology modules are.
\end{proof}

\Section{MGM Duality Theorem} \label{duality-theorem-secn}

 Let $R$ be a commutative ring and $\s$~be a finite sequence of
its elements $s_1$,~\dots, $s_m\in R$.
 We recall the constructions of complexes $C_\s^\bu(R)$ and
$C_\s^\bu(R)\sptilde$ from Section~\ref{derived-torsion-modules}.

\begin{lem} \label{tensor-product-contractible}
 The tensor product complex $C_\s^\bu(R)\ot_R C_\s^\bu(R)\sptilde$ is
a contractible complex of $R$\+modules.
\end{lem}

\begin{proof}
 It suffices to show that for every term $C_\s^i(R)$ of the complex
$C_\s^\bu(R)$ the tensor product $C_\s^i(R)\ot_R C_\s^\bu(R)\sptilde$
is a contractible complex, because the complex $C_\s^\bu(R)\ot_R
C_\s^\bu(R)\sptilde$ can be obtained from the complexes
$C_\s^i(R)\ot_R C_\s^\bu(R)\sptilde$ by iterating the operations of
shift and cone in the homotopy category of complexes of $R$\+modules.
 Hence it is enough to check that the complex
$C_\s^\bu(R)\sptilde[s_j^{-1}]$ is contractible for every~$j$.
 This was already done in the proof of
Lemma~\ref{cech-complex}(a).
\end{proof} 

 Following~\cite[Section~8]{PSY}, let us endow the complex $C_\s^\bu(R)$
with the \v Cech/singular cochain multiplication, making it
a (noncommutative) DG\+ring.
 The natural morphism of complexes $k\:R\rarrow C_\s^\bu(R)$ is
a morphism of DG\+rings with the image lying in the center of
$C_\s^\bu(R)$, making $C_\s^\bu(R)$ a DG\+algebra over~$R$.
 The following lemma is our version of~\cite[Lemmas~4.29 and~7.9]{PSY}.

\begin{lem}  \label{tensor-product-equivalences}
\textup{(a)} The two morphisms of complexes
$$
 C_\s^\bu(R)\sptilde\ot_RC_\s^\bu(R)\sptilde\birarrow C_\s^\bu(R)\sptilde
$$
induced by the natural morphism $C_\s^\bu(R)\sptilde\rarrow R$ are
homotopy equivalences of complexes of $R$\+modules. \par
\textup{(b)} The three morphisms of complexes
$$
 C_\s^\bu(R)\birarrow C_\s^\bu(R)\ot_R C_\s^\bu(R)\rarrow C_\s^\bu(R)
$$
provided by the unit and multiplication in the DG\+algebra
$C_\s^\bu(R)$ are homotopy equivalences of complexes of $R$\+modules.
\end{lem}

\begin{proof}
 Follows immediately from Lemma~\ref{tensor-product-contractible}.
\end{proof}

 Our next goal is to rewrite the geometric arguments of
Section~\ref{geometric-subcategories-equivalence} in the algebraic
language.
 A key ingredient is the derived category $\sD^\st(C_\s^\bu(R)\modl)$
of left DG\+modules over the DG\+ring $C_\s^\bu(R)$, which is used
in lieu of the derived category $\sD^\st(U\qcoh)\simeq\sD^\st(U\ctrh)$
of quasi-coherent sheaves/contraherent cosheaves on the open subscheme
$U=X\setminus Z$, where $X=\Spec R$ and $Z=\Spec R/I$.

 The DG\+ring $C_\s^\bu(R)$ being positively cohomologically graded,
derived categories of the second kind of DG\+modules over it may
differ from the conventional derived categories even for bounded
DG\+modules (cf.~\cite[Section~3.4, last paragraph of Section~0.4,
and Example~6.6]{Pkoszul}).
 So we restrict our exposition to the four conventional
derived categories $\sD^\b$, \,$\sD^+$, \,$\sD^-$, and~$\sD$.
 Notice that some, though not all, of the related geometric results
of~\cite[Sections~4.6 and~4.8]{Pcosh} are applicable to the absolute
derived categories $\sD^{\abs+}$, \,$\sD^{\abs-}$, and $\sD^{\abs}$ as well.

 So let $\st$ be one of the conventional derived category symbols $\b$,
$+$, $-$, or~$\varnothing$.
 The derived categories $\sD^\st(C_\s^\bu(R)\modl)$ are defined by
inverting the classes of quasi-isomorphisms in the homotopy categories
of (respectively bounded) left DG\+modules over $C_\s^\bu(R)$.
 Denote by $k_*\:\sD^\st(C_\s^\bu(R)\modl)\rarrow\sD^\st(R\modl)$
the functor of restriction of scalars with respect to the DG\+ring
morphism $k\:R\rarrow C_\s^\bu(R)$.

\begin{prop}  \label{two-adjoints-faithful-functor}
\textup{(a)} The triangulated functor~$k_*$ has a left adjoint functor
$k^*\:\sD^\st(R\modl)\rarrow\sD^\st(C_\s^\bu(R)\modl)$ and a right
adjoint functor\/ $\boR k^!\:\sD^\st(R\modl)\allowbreak\rarrow
\sD^\st(C_\s^\bu(R)\modl)$. \par
\textup{(b)} The compositions $k^*\circ k_*$ and\/ $\boR k^!\circ k_*$
are isomorphic to the identity functors on the category\/
$\sD^\st(C_\s^\bu(R)\modl)$, functor~$k_*$ is fully faithful, and
the functors~$k^*$ and\/~$\boR k^!$ are Verdier quotient functors.
\end{prop}

\begin{proof}
 Part~(a): the functor~$k^*$ is easy to define, as $C_\s^\bu(R)$ is
a finite complex of flat $R$\+modules, so setting $k^*(M^\bu)=C_\s^\bu(R)
\ot_RM^\bu$ suffices for any symbol~$\st$.
 In the case of $\st=\varnothing$, both the adjoint functors
$\boL f^*$ and $\boR f^!$ exist for any morphism of DG\+rings~$f$
\cite[Section~1.7]{Pkoszul}.
 To construct the derived functor $\boR k^!$ for bounded derived
categories $\sD^\st$, notice that $C_\s^\bu(R)$ is a finite complex of
$R$\+modules of projective dimension at most~$1$, so embedding
a complex of $R$\+modules $M^\bu$ into a complex of injective
$R$\+modules $J^\bu$, replacing $M^\bu$ with the cocone $N^\bu$ of
the morphism $J^\bu\rarrow J^\bu/M^\bu$ and setting
$\boR k^!(M^\bu)=\Hom_R(C_\s^\bu(R),N^\bu)$ does the job.
 These constructions of the functors~$k_*$, $k^*$, and $\boR k^!$
are actually applicable to the derived categories with exotic
symbols $\st=\abs+$, $\abs-$, or~$\abs$ just as well.
 The constructions of the functors $k_*$ and~$k^*$ also work for
the coderived categories, while the constructions of the functors
$k_*$ and~$\boR k^!$ are applicable to the contraderived categories.

 To prove part~(b), notice that by
Lemma~\ref{tensor-product-equivalences}(b) the adjunction morphisms
$$
 k_*k^*\birarrow k_*k^*k_*k^*\rarrow k_*k^*
$$
are isomorphisms of functors on the category $\sD^\st(R\modl)$, \
$k_*k^*\simeq k_*k^*k_*k^*$.
 It follows that for any objects $M^\bu$ and $N^\bu\in\sD^\st(R\modl)$
the adjunction morphisms $k^*N^\bu\rarrow k^*k_*k^*N^\bu\rarrow
k^*N^\bu$ induce isomorphisms of Hom modules
$$
 \Hom_\sD(k^*M^\bu,k^*N^\bu)\rarrow\Hom_\sD(k^*M^\bu,k^*k_*k^*N^\bu)
 \rarrow\Hom_\sD(k^*M^\bu,k^*N^\bu)
$$
in the derived category $\sD^\st(C_\s^\bu(R)\modl)$.
 In other words, the adjunction morphisms induce isomorphisms of
the functors represented by the objects $k^*N^\bu$ and $k^*k_*k^*N^\bu$
on the essential image of the functor~$k^*$ (viewed as a full
subcategory in $\sD^\st(C_\s^\bu(R)\modl)$).
 As the objects $k^*N^\bu$ and $k^*k_*k^*N^\bu$ also belong to
this essential image, it follows that the adjunction morphisms
$$
 k^*N^\bu\rarrow k^*k_*k^*N^\bu\rarrow k^*N^\bu
$$
are isomorphisms of functors, $k^*\simeq k^*k_*k^*$.
 Hence the adjunction morphisms
$$
 k^*k_*\rarrow k^*k_*k^*k_*\birarrow k^*k_*
$$
are also isomorphisms.
 Applying the same argument with the roles of the two categories
$\sD^\st(R\modl)$ and $\sD^\st(C_\s^\bu(R)\modl)$ switched, one can see
that the adjunction morphisms
$$
 k_*\rarrow k_*k^*k_*\rarrow k_*
$$
are also isomorphisms of functors, $k_*\simeq k_*k^*k_*$.

 Finally, we observe that the restrictions of scalars are conservative
functors between conventional derived categories, i.~e., a morphism of
DG\+modules over $C_\s^\bu(R)$ is a quasi-isomorphism whenever it is
a quasi-isomorphism of DG\+modules over $R$ (cf.~\cite[Remark~8.4.3
and Section~8.4.4]{Psemi} for a discussion of the conservativity
problem for functors between derived categories of the second kind).
 The functor~$k_*$ being conservative and the natural morphism
$k_*k^*k_*B^\bu\rarrow k_*B^\bu$ being an isomorphism in
$\sD^\st(R\modl)$ for any DG\+module $B^\bu\in\sD^\st(C_\s^\bu(R)\modl)$,
we can conclude that the morphism $k^*k_*B^\bu\rarrow B^\bu$ is
an isomorphism in $\sD^\st(C_\s^\bu(R)\modl)$.

 Hence it follows that the functor~$k_*$ is fully faithful,
the functors~$k^*$ and $\boR k^!$ are Verdier quotient functors,
and the composition $\boR k^!\circ k_*$ is isomorphic to
the identity functor.
 Alternatively, the latter assertion can be proven directly in
the way similar to the above argument.
\end{proof}

 Now we are ready to prove our version of the MGM duality theorem
for commutative rings with finitely generated ideals.
 For any commutative ring $R$ with an ideal $I$ generated by a finite
sequence of elements $s_1$,~\dots, $s_m\in R$, consider the full
subcategory of complexes with $I$\+torsion cohomology modules
$$
 \sD^\st_{I\tors}(R\modl)\subset\sD^\st(R\modl)
$$
and the full subcategory of complexes with $I$\+contramodule
cohomology modules
$$
 \sD^\st_{I\ctra}(R\modl)\subset\sD^\st(R\modl)
$$
in the derived category of $R$\+modules.
 The essential image of the fully faithful functor~$k_*$
from Proposition~\ref{two-adjoints-faithful-functor}
$$
 k_*\sD^\st(C_\s^\bu(R)\modl)\subset\sD^\st(R\modl)
$$
is a third triangulated subcategory in $\sD^\st(R\modl)$ that is
of interest to us.

\begin{thm}  \label{complexes-with-t-c-cohomology-equivalence}
 The subcategory $k_*\sD^\st(C_\s^\bu(R)\modl)$ is the right orthogonal
complement to the subcategory\/ $\sD^\st_{I\tors}(R\modl)$ and
the left orthogonal complement to the subcategory\/
$\sD^\st_{I\ctra}(R\modl)$ in\/ $\sD^\st(R\modl)$.
 The passage to the Verdier quotient category by the triangulated
subcategory $k_*\sD^\st(C_\s^\bu(R)\modl)\subset\sD^\st(R\modl)$
establishes an equivalence between the triangulated categories\/
$\sD^\st_{I\tors}(R\modl)$ and\/ $\sD^\st_{I\ctra}(R\modl)$,
$$
 \sD^\st_{I\tors}(R\modl)\simeq
 \sD^\st(R\modl)/k_*\sD^\st(C_\s^\bu(R)\modl)\simeq
 \sD^\st_{I\ctra}(R\modl).
$$
\end{thm}

 The assertion of the theorem can be equivalently restated as
the existence of two semiorthogonal decompositions of the category
$\sD^\st(R\modl)$, one of them formed by the two full subcategories
$k_*\sD^\st(C_\s^\bu(R)\modl)$ and $\sD^\st_{I\tors}(R\modl)\subset
\sD^\st(R\modl)$, and the other one consisting of the two full
subcategories $\sD^\st_{I\ctra}(R\modl)$ and
$k_*\sD^\st(C_\s^\bu(R)\modl)$ in $\sD^\st(R\modl)$
(see, e.~g., \cite[Section~1.3]{Pkoszul}).

\begin{proof}
 It is clear from Proposition~\ref{two-adjoints-faithful-functor} that
the passage to the quotient category by the image of the functor~$k_*$
establishes an equivalence between the kernels of the functors~$k^*$
and~$\boR k^!$,
$$
 \ker(k^*)\simeq \sD^\st(R\modl)/\im k_*\simeq \ker (\boR k^!).
$$
 In other words, we have two semiorthogonal decompositions, one of them
formed by the two full subcategories $\im(k_*)$ and $\ker(k^*)$, and
the other one by the two full subcategories $\ker(\boR k^!)$ and
$\im(k_*)$, in the derived category $\sD^\st(R\modl)$.
 To prove the theorem, we only have to identify the full subcategory
$\ker(k^*)$ with the full subcategory $\sD^\st_{I\tors}(R\modl)\subset
\sD^\st(R\modl)$ and the full subcategory $\ker(\boR k^!)$ with
the full subcategory $\sD^\st_{I\ctra}(R\modl)\subset\sD^\st(R\modl)$.

 As it always happens with semiorthogonal decompositions,
the adjunction morphisms $\Id\rarrow k_*k^*$ and $k_*\boR k^!\rarrow\Id$
have functorial cones.
 The subcategory $\ker(k^*)\subset\sD^\st(R\modl)$ coincides with
the image of the functor
$$
 \cocone(\Id\to k_*k^*)\:\sD^\st(R\modl)\rarrow\sD^\st(R\modl),
$$
while the subcategory $\ker(\boR k^!)\subset\sD^\st(R\modl)$ coincides
with the image of the functor
$$
 \cone(k_*\boR k^!\to\Id)\:\sD^\st(R\modl)\rarrow\sD^\st(R\modl).
$$
 Just as the functors $k_*k^*$ and $k_*\boR k^!$, both the functors
$\cocone(\Id\to k_*k^*)$ and $\cone(k_*\boR k^!\to\Id)$ are projectors
on their respective images: each of them is naturally isomorphic
to its composition with itself.
 Moreover, there are natural transformations
$$
 \cocone(\Id\to k_*k^*)\rarrow\Id \quad\text{and}\quad
 \Id\rarrow\cone(k_*\boR k^!\to\Id).
$$
 A complex $M^\bu\in\sD^\st(R\modl)$ belongs to the subcategory
$\ker(k^*)$ if and only if the morphism
$\cocone(\Id\to k_*k^*)(M^\bu) \rarrow M^\bu$ is a quasi-isomorphism,
while a complex $P^\bu\in\sD^\st(R\modl)$ belongs to the subcategory
$\ker(\boR k^!)$ if and only if the morphism
$P^\bu\rarrow\cone(k_*\boR k^!\to\Id)(P^\bu)$ is a quasi-isomorphism.

 In view of the constructions of the functors~$k^*$ and $\boR k^!$
in the proof of part~(a) of the proposition, the functor
$\cocone(\Id\to k_*k^*)$ is isomorphic to the functor of tensor
product with the complex $C_\s^\bu(R)\sptilde$,
$$
 C_\s^\bu(R)\sptilde\ot_R{-}\:\sD^\st(R\modl)\rarrow\sD^\st(R\modl),
$$
while the functor $\cone(k_*\boR k^!\to\Id)$ is the functor of
right derived homomorphisms from $C_\s^\bu(R)\sptilde$.
 The latter can be easily computed as the homomorphisms from
the complex of free $R$\+modules $\Tel^\bu(R,\s)$,
$$
 \Hom_R(\Tel^\bu(R,\s),{-})\:\sD^\st(R\modl)\rarrow\sD^\st(R\modl).
$$
 To sum up, a complex $M^\bu\in\sD^\st(R\modl)$ belongs to
the subcategory $\ker(k^*)$ if and only if the natural 
map $C_\s^\bu(R)\sptilde\ot_R M^\bu\rarrow M^\bu$ is a quasi-isomorphism,
while a complex $P^\bu\in\sD^\st(R\modl)$ belongs to the subcategory
$\ker(\boR k^!)$ if and only if the natural map
$P^\bu\rarrow\Hom_R(\Tel^\bu(R,\s),P^\bu)$ is a quasi-isomorphism.

 Finally, we recall that $C_\s^\bu(R)\sptilde$ and $\Tel^\bu(R,\s)$
are finite complexes of $R$\+modules, so Hartshorne's way-out
functor argument of~\cite[Propositions~I.7.1 and~I.7.3]{Har} applies.
 Hence the cohomology modules of the complex
$C_\s^\bu(R)\sptilde\ot_RM^\bu$ are $I$\+torsion $R$\+modules for 
any complex of $R$\+modules $M^\bu$ by Lemma~\ref{cech-complex}(a).
 Similarly, the cohomology modules of the complex
$\Hom_R(\Tel^\bu(R,\s),M^\bu)$ are $I$\+contramodule $R$\+modules
for any complex of $R$\+modules $M^\bu$ by
Lemma~\ref{telescope-complex}(a).
 We have shown that the cohomology modules of any complex
$M^\bu\in\ker(k^*)$ are $I$\+torsion $R$\+modules, and
the cohomology modules of any complex $P^\bu\in\ker(\boR k^!)$
are $I$\+contramodule $R$\+modules.

 Conversely, it follows from the above that a complex of $R$\+modules
belongs to the subcategory $\ker(k^*)$ or $\ker(\boR k^!)$ whenever
its cohomology modules, viewed as one-term complexes, do.
 When $M$ is an $I$\+torsion $R$\+module, the map
$C_\s^\bu(R)\sptilde\ot_RM\rarrow M$ is a quasi-isomorphism by
Lemma~\ref{cech-complex}(c).
 When $P$ is an $I$\+contramodule, the map
$P\rarrow\Hom_R(T^\bu(R,\s),P)$ is a quasi-isomorphism by
Lemma~\ref{telescope-complex}(c).
 Therefore, the full subcategory $\ker(k^*)\subset\sD^\st(R\modl)$
consists precisely of all the complexes with $I$\+torsion cohomology
modules, and the full subcategory $\ker(\boR k^!)\subset
\sD^\st(R\modl)$ consists precisely of all the complexes with
$I$\+contramodule cohomology modules.

 Notice that \emph{no} weak proregularity assumption has been used in
this proof (cf.~\cite[Propositions~6.12 and~6.15]{DG}).
\end{proof}

 When the ideal $I\subset R$ is weakly proregular, the MGM duality
theorem takes the form promised in
the formula~\eqref{derived-categories-equivalent} in
Section~\ref{description-of-results} of the introduction.

\begin{cor} \label{weakly-proregular-duality}
 Let $R$ be a commutative ring and $I\subset R$ be a weakly proregular
finitely generated ideal.
 Then for any symbol\/ $\st=\b$, $+$, $-$, or\/~$\varnothing$,
the (respectively bounded or unbounded) derived categories\/
$\sD^\st(R\modl_{I\tors})$ and\/ $\sD^\st(R\modl_{I\ctra})$ are naturally
equivalent,
$$
 \sD^\st(R\modl_{I\tors})\simeq\sD^\st(R\modl_{I\ctra}).
$$
\end{cor}

\begin{proof}
 Compare the results of
Corollaries~\ref{complexes-with-torsion-cohomology}
and~\ref{complexes-with-contramodule-cohomology},
and Theorem~\ref{complexes-with-t-c-cohomology-equivalence}.
 Another proof of this equivalence of derived categories
(in somewhat greater generality) will be obtained in
Theorem~\ref{dedualizing-torsion-contra-duality} below.
\end{proof}

\Section{Dedualizing Complexes for Ideals with Artinian Quotient Rings}
\label{dedualizing-artinian-quotient}

 Let $R$ be a Noetherian commutative ring and $I\subset R$ be
an ideal such that the quotient ring $R/I$ is Artinian.
(One would lose essentially no generality by assuming that $I$
is a maximal ideal in~$R$.)
 Fix an injective envelope of the $R$\+module $R/I$ and denote it
by~$C$.
 Then $C$ is an injective $I$\+torsion $R$\+module.

 For any $R$\+module $M$ and an integer $n\ge1$, we denote by
${}_nM$ the submodule of all elements annihilated by $I^n$ in~$M$.
 In particular, ${}_nC$ is a finitely generated $R/I^n$\+module
and an injective cogenerator of the abelian category $R/I^n\modl$.
 The following lemma is standard~\cite{Mat0}.

\begin{lem} \label{artinian-torsion-modules}
 Let $M$ be an $I$\+torsion $R$\+module.
 Then the following conditions are equivalent:
\begin{enumerate}
\item the $R/I$\+module ${}_1M$ is finitely generated;
\item for every $n\ge1$, the $R/I^n$\+module\/ ${}_nM$ is
finitely generated;
\item $M$ is a submodule of a finite direct sum of copies
of~$C$;
\item $M$ is an Artinian $R$\+module. \qed
\end{enumerate}
\end{lem}

 Given a derived category symbol $\st=\b$, $+$, $-$, $\varnothing$,
$\abs+$, $\abs-$, or~$\abs$, we denote by $\sD^\st(R\modl_{I\tors})$ and
$\sD^\st(R\modl_{I\ctra})$ the corresponding (conventional or absolute)
derived categories of the abelian categories $R\modl_{I\tors}$ and
$R\modl_{I\ctra}$ of $I$\+torsion and $I$\+contramodule $R$\+modules
(see Appendix~\ref{exotic-derived} for the definitions).

 Given an $R$\+module $M$ and a set $X$, we denote by $M[X]$
the $X$\+indexed direct sum of copies of~$M$.
 The $R$\+module $C[X]$ is called the \emph{cofree $I$\+torsion
$R$\+module cogenerated by the set~$X$} (cf.~\cite[Sections~1.4
and~1.9]{Pweak}).
 The abelian category $R\modl_{I\tors}$ has enough injective objects,
which are precisely the direct summands of cofree $I$\+torsion
$R$\+modules.
 Every injective object of $R\modl_{I\tors}$ is also injective in
$R\modl$.

 Identifying the bounded below derived category $\sD^+(R\modl_{I\tors})$
with the homotopy category of injective $I$\+torsion $R$\+modules
$\Hot^+(R\modl_{I\tors})$ and applying the functor
$M\longmapsto{}_nM=\Hom_R(R/I^n,M)$ to complexes of injective
$I$\+torsion $R$\+modules termwise, we obtain the right derived functor
$$
 M^\bu\longmapsto{}_n^\boR M^\bu\:\,
 \sD^+(R\modl_{I\tors})\lrarrow\sD^+(R/I^n\modl).
$$

 The full subcategory of $I$\+torsion $R$\+modules satisfying
the equivalent conditions of Lemma~\ref{artinian-torsion-modules}
is closed under the passages to subobjects, quotient objects,
and extensions in $R\modl_{I\tors}$.
 So the category of Artinian $I$\+torsion $R$\+modules is abelian.
 Notice that this abelian category has enough injective objects,
which are precisely the direct summands of finitely cogenerated
cofree $I$\+torsion $R$\+modules (i.~e., of the finite direct sums
of copies of~$C$).

\begin{lem}  \label{finitely-cogenerated-complex}
 Fix $n\ge1$.
 Then a complex $L^\bu\in\sD^+(R\modl_{I\tors})$ has Artinian
$R$\+modules of cohomology if and only if the complex\/
${}_n^\boR L^\bu\in\sD^+(R/I^n\modl)$ has finitely generated
$R/I^n$\+modules of cohomology.
\end{lem}

\begin{proof}
 First one applies Lemma~\ref{artinian-torsion-modules} in order
to prove that the complex ${}_n^\boR L$ has finitely generated
cohomology $R/I^n$\+modules for every Artinian $I$\+torsion
$R$\+module $L$ (viewed as a one-term complex).
 In the general case, both the ``if'' and ``only if'' assertions
are then deduced by induction in the cohomological degree.
\end{proof}

\begin{lem}  \label{finitely-cogenerated-cofree-complex}
 Let $L^\bu$ be a bounded below complex of $I$\+torsion $R$\+modules
with Artinian cohomology modules.
 Then there exists a bounded below complex of finitely cogenerated
cofree $I$\+torsion $R$\+modules $J^\bu$ together with
a quasi-isomorphism of complexes of $I$\+torsion $R$\+modules
$L^\bu\rarrow J^\bu$.
\end{lem}

\begin{proof}
 This is a standard step-by-step construction
(cf.~\cite[Lemma~1.2]{Pfp} and the proof
of Lemma~\ref{object-resolutions}(c) below).
\end{proof}

 Let $\fR=\varprojlim_n R/I^n$ denote the $I$\+adic completion
of the ring~$R$.
 So $\fR$ is a finite direct sum of complete Noetherian local rings.
 For any set $X$, we put $\fR[[X]]=\varprojlim_n R/I^n[X]$.
 For every $n\ge1$, the natural map $R/I^n\rarrow
\Hom_{R/I^n}({}_nC,\.{}_nC)$ is an isomorphism.
 Hence one has $\fR\simeq\Hom_R(C,C)$, and moreover, $\fR[[X]]
\simeq\Hom_R(C,C[X])$ for any set~$X$.

 Clearly, $\fR[[X]]$ is an $I$\+contramodule $R$\+module; we call it
the \emph{free $I$\+contramodule $R$\+module generated by the set~$X$}.
 For any $I$\+contramodule $R$\+module $P$, one has
$\Hom_R(\fR[[X]],P)=P^X$.
 There are enough projective objects in $R\modl_{I\ctra}$, and
an $I$\+contramodule $R$\+module is a projective object in
$R\modl_{I\ctra}$ if and only if it is a direct summand of
a free $I$\+contramodule $R$\+module (see
Lemma~\ref{free-contramodules}(a) below and the discussion in
the two paragraphs preceding it).

\begin{lem} \label{noetherian-contramodules}
 Let $P$ be an $I$\+contramodule $R$\+module.
 Then the following conditions are equivalent:
\begin{enumerate}
\item the $R/I$\+module $P/IP$ is finitely generated;
\item for every $n\ge1$, the $R/I^n$\+module $P/I^nP$ is
finitely generated;
\item $P$ is a quotient module of a finite direct sum of copies
of\/~$\fR$;
\item $P$ is a Noetherian object of the category of
$I$\+contramodule $R$\+modules.
\end{enumerate}
\end{lem}

\begin{proof}
 This lemma looks pretty much like the dual result to
Lemma~\ref{artinian-torsion-modules}, but in fact holds in
the greater generality of an arbitrary ideal $I$ in a Noetherian
commutative ring~$R$ (cf.~\cite[Theorem~0.2]{PSY2}).
 The implications (3)~$\Longrightarrow$(2) $\Longrightarrow$~(1)
are obvious, and (4)~$\Longrightarrow$~(3) is easily deduced from
the existence of a surjective $R$\+module morphism $\fR[[X]]\rarrow P$.
 To prove (1)~$\Longrightarrow$~(3), one chooses a finite set of
generators $X$ in the $R/I$\+module $P/IP$, lifts the composition
of morphisms $\fR[[X]]\rarrow R/I[X]\rarrow P/IP$ to
an $R$\+module morphism $f\:\fR[X]=\fR[[X]]\rarrow P$, notices
that the cokernel $L$ of the morphism~$f$ satisfies $L/IL=0$,
and applies the contramodule Nakayama lemma~\cite[Lemma~1.3.1]{Pweak}
together the isomorphism of contramodule categories
from~\cite[Theorem~B.1.1]{Pweak}.
 It also follows from the latter theorem that any $I$\+contramodule
$R$\+module has a natural $\fR$\+module structure, and all morphisms
between $I$\+contramodule $R$\+modules are $\fR$\+linear.
 So the implication (3)~$\Longrightarrow$~(4) holds because $\fR$
is a Noetherian ring.
\end{proof}

 Let us call an $I$\+contramodule $R$\+module \emph{finitely generated}
if it satisfies the equivalent conditions of
Lemma~\ref{noetherian-contramodules}.
 The full subcategory of finitely generated $I$\+contramodule
$R$\+modules is closed under the passages to subobjects, quotient
objects, and extensions in $R\modl_{I\ctra}$ (and under the passages
to the kernels, cokernels, and extensions in $R\modl$).
 The abelian category of finitely generated $I$\+contramodule
$R$\+modules has enough projective objects, which are precisely
the direct summands of finitely generated free $I$\+contramodule
$R$\+modules (i.~e., of the finite direct sums of copies of~$\fR$).

\begin{lem}
 The restriction of the forgetful functor\/ $\fR\modl\rarrow R\modl$
provides an isomorphism between the abelian categories of finitely
generated\/ $\fR$\+modules and finitely generated $I$\+contramodule
$R$\+modules.
\end{lem}

\begin{proof}
 This lemma also holds for any ideal $I$ in a Noetherian ring~$R$.
 First of all, by~\cite[Theorem~B.1.1]{Pweak}, the forgetful functor
is an isomorphism between the categories of $\fR I$\+contramodule
$\fR$\+modules and $I$\+contramodule $R$\+modules.
 Alternatively, notice that for any sets $X$ and $Y$ one has
$\Hom_\fR(\fR[[X]],\fR[[Y]])\simeq\fR[[Y]]^X\simeq
\Hom_R(\fR[[X]],\fR[[Y]])$.
 The full subcategories of projective objects in
$\fR\modl_{\fR I\ctra}$ and $R\modl_{I\ctra}$ being equivalent and
identified by the forgetful functor, it follows that the whole
abelian categories are identified, too.
 Secondly, any finitely generated $\fR$\+module is an $I$\+contramodule,
since $\fR$ is an $I$\+contramodule and the class of $I$\+contramodules
is closed under cokernels and finite direct sums.
\end{proof}

 Now we return to the setting of a Noetherian ring $R$ with an ideal
$I\subset R$ such that the quotient ring $R/I$ is Artinian.
 For any $R$\+module $M$, we denote by $M\spcheck$ the $R$\+module
$\Hom_R(M,C)$.

\begin{prop} \label{finitely-co-presented-hom}
\textup{(a)} The restrictions of the functors $M\longmapsto M\spcheck$
and $P\longmapsto P\spcheck$ are mutually inverse anti-equivalences
between the abelian categories of Artinian $I$\+torsion $R$\+modules
and finitely generated $I$\+contramodule $R$\+modules. \par
\textup{(b)} For any $I$\+torsion $R$\+module $M$ and any Artinian
$I$\+torsion $R$\+module $L$, the functor $M\longmapsto M\spcheck$
induces an isomorphism of the Hom modules
$$
 \Hom_R(M,L)\simeq\Hom_R(L\spcheck,\.M\spcheck).
$$
\end{prop}

\begin{proof}
 This is essentially the result of~\cite[Theorem~4.2]{Mat0}.
 One first notices that the functors $L\longmapsto L\spcheck$ and
$P\longmapsto P\spcheck$ are mutually inverse anti-equivalences
between the categories of injective Artinian $I$\+torsion
$R$\+modules (i.~e., direct summands of finite direct sums of
copies of $C$) and projective finitely generated $I$\+contramodule
$R$\+modules (i.~e., direct summands of finite direct sums of
copies of~$\fR$).
 Representing arbitrary Artinian $I$\+torsion $R$\+modules as
the kernels of morphisms between injectives and arbitrary
finitely generated $I$\+contramodule $R$\+modules as the cokernels
of morphisms between projectives, one can then deduce both
the parts~(a) and~(b).
\end{proof}

 A finite complex of $R$\+modules $N^\bu$ is said to have
\emph{$R$\+flat dimension~$\le d$} if one has
$H^{-n}(N^\bu\ot^\boL_R M)=0$ for all $R$\+modules $M$ and all
the integers $n>d$.
 A finite complex of $R$\+modules $N^\bu$ is said to have
\emph{$R$\+projective dimension~$\le d$} if one has
$H^n\boR\Hom_R(N^\bu,M)=0$ for all $R$\+modules $M$ and all $n>d$.

 A finite complex of $I$\+torsion $R$\+modules $L^\bu$ is said
to have \emph{$(R,I)$\+contraflat dimension~$\le d$} if one has
$H^{-n}(N^\bu\ot^\boL_R P)=0$ for all $I$\+contramodule $R$\+modules $P$
and all the integers $n>d$.
 A finite complex of $I$\+torsion $R$\+modules $L^\bu$ is said to
have \emph{$(R,I)$\+projective dimension\/~$\le d$} if one has
$\Hom_{\sD^\b(R\modl_{I\tors})}(L^\bu,M[n])=0$ for all $I$\+torsion
$R$\+modules $M$ and all $n>d$.

 We denote the $R$\+flat dimension of a finite complex of $R$\+modules
$N^\bu$ by $\fd_RN^\bu$, the $R$\+projective dimension of $N^\bu$ by
$\pd_RN^\bu$, the $(R,I)$\+contraflat dimension of a finite complex of
$I$\+torsion $R$\+modules $L^\bu$ by $\cfd_{(R,I)}L^\bu$, and
the $(R,I)$\+projective dimension of $L^\bu$ by $\pd_{(R,I)}L^\bu$.

\begin{prop} \label{contraflat-projective-dimension}
 For any finite complex of $I$\+torsion $R$\+modules $L^\bu$, one
has\/ $\pd_RL^\bu\ge \pd_{(R,I)}L^\bu\ge\cfd_{(R,I)}L^\bu
=\fd_RL^\bu=\fd_\fR L^\bu$.
\end{prop}

\begin{proof}
 By the definition, one has $\cfd_{(R,I)}L^\bu\le\fd_RL^\bu$.
 Since $\fR$ is a flat $R$\+module and all flat $\fR$\+modules
are also flat $R$\+modules, for any finite complex of
$\fR$\+modules $N^\bu$ one has $\fd_R N^\bu\le\fd_\fR N^\bu$.
 Furthermore, a finite complex of $\fR$\+modules $N^\bu$ has
$\fR$\+flat dimension~$\le d$ whenever $H^{-n}(N^\bu\ot^\boL_\fR M)=0$
for all finitely generated $\fR$\+modules $M$ and all $n>d$.
 For a finite complex of $I$\+torsion $R$\+modules $L^\bu$ and
an $\fR$\+module $M$ one has $L^\bu\ot_\fR M = L^\bu\ot_R M$,
hence $L^\bu\ot_\fR^\boL M = L^\bu\ot_R^\boL M$.
 Since finitely generated $\fR$\+modules are $I$\+contramodules,
it follows that $\fd_\fR L^\bu\le\cfd_{(R,I)}L^\bu$.
 We have shown that the three (contra)flat dimensions are equal
to each other.

 It remains to check the inequalities involving the projective
dimensions.
 Applying Theorem~\ref{torsion-fully-faithful} in the (easy)
case $\st=\b$, one can see that $\Hom_{\sD^\b(R\modl_{I\tors})}(L^\bu,M[n])=
\Hom_{\sD^\b(R\modl)}(L^\bu,M[n])$ for any finite complex of
$I$\+torsion $R$\+modules $L^\bu$ and any $I$\+torsion $R$\+module~$M$.
 Hence $\pd_RL^\bu\ge\pd_{(R,I)}L^\bu$.
 Finally, for any $R$\+module $M$ there is an isomorphism
$\Hom_R(L^\bu\ot^\boL_R M\;C)\simeq\boR\Hom_R(L^\bu\;\Hom_R(M,C))$,
and the $R$\+module $\Hom_R(M,C)$ is $I$\+torsion whenever
an $R$\+module $M$ is finitely generated.
 Besides, the $R$\+modules $H^{-n}(L^\bu\ot^\boL_RM)$ are $I$\+torsion,
so $H^n\Hom_R(L^\bu\ot^\boL_RM\;C)=0$ implies
$H^{-n}(L^\bu\ot^\boL_RM)=0$.
 This proves that $\pd_{(R,I)}L^\bu\ge\fd_R L^\bu$.
\end{proof}

 Now we come to the main definition of this section.
 As above, we assume that $R$ is a Noetherian commutative ring
and $I\subset R$ is an ideal such that the quotient ring $R/I$
is Artinian.
 In this setting, a finite complex of $I$\+torsion $R$\+modules $B^\bu$
is called a \emph{dedualizing complex} for the ideal $I$ in the ring
$R$ if
\begin{enumerate}
\renewcommand{\theenumi}{\roman{enumi}}
\item the complex $B^\bu$ has finite projective dimension as
a complex of $I$\+torsion $R$\+modules, that is
$\pd_{(R,I)}B^\bu<\infty$;
\item the homothety map $\fR\rarrow\Hom_{\sD^\b(R\modl_{I\tors})}
(B^\bu,B^\bu[*])$ provided by the action of $\fR$ in $B^\bu$
is an isomorphism of graded rings;
\item the cohomology modules of the complex $B^\bu$ are Artinian
$I$\+torsion $R$\+modules.
\end{enumerate}

 The dedualizing complex $B^\bu$ is viewed as an object of the derived
category $\sD^\b(R\modl_{I\tors})$.
 We refer to the book~\cite[Chapter~5]{Har}, and additionally to
the paper~\cite{Pfp} and the references therein, for a discussion of
the classical notions of a dualizing complex over a commutative ring
or a pair of noncommutative ones, after which the above definition
is largely modelled.

 The following example provides a natural choice of a dedualizing
complex for any ideal $I$ in a Noetherian commutative ring $R$ with
an Artinian quotient ring $R/I$.

\begin{ex} \label{dedualizing-artinian-quotient-example}
 Consider the derived functor of maximal $I$\+torsion submodule
$\boR\Gamma_I\:\sD^+(R\modl)\allowbreak\rarrow\sD^+(R\modl_{I\tors})$.
 According to Lemma~\ref{gamma-finite-homol-dim}(b), this functor
takes $\sD^\b(R\modl)$ into $\sD^\b(R\modl_{I\tors})$.
 Set $B^\bu=\boR\Gamma_I(R)$.
 We claim that the complex $B^\bu$ is a dedualizing complex for
the ideal $I\subset R$.

 According to Proposition~\ref{contraflat-projective-dimension}, in
order to check the projective dimension condition~(i) it suffices to
show that $\pd_R B^\bu<\infty$.
 The latter is clear since, by Lemma~\ref{gamma-finite-homol-dim}(a),
the complex $B^\bu$ is isomorphic to the telescope complex
$\Tel^\bu(R,\s)$ in the derived category $\sD^\b(R\modl)$ for any
finite generating sequence~$\s$ of the ideal $I\subset R$.
 In view of Lemma~\ref{agree-on-flats}, the property~(ii) is implicit
in Theorem~\ref{complexes-with-t-c-cohomology-equivalence} in the case
$\st=\b$ (see Example~\ref{dedualizing-torsion-example} below for
further details).

 Let us now prove the condition~(iii).
 According to Lemma~\ref{finitely-cogenerated-complex}, it suffices
to show that the complex ${}_n^\boR B^\bu$ has finitely generated
$R/I^n$\+modules of cohomology.
 The functor $\Gamma_I$ takes injective objects in $R\modl$ to
injective objects in $R\modl_{I\tors}$, so the composition of
derived functors ${}_n^\boR\boR\Gamma_I$ is the derived functor of
the composition of left exact functors
$M\longmapsto {}_n\Gamma_I(M)$.

 The latter functor assigns to an $R$\+module $M$ its maximal
$R/I^n$\+submodule~${}_nM$.
 It remains to notice that the functor $\Ext_R^*(R/I^n,M)$ takes
finitely generated $R$\+modules to cohomologically graded
$R$\+modules with finitely generated components, because it can
be computed in terms of a left resolution of the $R$\+module
$R/I^n$ by finitely generated projective $R$\+modules.
\end{ex}

 The following theorem is the main result of this section.

\begin{thm} \label{dedualizing-artinian-quotient-duality}
 Given a dedualizing complex\/ $B^\bu$ for an ideal $I$ in a Noetherian
commutative ring $R$ with an Artinian quotient ring $R/I$,
for any symbol\/ $\st=\b$, $+$, $-$, $\varnothing$, $\abs+$, $\abs-$,
or\/~$\abs$ there is an equivalence of derived categories\/
$\sD^\st(R\modl_{I\tors})\simeq\sD^\st(R\modl_{I\ctra})$ provided
by mutually inverse functors\/ $\boR\Hom_R(B^\bu,{-})$ and\/
$B^\bu\ot_R^\boL{-}$.
\end{thm}

\begin{proof}
 Assume for simplicity of notation that the complex $B^\bu$ is
concentrated in nonpositive cohomological degrees.
 Let $d$~be an integer greater or equal to the projective dimension
$\pd_{(R,I)}B^\bu$.
 By Proposition~\ref{contraflat-projective-dimension},
we then also have $d\ge\cfd_{(R,I)}B^\bu$.

 To construct the image of a complex of $I$\+torsion $R$\+modules
$M^\bu$ under the functor $\boR\Hom_R(B^\bu,{-})$, one has to choose
an exact sequence of complexes of $I$\+torsion $R$\+modules
$0\rarrow M^{\bu}\rarrow J^{0,\bu}\rarrow J^{1,\bu}\rarrow\dotsb$
with injective $I$\+torsion $R$\+modules~$J^{j,i}$.
 Then one applies the functor $\Hom_R(B^\bu,{-})$ to every complex
$0\rarrow J^{0,i}\rarrow J^{1,i}\rarrow J^{2,i}\rarrow\dotsb$,
obtaining a nonnegatively graded complex of $I$\+contramodule
$R$\+modules $0\rarrow P^{0,i}\rarrow P^{1,i}\rarrow P^{2,i}\rarrow
\dotsb$.
 According to the projective dimension condition on the dedualizing
complex~$B^\bu$, the complex $P^{\bu,i}$ has zero cohomology modules
at the cohomological degrees above~$d$; so it is quasi-isomorphic to
its canonical truncation complex $\tau_{\le d} P^{\bu,i}$.
 By the definition, one sets the object $\boR\Hom_R(B^\bu,M^\bu)$
in the derived category $\sD^\st(R\modl_{I\ctra})$ to be represented
by the total complex of the bicomplex $\tau_{\le d}P^{\bu,\bu}$
concentrated in the cohomological degrees $0\le j\le d$ and $i\in\Z$.

 Similarly, to construct the image of a complex of $I$\+contramodule
$R$\+modules $P^\bu$ under the functor $B^\bu\ot_R^\boL{-}$,
one has to choose an exact sequence of complexes of $I$\+contramodule
$R$\+modules $\dotsb\rarrow F^{-1,\bu}\rarrow F^{0,\bu}\rarrow
P^\bu\rarrow0$ with projective $I$\+contramodule
$R$\+modules~$F^{j,i}$.
 Then one applies the functor $B^\bu\ot_R{-}$ to every complex
$\dotsb\rarrow F^{-2,i}\rarrow F^{-1,i}\rarrow F^{0,i}\rarrow0$,
obtaining a nonpositively graded complex of $I$\+torsion $R$\+modules
$\dotsb\rarrow M^{-2,i}\rarrow M^{-1,i}\rarrow M^{0,i}\rarrow0$.
 By~\cite[Theorem~3.4]{Yek1} or~\cite[Proposition~B.9.1]{Pweak},
the projective $I$\+contramodule $R$\+modules are flat $R$\+modules.
 According to the contraflat dimension condition on the complex
$B^\bu$, the complex $M^{\bu,i}$ has zero cohomology modules at
the cohomological degrees below~$-d$; so it is quasi-isomorphic to
its canonical truncation complex $\tau_{\ge -d} M^{\bu,i}$.
 One sets the object $B^\bu\ot_R^\boL P^\bu$ in the derived category
$\sD^\st(R\modl_{I\tors})$ to be represented by the total complex of
the bicomplex $\tau_{\ge-d}(M^{\bu,\bu})$ concentrated in
the cohomological degrees $-d\le j\le0$ and $i\in\Z$.

 These constructions of two derived functors are but particular cases
of the construction of a derived functor of finite homological dimension
spelled out in Appendix~\ref{derived-finite-homol-dim}.
 According to the results of that appendix, the above constructions
produce well-defined triangulated functors $\boR\Hom_R(B^\bu,{-})\:
\sD^\st(R\modl_{I\tors})\rarrow\sD^\st(R\modl_{I\ctra})$ and
$B^\bu\ot_R^\boL{-}\:\sD^\st(R\modl_{I\ctra})\rarrow
\sD^\st(R\modl_{I\tors})$ for any derived category symbol
$\st=\b$, $+$, $-$, $\varnothing$, $\abs+$, $\abs-$, or~$\abs$.
 Moreover, the former functor is right adjoint to the latter one.
 All these assertions only depend on the first condition~(i) in
the definition of a dedualizing complex.

 It remains to prove that the adjunction morphisms are isomorphisms.
 Since the total complexes of finite acyclic complexes of complexes
are absolutely acyclic, in order to check that the morphism
$P^\bu\rarrow\boR\Hom_R(B^\bu\;B^\bu\ot_R^\boL P^\bu)$ is
an isomorphism in the derived category $\sD^\st(R\modl_{I\ctra})$ for
all the $\st$\+bounded complexes of $I$\+contramodule $R$\+modules
$P^\bu$ it suffices to consider the case of a one-term complex
$P^\bu=P$ corresponding to a single $I$\+contramodule $R$\+module~$P$.
 Furthermore, since a morphism in $\sD^\b(R\modl_{I\ctra})$ is
an isomorphism whenever it is an isomorphism in
$\sD^-(R\modl_{I\ctra})$, one can view the one-term complex $P$
as an object of the bounded above derived category
$\sD^-(R\modl_{I\ctra})$ and replace it with a free $I$\+contramodule
$R$\+module resolution $F^\bu$ of the contramodule~$P$.
 Applying the same totalization argument to the complex~$F^\bu$,
the question reduces to proving that the adjunction morphism
$F\rarrow\boR\Hom_R(B^\bu\;B^\bu\ot_R^\boL F)$ is
an isomorphism in $\sD^\b(R\modl_{I\ctra})$ for any free
$I$\+contramodule $R$\+module~$F$.

 So let $X$ be a set and $F=\fR[[X]]$ be the $I$\+contramodule
$R$\+module generated by~$X$; then one has $B^\bu\ot_R^\boL F
= B^\bu\ot_RF= B^\bu[X]$.
 According to the condition~(iii) together with
Lemma~\ref{finitely-cogenerated-cofree-complex}, there exists
a bounded below complex of finitely cogenerated cofree $I$\+torsion
$R$\+modules $E^\bu$ together with a quasi-isomorphism of complexes
of $I$\+torsion $R$\+modules $B^\bu\rarrow E^\bu$.
 We have to check that the natural map
$$
 \fR[[X]]\lrarrow\Hom_R(B^\bu,E^\bu[X])
$$
is a quasi-isomorphism of complexes of $I$\+contramodule $R$\+modules.

 Consider the morphism of complexes of $I$\+torsion $R$\+modules
\begin{equation} \label{artinian-quotient-main-comparison-map}
 B^\bu\ot_R\Hom_R(E^\bu,C)\lrarrow C
\end{equation}
induced by the morphism $B^\bu\rarrow E^\bu$.
 Applying the functor $M\longmapsto M\spcheck=\Hom_R(M,C)$
to the morphism~\eqref{artinian-quotient-main-comparison-map}
and taking into account Proposition~\ref{finitely-co-presented-hom}(b),
we get the map
$$
 \fR=\Hom_R(C,C)\lrarrow\Hom_R(E^\bu{}\spcheck,\.B^\bu{}\spcheck)
 = \Hom_R(B^\bu,E^\bu),
$$
which is a quasi-isomorphism of complexes of ($I$\+contramodule)
$R$\+modules by the condition~(ii).
 It follows that
the morphism~\eqref{artinian-quotient-main-comparison-map}
is a quasi-isomorphism of complexes of $I$\+torsion $R$\+modules.
 
 Now, applying the functor $M\longmapsto\Hom_R(M,C[X])$ to
the quasi-isomorphism~\eqref{artinian-quotient-main-comparison-map}
and noticing that $\Hom_R(E^\bu,C)=E^\bu{}\spcheck$ is a complex of
finitely generated (free) $\fR$\+modules and
$\Hom_R(E^\bu{}\spcheck,C)=E^\bu$ by
Proposition~\ref{finitely-co-presented-hom}(a),
we obtain the desired quasi-isomorphism of complexes of
$I$\+contramodule $R$\+modules
\begin{multline*}
 \fR[[X]]=\Hom_R(C,C[X])\lrarrow
 \Hom_R(B^\bu,\Hom_R(E^\bu{}\spcheck,\.C[X]))
 \\ =\Hom_R(B^\bu,\.\Hom_R(E^\bu{}\spcheck,C)[X])
 = \Hom_R(B^\bu,E^\bu[X]).
\end{multline*}

 Similarly, in order to prove that the adjunction morphism
$B^\bu\ot_R^\boL\boR\Hom_R(B^\bu,M^\bu)\allowbreak\rarrow M^\bu$
is an isomorphism in the derived category $\sD^\st(R\modl_{I\tors})$
for any $\st$\+bounded complex of $I$\+torsion $R$\+modules $M^\bu$,
it suffices to check that this morphism is an isomorphism in
$\sD^\b(R\modl_{I\tors})$ for any cofree $I$\+torsion $R$\+module
viewed as a one-term complex in $\sD^\b(R\modl_{I\tors})$.
 Let $J=C[X]$ be a cofree $I$\+torsion $R$\+module; then one has
$\boR\Hom_R(B^\bu,J)=\Hom_R(B^\bu,J)=\Hom_R(B^\bu,C[X])$.

 As above, let $E^\bu$ be a bounded below complex of finitely
cogenerated cofree $I$\+torsion $R$\+modules endowed with
a quasi-isomorphism $B^\bu\rarrow E^\bu$.
 Then $\Hom_R(E^\bu,C)$ is a bounded above complex of (finitely
generated) free $I$\+contramodule $R$\+modules quasi-isomorphic
to $\Hom_k(B^\bu,C)$.
 We have to show that the map
$$
 B^\bu\ot_R\Hom_R(E^\bu,C[X])\lrarrow C[X]
$$
is a quasi-isomorphism (of complexes of $I$\+torsion $R$\+modules).

 Now, the natural map into the left-hand side from the $X$\+indexed
direct sum of copies of the complex $B^\bu\ot_R\Hom_R(E^\bu,C)$ is
an isomorphism, because $\Hom_R(C,C[X])=\fR[[X]]$ and
$B^\bu\ot_R\fR[[X]]=B^\bu[X]$.
 It remains to recall
the quasi-isomorphism~\eqref{artinian-quotient-main-comparison-map}.
\hbadness=1200
\end{proof}

\begin{rem}
 Returning to the discussion of the dualizing and dedualizing complexes
in the introduction---the definition of a dualizing complex for
a Noetherian commutative ring $R$ with an ideal $I\subset R$ was given
in~\cite[Section~C.1]{Pcosh} (a somewhat more general definition of
a dualizing complex for a Noetherian formal scheme can be found in
the much earlier~\cite[Section~5]{Yek0}).
 In particular, when the quotient ring $R/I$ is Artinian, the injective
$R$\+module $C$, viewed as a one-term complex, is a dualizing complex
for $(R,I)$.

 Note that this is not the same thing as a dualizing complex for
the ring $R$, even in the case of a maximal ideal in a complete
local ring.
 Rather, for any Noetherian commutative ring $R$ with an ideal
$I\subset R$, if $D_R^\bu$ is a dualizing complex for $R$, then
$D^\bu=\boR\Gamma_I(D_R^\bu)$ is a dualizing complex for $(R,I)$.
 E.~g., consider the case of the ring of formal power series
$R=k[[z]]$ in one variable~$z$ over a field~$k$.
 Then the ring $R$, being regular, can be considered as a dualizing
complex over itself.
 Replacing it with a quasi-isomorphic complex of injective
$R$\+modules, we obtain a two-term complex $k((z))\rarrow
k((z))/k[[z]]$, where $k((z))$ is the field of Laurent power series.
 To obtain a dualizing complex for the ideal $I=zk[[z]]\subset R$,
one drops the uniquely divisible component $k((z))$, arriving to
the one-term complex of $R$\+modules $k((z))/k[[z]]\simeq C$
\cite[Example~5.1]{Yek0}.

 For any Noetherian commutative ring $R$ with an ideal $I$ such that
the quotient ring $R/I$ is Artinian, the functors $\Hom_R(C,{-})$
and $C\ot_R{-}$, taking $C[X]$ to $\fR[[X]]$ and back, establish
a covariant equivalence between the additive categories of injective
$I$\+torsion $R$\+modules and projective $I$\+contramodule $R$\+modules.
 The homotopy category of unbounded complexes of injective objects
in $R\modl_{I\tors}$ is equivalent to the coderived category
$\sD^\co(R\modl_{I\tors})$, while the homotopy category of unbounded
complexes of projective objects in $R\modl_{I\ctra}$ is equivalent
to the contraderived category $\sD^\ctr(R\modl_{I\ctra})$.
 Therefore~\cite[Theorem~C.1.4]{Pcosh}, the derived functors of
Hom and tensor product with the dualizing complex $D^\bu=C$ for
$(R,I)$ provide an equivalence between the coderived and
the contraderived categories
$$
 \sD^\co(R\modl_{I\tors})\simeq\sD^\ctr(R\modl_{I\ctra}).
$$

 On the other hand, according to our
Theorem~\ref{dedualizing-artinian-quotient-duality},
the derived functors of Hom and tensor product with the dedualizing
complex $B^\bu$ (e.~g., $B^\bu=\boR\Gamma_I(R)$, as in
Example~\ref{dedualizing-artinian-quotient-example}) provide
equivalences of the conventional and absolute derived categories
$$
 \sD(R\modl_{I\tors})\simeq\sD(R\modl_{I\ctra})\quad\text{and}\quad
 \sD^\abs(R\modl_{I\tors})\simeq\sD^\abs(R\modl_{I\ctra}).
$$
 For an ideal $I$ in a regular Noetherian ring $R$ (with
an Artinian quotient ring $R/I$), one can choose $B^\bu=D^\bu$.
 In this case, the abelian categories $R\modl_{I\tors}$ and
$R\modl_{I\ctra}$ have finite homological dimension, so there is
no difference between the conventional and exotic derived category
constructions for them.
\end{rem}

\Section{Dedualizing Complexes for Weakly Proregular Ideals}
\label{dedualizing-weakly-proregular}

 The aim of this section is to prove
Theorem~\ref{dedualizing-torsion-contra-duality}.
 Taken together with Example~\ref{dedualizing-torsion-example},
it extends the MGM duality result established in
Corollary~\ref{weakly-proregular-duality}
to the absolute derived categories $\sD^{\abs+}$, \,$\sD^{\abs-}$,
and $\sD^\abs$ of torsion modules and contramodules over
a commutative ring $R$ with a weakly proregular ideal~$I$.

 Let $I$ be a finitely generated ideal in a commutative ring $R$,
and let $s_1$,~\dots, $s_m$ be a sequence of its generators,
denoted by~$\s$ for brevity.
 For any element $s\in R$, the complex $T^\bu(R,s)$ from
Section~\ref{derived-contramodules} is the union of its subcomplexes
of free $R$\+mod\-ules $T^\bu_n(R,s)$ generated by the symbols
$\epsilon_0$,~\dots, $\epsilon_{n-1}$ in degree~$0$ and
$\delta_1$,~\dots,~$\delta_n$ in degree~$1$.
 Accordingly, the complex $T^\bu(R,\s)$ is the union of its subcomplexes
of finitely generated free $R$\+modules $T^\bu_n(R,s_1)\ot_R\ot\dotsb
\ot_RT^\bu_n(R,s_m)$.
 The complex $T^\bu_n(R,\s)$ is homotopy equivalent to the complex
$\Tel^\bu_n(R,\s)$.

 Let us denote the dual complexes by $K^n_\bu(R,\s)=
\Hom_R(T_n^\bu(R,\s),R)$.
 In particular, $K^1_\bu(R,\s)$ is the conventional Koszul complex
associated with the sequence of elements $s_1$,~\dots, $s_m\in R$,
while the complex $K^n_\bu(R,\s)$ is quasi-isomorphic to
the Koszul complex $K^1_\bu(R,\s^n)$ associated with the sequence
of elements $s_1^n$,~\dots, $s_m^n\in R$.

 The Thomason--Trobaugh--Neeman theory of compactly generated
triangulated categories is fairly widely known by now (see, e.~g.,
\cite{Neem0} or~\cite{Rou}).
 The following one is an old result.

\begin{prop} \label{derived-torsion-compact}
 For any given $n\ge1$, the complex $K^n_\bu(R,\s)$ is a compact
generator of the triangulated subcategory\/ $\sD_{I\tors}(R\modl)
\subset\sD(R\modl)$ of complexes with $I$\+torsion cohomology
modules in the derived category of $R$\+modules.
 A complex of $R$\+modules with $I$\+torsion cohomology modules
is a compact object of\/ $\sD_{I\tors}(R\modl)$ if and only if it
is a compact object of\/ $\sD(R\modl)$.
\end{prop}

\begin{proof}
 The first assertion goes back to~\cite[Proposition~6.1]{BN};
see also~\cite[Proposition~6.6]{Rou}.
 It suffices to consider the case $n=1$.
 One can easily see that the complexes $K^n_\bu(R,\s)$, viewed as
objects of the homotopy category $\Hot(R\modl)$, can be obtained from
the complex $K^1_\bu(R,\s)$ using the operations of shift and cone.
 Let $M^\bu$ be a complex of $R$\+modules with $I$\+torsion cohomology
modules.
 Then the inductive limit of the complexes $\Hom_R(K^n_\bu(R,\s),M^\bu)
\simeq T_n^\bu(R,\s)\ot_RM^\bu$ is quasi-isomorphic to the complex
$M^\bu$ by Lemma~\ref{cech-complex}(c).
 Hence the complex $M^\bu$ is acyclic whenever all the complexes
$\Hom_R(K^n_\bu(R,\s),M^\bu)$ are.
 The complexes $K^n_\bu(R,\s)$ are compact objects in
$\sD_{I\tors}(R\modl)$, since they are compact in $\sD(R\modl)$.
 This proves the first assertion.
 It follows that the compact objects of $\sD_{I\tors}(R\modl)$ are
precisely those objects that can be obtained from $K^1_\bu(R,\s)$
using the operations of shift, cone, and the passage to a direct
summand, implying the second one (cf.~\cite[Lemma~5.3]{PSY1}).
\end{proof}

\begin{lem} \label{implicit-in-duality-theorem}
\textup{(a)} For any $I$\+torsion $R$\+module $M$, the natural
morphism of complexes of $R$\+modules
$$
 T^\bu(R,\s)\ot_R\Hom_R(T^\bu(R,\s),M)\lrarrow M
$$
is a quasi-isomorphism. \par
\textup{(b)} For any $I$\+contramodule $R$\+module $P$, the natural
morphism of complexes of $R$\+modules
$$
 P\lrarrow\Hom_R(T^\bu(R,\s)\;T^\bu(R,\s)\ot_RP)
$$
is a quasi-isomorphism.
\end{lem}

\begin{proof}
 These computations are implicit in the proof of
Theorem~\ref{complexes-with-t-c-cohomology-equivalence}.
 In view of Lemma~\ref{cech-complex}(c), part~(a) generalizes to
the assertion that for any $R$\+module $M$ the morphism
$$
 T^\bu(R,\s)\ot_RM\rarrow T^\bu(R,\s)\ot_R\Hom_R(T^\bu(R,\s),M)
$$
induced by the natural morphism of complexes $T^\bu(R,\s)\rarrow R$
is a quasi-isomorphism.
 This is essentially~\cite[Lemma~7.6]{PSY}.
 One shows that the morphism of complexes
$$
 T^\bu_n(R,\s)\ot_RM\rarrow\Hom_R(T^\bu(R,\s)\;T^\bu_n(R,\s)\ot_RM)
$$
is a quasi-isomorphism for every~$n\ge1$.
 A cocone of the latter morphism is a complex computing
the $R$\+modules of morphisms
$\Hom_{\sD^\b(R\modl)}(C_\s^\bu(R)\;T^\bu_n(R,\s)\ot_RM[*])$
in the derived category $\sD^\b(R\modl)$.
 So it suffices to check that the $R$\+modules
$\Hom_{\sD^\b(R\modl)}(C_\s^i(R)\;T^\bu_n(R,\s)\ot_RM[*])$
vanish for every $i\ge0$.
 Indeed, the $R$\+module $C_\s^i(R)$ is a finite direct sum of
$R$\+modules in each of which one of the elements~$s_j$ acts
invertibly, while in the complex $T^\bu_n(R,\s)\ot_RM$
all the elements $s_j^n$ act contractibly.

 Similarly, in view of Lemma~\ref{telescope-complex}(c), part~(b)
generalizes to the assertion that for any $R$\+module $M$
the morphism
$$
 \Hom_R(T^\bu(R,\s)\;T^\bu(R,\s)\ot_RM)\rarrow
 \Hom_R(T^\bu(R,\s),M)
$$
induced by the morphism of complexes $T^\bu(R,\s)\rarrow R$ is
a quasi-isomorphism.
 This is essentially~\cite[Lemma~7.2]{PSY}.
 A cocone of the morphism $T^\bu(R,\s)\ot_RM\rarrow M$ is
quasi-isomorphic to the complex $C_\s^\bu(M)$, so it suffices to
check that the complex $\Hom_R(T^\bu(R,\s),C_\s^\bu(M))$ is acyclic.
 Indeed, let us show that the complex $\Hom_R(T^\bu(R,\s),C_\s^i(M))$
is acyclic for every $i\ge0$.
 By Lemma~\ref{telescope-complex}(a), the cohomology modules of
this complex are finite direct sums of $R$\+modules each of which
is simultaneously an $s_j$\+contramodule and an $R[s_j^{-1}]$\+module
for one of the elements~$s_j$.
 All such modules vanish.
\end{proof}

 Let $\fR=\varprojlim_n R/I^n$ denote the $I$\+adic completion of
the ring~$R$.
 The ring $\fR$ is a complete and separated topological ring in its
projective limit topology, which coincides with its $I$\+adic topology.
 An $I$\+torsion $R$\+module is the same thing as a discrete
$\fR$\+module.
 For any set $X$, we denote by $\fR[[X]]=\varprojlim_n R/I^n[X]$
the $R$\+module of all maps of sets $f\:X\rarrow\fR$ converging to
zero in the topology of $\fR$ (i.~e., for any open subset $U\subset\fR$
one has $f(x)\in U$ for all but a finite number of elements $x\in X$).
 As any $R$\+module that is separated and complete in its $I$\+adic
topology, $\fR[[X]]$ is an $I$\+contramodule (see
Section~\ref{derived-contramodules}).

 For the rest of the section we assume that the ideal $I\subset R$
is weakly proregular.
 Then the $R$\+module $\fR[[X]]$ is called the \emph{free
$I$\+contramodule $R$\+module generated by the set~$X$}
(cf.~\cite[Section~2.1]{Prev}).
 Part~(a) of the next lemma justifies the terminology.

\begin{lem} \label{free-contramodules}
\textup{(a)} For any $I$\+contramodule $R$\+module $P$ there is
a bijective correspondence between the $R$\+module morphisms\/
$\fR[[X]]\rarrow P$ and the maps of sets $X\rarrow P$.
 Free $I$\+contramodule $R$\+modules are projective objects in
the abelian category $R\modl_{I\ctra}$, there are enough of them, and
any projective $I$\+contramodule $R$\+module is a direct summand
of a free one. \par
\textup{(b)} For any set $X$ and any integer $n\ge1$,
the morphisms of complexes
$$
 C_\s^\bu(R[X])\sptilde\lrarrow C_\s^\bu(\fR[[X]])\sptilde
 \quad\textup{and}\quad
 T^\bu_n(R,\s)\ot_R R[X]\lrarrow T^\bu_n(R,\s)\ot_R\fR[[X]]
$$ 
induced by the completion map $R[X]\rarrow\fR[[X]]$ are
quasi-isomorphisms. \par
\textup{(c)} For any set $X$, the morphism of complexes
$$
 \Hom_R(T^\bu(R,\s),R[X])\lrarrow\Hom_R(T^\bu(R,\s),\fR[[X]])
$$
induced by the completion map $R[X]\rarrow\fR[[X]]$ is
a quasi-isomorphism.
\end{lem}

\begin{proof}
 Part~(a): by Lemma~\ref{agree-on-flats}, one has
$\Delta_I(R[X])\simeq\fR[[X]]$, so $\Hom_R(R[X],P)\simeq
\Hom_R(\fR[[X]],P)$ for any $I$\+contramodule~$P$.
 The remaining assertions immediately follow.
 Parts~(b\+c): by Lemmas~\ref{agree-on-flats}
and~\ref{delta-finite-homol-dim}(a), the complex
$\Hom_R(T^\bu(R,\s),R[X])$ is quasi-isomorphic to the $R$\+module
$\fR[[X]]$ and the morphism of complexes $R[X]\rarrow
\Hom_R(T^\bu(R,\s),R[X])$ computes the completion morphism
$R[X]\rarrow\fR[[X]]$.
 Now part~(b) is provided by the proof of
Lemma~\ref{implicit-in-duality-theorem}(a) and
part~(c) follows from Lemma~\ref{tensor-product-equivalences}(a).
\end{proof}

 It was shown in~\cite[Theorem~3.4]{Yek1}
and~\cite[Proposition~B.9.1]{Pweak} (see
also~\cite[Theorem~1.5(2)]{PSY2} and~\cite[Proposition~C.5.4]{Pcosh})
that the free $I$\+contramodule $R$\+modules $\fR[[X]]$ are flat
$R$\+modules when the ring $R$ is Noetherian.
 Part~(c) of the following lemma is a weak generalization of
this result to the weakly proregular case.

\begin{lem}  \label{finite-projective-torsion-cohomology-complex}
\textup{(a)} For any finite complex of finitely generated projective
$R$\+modules $K^\bu$ with $I$\+torsion cohomology modules,
the complex\/ $\Hom_R(K^\bu,R)$ also has $I$\+torsion
cohomology modules.
 Moreover, for any complex of $R$\+modules $M^\bu$, the complex\/
$\Hom_R(K^\bu,M^\bu)$ has $I$\+torsion cohomology modules. \par
\textup{(b)} For any finite complex of finitely generated
projective $R$\+modules $K^\bu$ with $I$\+tor\-sion cohomology modules
and any set $X$, the morphism of complexes of $R$\+modules
$$
 K^\bu\ot_R R[X]\lrarrow K^\bu\ot_R\fR[[X]]
$$
induced by the completion map $R[X]\rarrow\fR[[X]]$ is
a quasi-isomorphism. \par
\textup{(c)} Let $K^\bu\rarrow M^\bu$ be a quasi-isomorphism between
a complex of $I$\+torsion $R$\+modules $M^\bu$ and a finite complex
of finitely generated projective $R$\+modules $K^\bu$ with
$I$\+torsion cohomology modules.
 Then for any set $X$ the induced morphism
$$
 K^\bu\ot_R\fR[[X]]\lrarrow M^\bu\ot_R\fR[[X]]
$$
is a quasi-isomorphism of complexes of $R$\+modules.
\end{lem}

\begin{proof}
 Part~(a) does not depend on the weak proregularity assumption yet.
 To prove it, choose a finite sequence of elements~$\s$ generating
the ideal $I\subset R$ and an integer $n\ge1$.
 By Proposition~\ref{derived-torsion-compact}, the complex $K^\bu$,
viewed as an object of the homotopy category $\Hot(R\modl)$, is
a direct summand of a complex obtained from the complex
$K^n_\bu(R,\s)$ using the operations of shift and cone.
 Since the complex $K^n_\bu(R,\s)$ is self-dual up to a shift,
the first assertion follows.
 Moreover, all the cohomology modules of the complex $K^n_\bu(R,\s)$
are annihilated by a large enough power of the ideal~$I$, and
consequently so are all the cohomology modules of the complex $K^\bu$
and every cohomology module of the complex $\Hom_R(K^\bu,M^\bu)$.

 By virtue of the same argument based on
Proposition~\ref{derived-torsion-compact}, part~(b) follows from
Lemma~\ref{free-contramodules}(b).
 To prove part~(c), notice that the morphism $K^\bu\ot_R R[X]\rarrow
M^\bu\ot_R R[X]$ is a quasi-isomorphism because the $R$\+module $R[X]$
is flat, the morphism $M^\bu\ot_R R[X]\rarrow M^\bu\ot_R\fR[[X]]$ is
an isomorphism of complexes of $R$\+modules since $M^\bu$ is a complex
of $I$\+torsion $R$\+modules, and the morphism $K^\bu\ot_R R[X]\rarrow
K^\bu\ot_R\fR[[X]]$ is a quasi-isomorphism by part~(b).
\end{proof}

 Clearly, the abelian category of $I$\+torsion $R$\+modules
$R\modl_{I\tors}$ has enough injective objects.
 Moreover, the maximal $I$\+torsion submodule $\Gamma_I(J)$ of
any injective $R$\+module $J$ is an injective object of
$R\modl_{I\tors}$.
 There are enough injective objects of this particular form in
$R\modl_{I\tors}$, so any injective object of $R\modl_{I\tors}$
a direct summand of the $R$\+module $\Gamma_I(J)$ for some
injective $R$\+module~$J$.

 For a Noetherian ring $R$, it follows from the Artin--Rees lemma
that any injective object of the category $R\modl_{I\tors}$ is
simultaneously an injective object of the category $R\modl$.
 Part~(b) of the next lemma can be sometimes used in lieu of
this assertion in the case of a weakly proregular ideal~$I$.

\begin{lem} \label{injective-torsion}
\textup{(a)} Let $J$ be an injective $R$\+module and $H=\Gamma_I(J)$
be its maximal $I$\+torsion submodule.
 Then for any bounded above complex of projective $R$\+modules $L^\bu$
with $I$\+torsion cohomology modules the morphism of complexes
of $R$\+modules
$$
 \Hom_R(L^\bu,H)\lrarrow\Hom_R(L^\bu,J)
$$
induced by the embedding map $H\rarrow J$ is a quasi-isomorphism. \par
\textup{(b)} Let $L^\bu\rarrow M^\bu$ be a quasi-isomorphism between
a complex of $I$\+torsion $R$\+modules $M^\bu$ and a bounded above
complex of projective $R$\+modules $L^\bu$ with $I$\+torsion
cohomology modules.
 Then for any injective object $H$ of the abelian category of
$I$\+torsion $R$\+modules the induced morphism
$$
 \Hom_R(M^\bu,H)\lrarrow\Hom_R(L^\bu,H)
$$
is a quasi-isomorphism of complexes of $R$\+modules.
\end{lem}

\begin{proof}
 Part~(a) is implicit in the adjunction of derived functors from
the proof of Theorem~\ref{torsion-fully-faithful} together with
the result of Corollary~\ref{complexes-with-torsion-cohomology}.
 To give a direct proof, notice that by the weak proregularity
assumption on the ideal $I$ the complex $C_\s^\bu(J)\sptilde$
is quasi-isomorphic to the $R$\+module $H$ and the morphism of
complexes $C_\s^\bu(J)\sptilde\rarrow J$ computes the embedding
morphism $H\rarrow J$.
 The quasi-isomorphism of complexes of $R$\+modules
$H\rarrow C_\s^\bu(J)\sptilde$ induces a quasi-isomorphism
$\Hom_R(L^\bu,H)\rarrow\Hom_R(L^\bu,C_\s^\bu(J)\sptilde)$, so it remains
to show that the complex $\Hom_R(L^\bu,C_\s^\bu(J))$ is acyclic.
 Here one argues as in the proof of
Lemma~\ref{implicit-in-duality-theorem}(b) together with
the proof of Lemma~\ref{telescope-complex}(a).
 For every $i\ge0$, the cohomology modules of the complex
$\Hom_R(L^\bu,C_\s^i(J))$ are $I$\+contramodule $R$\+modules,
since the complex $L^\bu\ot_R T^\bu(R,s)'$ is contractible for
any $s\in I$ as a bounded above complex of projective $R$\+modules
quasi-isomorphic to the acyclic complex $L^\bu[s^{-1}]$. 
 Being at the same time isomorphic to finite direct sums of
$R$\+modules each of which is an $R[s_j^{-1}]$\+module for one of
the elements~$s_j$, these cohomology modules have to vanish.

 Part~(b) is deduced in the way similar to the proof of
Lemma~\ref{finite-projective-torsion-cohomology-complex}(c).
 One can assume that $H=\Gamma_I(J)$ for some injective $R$\+module
$J$, as there are enough injective $I$\+torsion $R$\+modules of
this form.
 Then the morphism $\Hom_R(M^\bu,J)\rarrow\Hom_R(L^\bu,J)$ is
a quasi-isomorphism because the $R$\+module $J$ is injective,
the morphism $\Hom_R(M^\bu,H)\rarrow\Hom_R(M^\bu,J)$ is an isomorphism,
and the morphism $\Hom_R(L^\bu,H)\rarrow\Hom_R(L^\bu,J)$ is
a quasi-isomorphism by part~(a).
\end{proof}

\begin{rem} \label{ext-tor-comparison}
 In other words, Lemma~\ref{injective-torsion}(b) simply means that
the Ext modules $\Ext_R^i(M,H)$ computed in the abelian category of
$R$\+modules between an $I$\+torsion $R$\+module $M$ and
an injective $I$\+torsion $R$\+module $H$ vanish for all $i>0$.
 This is but a particular case of Theorem~\ref{torsion-fully-faithful}.
 Similarly, it follows from Theorem~\ref{contra-fully-faithful}
that $\Ext_R^i(\fR[[X]],P)=0$ for any $I$\+contramodule $R$\+module
$P$ and all~$i>0$.
 Taking into account the natural isomorphism
$\Hom_R(\Tor_i^R(M,N),J)\simeq\Ext_R^i(N,\Hom_R(M,J))$
for an injective $R$\+module $J$ and any $R$\+modules $M$, $N$,
together with the fact that the $R$\+module $\Hom_R(M,J)$ is
an $I$\+contramodule for any $I$\+torsion $R$\+module $M$ and
any $R$\+module $J$ (see Section~\ref{derived-contramodules}),
one concludes that $\Tor_i^R(M,\fR[[X]])=0$ for any $I$\+torsion
$R$\+module $M$ and any set~$X$.
 This is a stronger version of
Lemma~\ref{finite-projective-torsion-cohomology-complex}(c).
\end{rem}

 As in Section~\ref{dedualizing-artinian-quotient}, a finite complex of
$I$\+torsion $R$\+modules $L^\bu$ is said to have \emph{projective
dimension\/~$\le d$} if one has $\Hom_{\sD^\b(R\modl_{I\tors})}
(L^\bu,M[n])=0$ for all $I$\+torsion $R$\+modules $M$ and all
the integers $n>d$.

 By Lemma~\ref{free-contramodules}(a), the bounded above derived
category $\sD^-(R\modl_{I\ctra})$ is equivalent to the homotopy
category of bounded above complexes of projective $I$\+contramodule
$R$\+modules.
 Given a complex of $I$\+torsion $R$\+modules $M^\bu$ and a bounded
above complex of $I$\+contramodule $R$\+modules $P^\bu$, we denote
by $\Ctrtor^{(R,I)}_*(M^\bu,P^\bu)$ the homology $R$\+modules
$$
 \Ctrtor^{(R,I)}_n(M^\bu,P^\bu)=H^{-n}(M^\bu\ot_RF^\bu)
$$
of the tensor product of the complex of $I$\+torsion $R$\+modules
$M^\bu$ with a bounded above complex of projective $I$\+contramodule
$R$\+modules $F^\bu$ quasi-isomorphic to the complex~$P^\bu$.
 By Remark~\ref{ext-tor-comparison}, the functors
$\Ctrtor^{(R,I)}_n(M^\bu,P^\bu)$ and $H^{-n}(M^\bu\ot_R^\boL P^\bu)$
agree wherever the former of them is defined.

 A finite complex of $I$\+torsion $R$\+modules $L^\bu$ is said to
have \emph{contraflat dimension\/~$\le d$} if one has
$\Ctrtor^{(R,I)}_n(L^\bu,P)=0$ for all $I$\+contramodule $R$\+modules
$P$ and all the integers $n>d$.
 The next one is the main definition of this section.

 A finite complex of $I$\+torsion $R$\+modules $B^\bu$ is called
a \emph{dedualizing complex} for the ideal $I\subset R$ if
the following conditions hold:
\begin{enumerate}
\renewcommand{\theenumi}{\roman{enumi}}
\item the complex $B^\bu$ has finite projective and contraflat
dimensions as a complex of $I$\+torsion $R$\+modules;
\item the homothety map $\fR\rarrow\Hom_{\sD^\b(R\modl_{I\tors})}
(B^\bu,B^\bu[*])$ is an isomorphism of graded rings;
\item for any finite complex of finitely generated projective
$R$\+modules $K^\bu$ with $I$\+torsion cohomology modules,
the complex $\Hom_R(K^\bu,B^\bu)$ is a compact object of
the derived category $\sD(R\modl_{I\tors})$.
\end{enumerate}

 The dedualizing complex $B^\bu$ is viewed as an object of the derived
category $\sD^\b(R\modl_{I\tors})$.

\begin{thm} \label{dedualizing-contravariant-duality}
 Let $I$ be a weakly proregular finitely generated ideal in
a commutative ring~$R$.
 Then for any finite complex of $I$\+torsion $R$\+modules $B^\bu$
satisfying the conditions~(ii\+iii) there is an involutive
anti-equivalence on the full subcategory of compact objects in\/
$\sD(R\modl_{I\tors})$ provided by a derived functor\/
$\boR\Hom_R({-},B^\bu)$.
\end{thm}

\begin{proof}
 By Theorem~\ref{torsion-fully-faithful} and
Corollary~\ref{complexes-with-torsion-cohomology}, one has
$\sD(R\modl_{I\tors})=\sD_{I\tors}(R\modl)\subset\sD(R\modl)$.
 According to Proposition~\ref{derived-torsion-compact}, any compact
object of $\sD_{I\tors}(R\modl)$ is also compact in $\sD(R\modl)$, so
the full subcategory of compact objects in $\sD_{I\tors}(R\modl)$ is
equivalent to the homotopy category of finite complexes of finitely
generated projective $R$\+modules with $I$\+torsion cohomology modules
(see~\cite[Theorem~2.8]{BS}).
 To compute the image of a compact object
$M^\bu\in\sD(R\modl_{I\tors})$ under the functor
$\boR\Hom_R({-},B^\bu)$, one chooses a finite complex of finitely
generated projective $R$\+modules $K^\bu$ together with
a quasi-isomorphism of complexes of $R$\+modules $K^\bu\rarrow M^\bu$
and applies the functor $\Hom_R({-},B^\bu)$ to the complex~$K^\bu$.
 The natural isomorphism of complexes
$$
 \Hom_R(L^\bu,\Hom_R(K^\bu,B^\bu))\simeq\Hom_R(K^\bu,\Hom_R(L^\bu,B^\bu))
$$
for any complexes of $R$\+modules $K^\bu$ and $L^\bu$ makes
the contravariant endofunctor $\boR\Hom_R({-},B^\bu)$ on the category
of compact objects in $\sD(R\modl_{I\tors})$ right adjoint to itself.
 It remains to show that the adjunction morphism $K^\bu\rarrow
\boR\Hom_R(\Hom_R(K^\bu,B^\bu),B^\bu)$ is an isomorphism in
$\sD(R\modl_{I\tors})$ for any finite complex of finitely generated
projective $R$\+modules $K^\bu$ with $I$\+torsion cohomology modules.
{\emergencystretch=1em\par}

 Let ${}'\!\.B^\bu$ be a bounded below complex of injective $R$\+modules
endowed with a quasi-isomorphism $B^\bu\rarrow{}'\!\.B^\bu$.
 Then it suffices to check that the natural morphism
$K^\bu\rarrow\Hom_R(\Hom_R(K^\bu,B^\bu),{}'\!\.B^\bu)\simeq
K^\bu\ot_R\Hom_R(B^\bu,{}'\!\.B^\bu)$ is a quasi-isomorphism of complexes
of $R$\+modules.
 By the condition~(ii) together with
Theorem~\ref{torsion-fully-faithful},
the complex $\Hom_R(B^\bu,{}'\!\.B^\bu)$ is quasi-isomorphic to
the ring $\fR$, so we have to show that the morphism
$K^\bu\rarrow\fR\ot_RK^\bu$ is a quasi-isomorphism.
 But this is the particular case of
Lemma~\ref{finite-projective-torsion-cohomology-complex}(b)
for a one-element set~$X$.
\end{proof}

 The following example explains where one can get a dedualizing
complex for an arbitrary weakly proregular finitely generated ideal
in a commutative ring.

\begin{ex} \label{dedualizing-torsion-example}
 Let $I$ be a weakly proregular finitely generated ideal in
a commutative ring~$R$.
 As in Example~\ref{dedualizing-artinian-quotient-example}, consider
the derived functor of maximal $I$\+torsion submodule
$\boR\Gamma_I\:\sD^\b(R\modl)\rarrow\sD^\b(R\modl_{I\tors})$.
 It is claimed that the complex $B^\bu=\boR\Gamma_I(R)$ is
a dedualizing complex for the ideal $I\subset R$.

 Indeed, one has $\Hom_{\sD^\b(R\modl_{I\tors})}(B^\bu,M[*])\simeq
\Hom_{\sD^\b(R\modl)}(B^\bu,M[*])$ for any $I$\+torsion $R$\+module $M$
by (the case $\st=\b$ of) Theorem~\ref{torsion-fully-faithful}, so
the complex $B^\bu$, having finite projective dimension as
a complex of $R$\+modules for the reason of being isomorphic to
the complex $T^\bu(R,\s)$ in $\sD^\b(R\modl)$, also has finite
projective dimension as a complex of $I$\+torsion $R$\+modules.

 Furthermore, the quasi-isomorphism of complexes of $R$\+modules
$T^\bu(R,\s)\rarrow B^\bu$ induces a quasi-isomorphism of tensor
products $T^\bu(R,\s)\ot_R\fR[[X]]\rarrow B^\bu\ot_R\fR[[X]]$.
 Indeed, the complex $B^\bu$ being a complex of $I$\+torsion
$R$\+modules, the morphism $B^\bu\ot_R R[X]\rarrow B^\bu\ot_R\fR[[X]]$
induced by the completion morphism $R[X]\rarrow\fR[[X]]$ is
an isomorphism of complexes, while the morphism
$T^\bu(R,\s)\ot_R R[X]\rarrow B^\bu\ot_R R[X]$ is
a quasi-isomorphism since the $R$\+module $R[X]$ is flat, and
the morphism $T^\bu(R,\s)\ot_R R[X]\rarrow T^\bu(R,\s)\ot_R\fR[[X]]$
is a quasi-isomorphism by Lemma~\ref{free-contramodules}(b).
 Since $T^\bu(R,\s)$ is a complex of flat $R$\+modules, it follows
that the homology $R$\+modules of the complex $T^\bu(R,\s)\ot_R P$
are isomorphic to the $R$\+modules $\Ctrtor^{(R,I)}_*(B^\bu,P)$
for any $I$\+contramodule $R$\+module $P$, so the complex $B^\bu$
has finite contraflat dimension, too
(cf.\ Lemma~\ref{finite-projective-torsion-cohomology-complex}(c)
and Remark~\ref{ext-tor-comparison}).
 This proves the condition~(i).

 The condition~(ii) is provided by
Lemma~\ref{implicit-in-duality-theorem}(b) applied to
the $I$\+contramodule $R$\+module $\fR$ together with
Lemma~\ref{free-contramodules}(b) applied to a one-element set~$X$.
 Finally, by Proposition~\ref{derived-torsion-compact} it suffices
to prove the condition~(iii) for one of the complexes
$K^\bu=K_\bu^n(R,\s)$, for which it follows from
Lemma~\ref{cech-complex}(c), as the complex
$\Hom_R(K_\bu^n(R,\s),T^\bu(R,\s))$ is quasi-isomorphic to
$T^\bu_n(R,\s)\ot_RC_\s^\bu(R)\sptilde$ and to $T^\bu_n(R,\s)$.

 According to
Lemma~\ref{finite-projective-torsion-cohomology-complex}(a),
the involutive anti-equivalence $\Hom_R({-},R)$ of the category of
finite complexes of finitely generated projective $R$\+modules takes
its full subcategory of complexes with $I$\+torsion cohomology
modules into itself.
 One can say that the choice of the dedualizing complex $B^\bu=
\boR\Gamma_I(R)$ corresponds to the choice of the anti-equivalence
$\boR\Hom_R({-},B^\bu)$ on the subcategory of compact objects in
$\sD(R\modl_{I\tors})$ obtained as the restriction of
the anti-equivalence $\boR\Hom_R({-},R)$ on the category of
compact objects in $\sD(R\modl)$.
\end{ex}

\begin{rem}
 We do not know how the definition of a dedualizing complex considered
in this section relates to the one from
Section~\ref{dedualizing-artinian-quotient}.
 The difference lies in the conditions~(iii).
 One might wish to replace the condition~(iii) in the above definition
with a more conventional finiteness condition by requiring, in
the spirit of the condition~(iii) of
Section~\ref{dedualizing-artinian-quotient}, the cohomology
$R$\+modules of the complex $B^\bu$ to have finitely generated
submodules of elements annihilated by $I^n$ for all (or some) $n\ge1$.
 However, the class of all $I$\+torsion $R$\+modules with finitely
generated submodules of elements annihilated by $I^n$ does not seem
to have good enough homological properties to make such a definition
reasonable even for a Noetherian ring $R$ with a non-Artinian
quotient ring~$R/I$.
 Indeed, the following counterexample demonstrates that this class of
$R$\+modules is \emph{not} closed under the passages to
quotient modules (though it is, of course, preserved by
the passages to arbitrary submodules).

 Let $R=k[x,s]$ be the ring of polynomials in two variables over
a field~$k$ with the ideal $I=(s)$.
 Consider the $I$\+torsion $R$\+module $k[x,s,s^{-1}]/sk[x,s]$ and
the $R$\+submodule $M$ in it with the $k$\+basis consisting of
all the vectors $x^js^{-i}$ with $i\le\nobreak j$.
 Then for any $n\ge1$ the submodule of elements annihilated by~$s^n$
in $M$ is a finitely generated free $k[x]$\+module with the generators
$1$, $xs^{-1}$,~\dots, $x^{n-1}s^{-n+1}$, but the quotient $R$\+module
$M/sM$ of the $R$\+module $M$ is an infinite-dimensional $k$\+vector
space with the basis $x^is^{-i}$, \,$i\ge0$, where both $x$ and~$s$
act by zero.

 Hence even though the argument of
Example~\ref{dedualizing-artinian-quotient-example} still shows that
the derived functor of maximal submodule annihilated by $I^n$ takes
the complex $B^\bu=\boR\Gamma_I(R)$ to a complex with finitely
generated cohomology $R/I^n$\+modules when the ring $R$ is Noetherian,
no analogue of Lemmas~\ref{finitely-cogenerated-complex}
or~\ref{finitely-cogenerated-cofree-complex} seems to be applicable
in this case.
\end{rem}

 The next theorem is the main result of this section and this paper.

\begin{thm} \label{dedualizing-torsion-contra-duality}
 Let $I$ be a weakly proregular finitely generated ideal in
a commutative ring~$R$.
 Then, given a dedualizing complex $B^\bu$ for the ideal $I\subset R$,
for any symbol\/ $\st=\b$, $+$, $-$, $\varnothing$, $\abs+$, $\abs-$,
or\/~$\abs$ there is an equivalence of derived categories\/
$\sD^\st(R\modl_{I\tors})\simeq\sD^\st(R\modl_{I\ctra})$ provided by
mutually inverse functors\/ $\boR\Hom_R(B^\bu,{-})$ and\/
$B^\bu\ot_R^\boL{-}$.
\end{thm}

\begin{proof}
 As in the proof of Theorem~\ref{dedualizing-artinian-quotient-duality},
we assume for simplicity of notation that the complex $B^\bu$ is
concentrated in nonpositive cohomological degrees.
 Let $d$~be an integer greater or equal to both the projective and
the contraflat dimension of $B^\bu$ as a complex of $I$\+torsion
$R$\+modules.

 To construct the image of a complex of $I$\+torsion $R$\+modules
$M^\bu$ under the functor $\boR\Hom_R(B^\bu,{-})$, choose an exact
sequence of complexes of $I$\+torsion $R$\+modules $0\rarrow M^\bu
\rarrow J^{0,\bu}\rarrow J^{1,\bu}\rarrow\dotsb$ with injective
$I$\+torsion $R$\+modules $J^{j,i}$ (i.~e., the $R$\+modules $J^{j,i}$
should be injective objects in the abelian category $R\modl_{I\tors}$).
 Applying the functor $\Hom_R(B^\bu,{-})$ to every complex $0\rarrow
J^{0,i}\rarrow J^{1,i}\rarrow J^{2,i}\rarrow\dotsb$ with $i\in\Z$,
one obtains a nonnegatively graded complex of $I$\+contramodule
$R$\+modules $0\rarrow P^{0,i}\rarrow P^{1,i}\rarrow P^{2,i}
\rarrow\dotsb$.
 According to the projective dimension condition on the dedualizing
complex $B^\bu$, the complex $P^{\bu,i}$ has zero cohomology modules
at the cohomological degrees above~$d$; so it quasi-isomorphic to
its canonical truncation complex $\tau_{\le d}P^{\bu,i}$.
 By the definition, set the object $\boR\Hom_R(B^\bu,M^\bu)$ in
the derived category $\sD^\st(R\modl_{I\ctra})$ be represented by
the total complex of the bicomplex $\tau_{\le d}P^{\bu,\bu}$
concentrated in the cohomological degrees $0\le j\le d$ and $i\in\Z$
(or, when the symbol~$\st$ presumes bounded complexes, the bicomplex
$P^{\bu,\bu}$ is bounded in the respective sense along the grading~$i$).

 Similarly, to construct the image of a complex of $I$\+contramodule
$R$\+modules $P$ under the functor $B^\bu\ot_R^\boL{-}$, one has to
choose an exact sequence of complexes of $I$\+contramodule $R$\+modules
$\dotsb\rarrow F^{-1,\bu}\rarrow F^{0,\bu}\rarrow P^\bu\rarrow0$ with
projective $I$\+contramodule $R$\+modules~$F^{j,i}$.
 Then one applies the functor $B^\bu\ot_R{-}$ to every complex
$\dotsb\rarrow F^{-2,i}\rarrow F^{-1,i}\rarrow F^{0,i}\rarrow 0$ with
$i\in\Z$, obtaining a nonpositively graded complex of $I$\+torsion
$R$\+modules $\dotsb\rarrow M^{-2,i}\rarrow M^{-1,i}\rarrow M^{0,i}
\rarrow0$.
 According to the contraflat dimension condition on the complex $B^\bu$,
the complex $M^{\bu,i}$ has zero cohomology modules at the cohomological
degrees below~$-d$; so it is quasi-isomorphic to its canonical
truncation complex $\tau_{\ge-d}M^{\bu,i}$.
 One sets the object $B^\bu\ot_R^\boL\nobreak P^\bu$ in the derived
category $\sD^\st(R\modl_{I\tors})$ to be represented by the total complex
of the bicomplex $\tau_{\ge-d}(M^{\bu,\bu})$ concentrated in
the cohomological degrees $-d\le j\le 0$ and $i\in\Z$.

 As in the proof of Theorem~\ref{dedualizing-artinian-quotient-duality},
these constructions of two derived functors are but particular cases
of the construction of a derived functor of finite homological dimension
elaborated in Appendix~\ref{derived-finite-homol-dim}.
 According to the results of Appendix~\ref{derived-finite-homol-dim},
the above constructions produce well-defined triangulated functors
$$
 \boR\Hom_R(B^\bu,{-})\:\sD^\st(R\modl_{I\tors})\lrarrow
 \sD^\st(R\modl_{I\ctra})
$$
and
$$
 B^\bu\ot_R^\boL{-}\:\sD^\st(R\modl_{I\ctra})\lrarrow
 \sD^\st(R\modl_{I\tors})
$$
for any derived category symbol $\st=\b$, $+$, $-$, $\varnothing$,
$\abs+$, $\abs-$, or~$\abs$.
 Moreover, the former functor is right adjoint to the latter one.
 All these assertions only depend on the condition~(i) in the definition
of a dedualizing complex.

 It remains to prove that the adjunction morphisms are isomorphisms.
 Since the total complexes of finite acyclic complexes of complexes
are absolutely acyclic, in order to check that the morphism
$B^\bu\ot_R^\boL\boR\Hom_R(B^\bu,M^\bu)\rarrow M^\bu$ is an isomorphism
in the derived category $\sD^\st(R\modl_{I\tors})$ for all
the $\st$\+bounded complexes of $I$\+torsion $R$\+modules $M^\bu$
it suffices to consider the case of a one-term complex $M^\bu=M$
corresponding to a single $I$\+torsion $R$\+module~$M$.
 Furthermore, since a morphism in $\sD^\b(R\modl_{I\tors})$ is
an isomorphism whenever it is an isomorphism in $\sD^+(R\modl_{I\tors})$,
one can view the one-term complex $M$ as an object of the derived
category $\sD^+(R\modl_{I\tors})$ and replace it with an injective
resolution $J^\bu$ of the $R$\+module $M$ in the abelian category
$R\modl_{I\tors}$.
 Applying the same totalization argument to the complex $J^\bu$,
the question reduces to proving that the adjunction morphism
$B^\bu\ot_R^\boL\boR\Hom_R(B^\bu,J)\rarrow J$ is an isomorphism in
$\sD^\b(R\modl_{I\tors})$ for any injective $I$\+torsion
$R$\+module~$J$.

 One has $\boR\Hom_R(B^\bu,J)=\Hom_R(B^\bu,J)$.
 Let $P^\bu$ be a bounded above complex of projective $I$\+contramodule
$R$\+modules endowed with a quasi-isomorphism of complexes of
$I$\+contramodule $R$\+modules $P^\bu\rarrow\Hom_R(B^\bu,J)$, so
$B^\bu\ot_R^\boL\Hom_R(B^\bu,J)=B^\bu\ot_R P^\bu$.
 We have to show that the natural morphism $B^\bu\ot_RP^\bu\rarrow J$
is a quasi-isomorphism of complexes of $I$\+torsion $R$\+modules.
 Let $\s$ be a finite sequence of elements generating the ideal
$I\subset R$.
 By Lemma~\ref{cech-complex}(c), it suffices to check that
the induced morphism of complexes
$$
 T_n^\bu(R,\s)\ot_RB^\bu\ot_RP^\bu\lrarrow T_n^\bu(R,\s)\ot_R J
$$
is a quasi-isomorphism for every $n\ge1$ (as the complex of flat
$R$\+modules $C_\s^\bu(R)\sptilde$ is quasi-isomorphic to the complex
of free $R$\+modules $T^\bu(R,\s)$, which is the inductive limit of
complexes of finitely generated free $R$\+modules $T_n^\bu(R,\s)$).

 By the condition~(iii) together with
Lemma~\ref{finite-projective-torsion-cohomology-complex}(a) and
Proposition~\ref{derived-torsion-compact}, there exists a finite
complex of finitely generated projective $R$\+modules $L^\bu$
together with a quasi-isomorphism of complexes of $R$\+modules
$L^\bu\rarrow T_n^\bu(R,\s)\ot_RB^\bu$.
 By Lemma~\ref{finite-projective-torsion-cohomology-complex}(c),
the induced morphism of complexes of $R$\+modules
$$
 L^\bu\ot_RP^\bu\lrarrow T_n^\bu(R,\s)\ot_R B^\bu\ot_R P^\bu
$$
is a quasi-isomorphism.
 So is the morphism of complexes of $R$\+modules
$L^\bu\ot_R P^\bu\rarrow L^\bu\ot_R \Hom_R(B^\bu,J)$.
 Hence it remains to check that the morphism
$$
 L^\bu\ot_R\Hom_R(B^\bu,J)\lrarrow T_n^\bu(R,\s)\ot_R J
$$
induced by the morphism $L^\bu\rarrow T_n^\bu(R,\s)\ot_R B^\bu$
is a quasi-isomorphism.

 One has $L^\bu\ot_R\Hom_R(B^\bu,J)\simeq\Hom_R(\Hom_R(L^\bu,B^\bu),J)$.
 By Theorem~\ref{dedualizing-contravariant-duality}, the natural
morphism of complexes of $R$\+modules
$$
 K^n_\bu(R,\s) = \Hom_R(T_n^\bu(R,\s),R)\lrarrow\Hom_R(L^\bu,B^\bu)
$$
is a quasi-isomorphism.
 Applying Lemma~\ref{injective-torsion}(b), we conclude that
the induced morphism of complexes
$$
 \Hom_R(\Hom_R(L^\bu,B^\bu),J)\lrarrow\Hom_R(K_\bu^n(R,\s),J)
 \.\simeq\. T_n^\bu(R,\s)\ot_R J
$$
is a quasi-isomorphism, too.

 Similarly, in order to prove that the adjunction morphism
$P^\bu\rarrow\boR\Hom_R(B^\bu\;\allowbreak B^\bu\ot_R^\boL P^\bu)$
is an isomorphism in the derived category $\sD^\st(R\modl_{I\ctra})$
for any $\st$\+bounded complex of $I$\+contramodule $R$\+modules
$P^\bu$, it suffices to check that it is an isomorphism in
$\sD^\b(R\modl_{I\ctra})$ for any free $I$\+contramodule $R$\+module
$F=\fR[[X]]$ viewed as a one-term complex in $\sD^\b(R\modl_{I\ctra})$.
 One has $B^\bu\ot_R^\boL\fR[[X]]\simeq B^\bu\ot_R R[X]=B^\bu[X]$.
 Let $B^\bu\rarrow E^\bu$ and $B^\bu[X]\rarrow E^\bu_X$ be
quasi-isomorphisms from the complexes $B^\bu$ and $B^\bu[X]$ to
bounded below complexes of injective objects in $R\modl_{I\tors}$.
 We have to show that the natural morphism $\fR[[X]]\rarrow
\Hom_R(B^\bu,E_X^\bu)$ is a quasi-isomorphism of complexes of
$R$\+modules.

 By Lemma~\ref{cech-complex}(c), the morphism of complexes
of $I$\+torsion $R$\+modules $T^\bu(R,\s)\ot_RB^\bu\rarrow B^\bu$
is a quasi-isomorphism.
 Hence the induced morphism
\begin{equation} \label{telescope-insertion-morphism}
 \Hom_R(B^\bu,E_X^\bu)\lrarrow\Hom_R(T^\bu(R,\s)\ot_RB^\bu\;E_X^\bu)
\end{equation}
is a quasi-isomorphism, too.
 The complex in the target of this morphism is
the projective limit of the projective system of complexes
$\Hom_R(T_n^\bu(R,\s)\ot_R B^\bu\; E_X^\bu)$ and termwise
surjective morphisms between them.
 By the condition~(iii), every complex $T_n^\bu(R,\s)\ot_RB^\bu$
is a compact object of the derived category $\sD(R\modl_{I\tors})$.
 Hence the $X$\+indexed family of morphisms of complexes
$E^\bu\rarrow E_X^\bu$ corresponding to the embeddings
$B^\bu\rarrow B^\bu[X]$ induces a quasi-isomorphism
\begin{equation} \label{first-proj-system-quasi-iso}
 \Hom_R(T_n^\bu(R,\s)\ot_R B^\bu\;E^\bu)[X]\lrarrow
 \Hom_R(T_n^\bu(R,\s)\ot_R B^\bu\;E^\bu_X).
\end{equation}

 Furthermore, one has
$$
 \Hom_R(T_n^\bu(R,\s)\ot_R B^\bu\;E^\bu)
 \.\simeq\.\Hom_R(T_n^\bu(R,\s),\Hom_R(B^\bu,E^\bu)).
$$
 By the condition~(ii), the natural morphism $\fR\rarrow
\Hom_R(B^\bu,E^\bu)$ is a quasi-isomorphism, hence so is
the induced morphism
$$
 \Hom_R(T_n^\bu(R,\s),\fR)\lrarrow
 \Hom_R(T_n^\bu(R,\s),\Hom_R(B^\bu,E^\bu)).
$$
 By Lemma~\ref{finite-projective-torsion-cohomology-complex}(a\+b),
the morphism of complexes
$$
 \Hom_R(T_n^\bu(R,\s),R)\rarrow\Hom_R(T_n^\bu(R,\s),\fR)
$$
induced by the completion morphism $R\rarrow\fR$ is
a quasi-isomorphism, too.
 Passing to the composition and taking the direct sum over~$X$,
we obtain a quasi-isomorphism of complexes of $R$\+modules
\begin{equation} \label{second-proj-system-quasi-iso}
 \Hom_R(T_n^\bu(R,\s),R[X])\lrarrow
 \Hom_R(T_n^\bu(R,\s)\ot_R B^\bu\;E^\bu)[X]
\end{equation}
for every $n\ge1$.

 Both sides of the quasi-isomorphisms
\eqref{first-proj-system-quasi-iso}
and~\eqref{second-proj-system-quasi-iso} naturally form projective
systems of complexes and termwise surjective morphisms between them,
and the quasi-isomorphisms form commutative diagrams with
the morphisms in the projective systems.
 Therefore, the induced morphisms between the projective limits are
also quasi-isomorphisms.
 Furthermore, the complexes $\Hom_R(T_n^\bu(R,\s),R[X])$ have
$I$\+torsion cohomology modules, so the natural morphism
$R[X]\rarrow\Hom_R(T_n^\bu(R,\s),R[X])$ extends naturally to
a cohomology morphism
$$
 \fR[[X]]\lrarrow H^0\Hom_R(T_n^\bu(R,\s),R[X]).
$$
 Due to the weak proregularity condition, after the passage to
the projective limit this provides a cohomology morphism 
$$\textstyle
 \fR[[X]]\lrarrow H^0\varprojlim_n\Hom_R(T_n^\bu(R,\s),R[X])
$$
forming a commutative diagram with the cohomology morphisms
induced by the morphism of complexes $\fR[[X]]\rarrow
\Hom_R(B^\bu,E_X^\bu)$, the morphism of
complexes~\eqref{telescope-insertion-morphism}, and
the projective limits of the morphisms of
complexes~(\ref{first-proj-system-quasi-iso}\+-%
\ref{second-proj-system-quasi-iso}).

 Finally, by Lemma~\ref{free-contramodules}(c) the morphism
of complexes
$$\textstyle
 \varprojlim_n\Hom_R(T_n^\bu(R,\s),R[X])\.\simeq\.
 \Hom_R(T^\bu(R,\s),R[X])\lrarrow\Hom_R(T^\bu(R,\s),\fR[[X]])
$$
induced by the completion morphism $R[X]\rarrow\fR[[X]]$ is
a quasi-isomorphism, while by Lemma~\ref{telescope-complex}(c)
the morphism
$$
 \fR[[X]]\lrarrow\Hom_R(T^\bu(R,\s),\.\fR[[X]])
$$
is a quasi-isomorphism.
 Comparing these observations finishes the proof.
\end{proof}

\appendix

\Section{Exotic Derived Categories}  \label{exotic-derived}

 This appendix, consisting almost entirely of the definitions,
is included for the benefit of the reader who may have no other
sources handy when looking into this paper.
 Detailed expositions can be found in~\cite{Psemi,Pkoszul,EP}; and
the most relevant source for our present purposes
is~\cite[Appendix~A]{Pcosh}.

 Given an additive category $\sA$, we denote by $\Hot(\sA)$
the homotopy category of complexes over $\sA$, by $\Hot^+(\sA)$ its
full subcategory of bounded below complexes, by $\Hot^-(\sA)\subset
\Hot(\sA)$ the full subcategory of bounded above complexes, and
by $\Hot^\b(\sA)$ the full subcategory of complexes bounded on
both sides.

 For a general discussion of exact categories, we refer to
the overview paper~\cite{Bueh}.
 A sequence of objects and morphisms $E\rarrow F\rarrow G$ in
an exact category $\sE$ is said to be \emph{exact} (at the term $F$)
if it is composed of an admissible epimorphism $E\rarrow S$, a short
exact sequence $0\rarrow S\rarrow F\rarrow T\rarrow0$, and an admissible
monomorphism $T\rarrow G$.
 A complex $E^\bu$ in an exact category $\sE$ is said to be
\emph{acyclic} if it is homotopy equivalent to an exact complex,
or equivalently, if it is a direct summand of an exact complex.
 For any symbol $\st=\b$, $+$, $-$, or~$\varnothing$, the quotient
category of the homotopy category $\Hot^\st(\sE)$ by its thick
subcategory of acyclic complexes is called the (conventional)
\emph{derived category} of $\sE$ and denoted by $\sD^\st(\sE)$.

 The exotic derived category symbols are $\st=\abs+$, $\abs-$,
$\co$, $\ctr$, or~$\abs$.
 The related exotic derived category $\sD^\st(\sE)$ is defined as
a certain quotient category of the corresponding homotopy category,
which is $\Hot^+(\sE)$ when $\st=\abs+$, or $\Hot^-(\sE)$ when
$\st=\abs-$, or $\Hot(\sE)$ in the remaining cases $\st=\co$,
$\ctr$, or~$\abs$.
 It is sometimes convenient to denote simply by $\Hot^\st(\sE)$
the (respectively bounded or unbounded) homotopy category
corresponding to an exotic symbol~$\st$.

 A short exact sequence of complexes in an exact category $\sE$ can be
viewed as a bicomplex with three rows.
 As such, it has a total complex.
 A complex in an exact category $\sE$ is called \emph {absolutely
acyclic} if it belongs to the minimal thick subcategory of
the homotopy category $\Hot(\sE)$ containing the totalizations of
all the short exact sequences of complexes in~$\sE$.
 The quotient category of the homotopy category $\Hot(\sE)$ by its
thick subcategory of absolutely acyclic complexes is denoted by
$\sD^\abs(\sE)$ and called the \emph{absolute derived category} of
an exact category~$\sE$.

 The definitions of the absolute derived categories $\sD^{\abs+}(\sE)$
and $\sD^{\abs-}(\sE)$ of bounded above or below complexes in
an exact category $\sE$ are similar.
 A $\st$\+bounded complex in $\sE$ is said to be absolutely acyclic
with respect to the class of $\st$\+bounded complexes if it belongs to
the minimal thick subcategory of the homotopy category $\Hot^\st(\sE)$
containing the totalizations of all the short exact sequences of
$\st$\+bounded complexes in~$\sE$.
 Given a symbol $\st=\abs+$ or~$\abs-$, the quotient category of
the homotopy category $\Hot^\st(\sE)$ by its thick subcategory of
complexes absolutely acyclic with respect to the class of
$\st$\+bounded complexes is denoted by $\sD^\st(\sE)$.

 Any bounded acyclic complex is absolutely acyclic (with respect to
the class of bounded complexes).
 Any bounded above complex that is absolutely acyclic is also
absolutely acyclic with respect to the class of bounded above
complexes; and any bounded below complex that is absolutely acyclic
is also absolutely acyclic with respect to the class of bounded
below complexes~\cite[Lemma~A.1.1]{Pcosh}.

 Assume that arbitrary infinite direct sums exist in the exact
category $\sE$ and the class of short exact sequences is closed under
infinite direct sums.
 Then a complex in the category $\sE$ is called \emph{coacyclic} if
it belongs to the minimal triangulated subcategory of $\Hot(\sE)$
containing the absolutely acyclic complexes and closed under
infinite direct sums.
 The quotient category of the homotopy category $\Hot(\sE)$ by its
thick subcategory of coacyclic complexes is denoted by $\sD^\co(\sE)$
and called the \emph{coderived category} of the exact category~$\sE$.

 Similarly, assume that arbitrary infinite products exist in the exact
category $\sE$ and the class of short exact sequences is closed under
infinite products.
 Then a complex in the category $\sE$ is called \emph{contraacyclic} if
it belongs to the minimal triangulated subcategory of $\Hot(\sE)$
containing the absolutely acyclic complexes and closed under infinite
products.
 The quotient category of the homotopy category $\Hot(\sE)$ by its
thick subcategory of contraacyclic complexes is denoted by
$\sD^\ctr(\sE)$ and called the \emph{contraderived category} of
the exact category~$\sE$.

 Any coacyclic complex in an exact category with exact functors of
infinite direct sum is acyclic, and any contraacyclic complex in
an exact category with exact functors of infinite product is acyclic,
but the converse is not generally true~\cite[Examples~3.3]{Pkoszul}.
 So the coderived and the contraderived categories are ``larger''
than the conventional unbounded derived category; and the absolute
derived category is larger still.

 Any bounded below acyclic complex in an exact category is
coacyclic~\cite[Lemma~2.1]{Psemi}.
 Any bounded above acyclic complex in an exact category is
contraacyclic~\cite[Lemma~4.1]{Psemi}.
 Any acyclic complex in an exact category of finite homological
dimension is absolutely acyclic~\cite[Remark~2.1]{Psemi}.

 The conventional derived categories (with the symbols $+$, $-$,
or~$\varnothing$) are also known as \emph{derived categories of
the first kind}.
 The coderived, contraderived, and absolute derived categories
(with the symbols $\abs+$, $\abs-$, $\co$, $\ctr$, or~$\abs$)
are known as \emph{derived categories of the second kind}.
 The bounded derived category $\sD^\b(\sE)$ of an exact category $\sE$
can be thought of as belonging to both classes at the same time.

\Section{Derived Functors of Finite Homological Dimension}
\label{derived-finite-homol-dim}

 The constructions of derived functors of finite homological dimension
between various derived categories $\sD^\st$ in
Sections~\ref{derived-torsion-modules}\+-\ref{derived-contramodules}
are based on the technique of~\cite[Section~A.5]{Pcosh}, which is
not sufficient for the purposes of
Sections~\ref{dedualizing-artinian-quotient}\+-%
\ref{dedualizing-weakly-proregular}.
 The more sophisticated technique required there is developed in
this appendix.

 For any additive category $\sE$, we denote by $\sC^+(\sE)$
the DG\+category of bounded below complexes in~$\sE$.
 The additive category of bounded below complexes in $\sE$ and closed
morphisms of degree zero between them will be denoted by the same
symbol $\sC^+(\sE)$; it will be clear from the context which category
structure on $\sC^+(\sE)$ is presumed.
 When $\sE$ is an exact category, the full subcategory $\sC^{\ge0}(\sE)
\subset\sC^+(\sE)$ of all the complexes $0\rarrow E^0\rarrow E^1\rarrow
E^2\rarrow\dotsb$ in $\sE$ (and closed morphisms of degree zero between
them) has a natural exact category structure where a short sequence
of complexes is exact if and only if it is termwise exact in~$\sE$.

 Let $\sA$ be an exact category and $\sJ\subset\sA$ be a full
subcategory, closed under extensions and the passages to the cokernels
of admissible monomorphisms in $\sA$ and such that any object of $\sA$
is the source of an admissible monomorphism into an object of~$\sJ$.
 Being closed under extensions, the full subcategory $\sJ$ inherits
the exact category structure of the ambient category~$\sA$.
 Notice that a closed morphism in $\sC^+(\sJ)$ is a quasi-isomorphism
of complexes in the exact category $\sJ$ if and only if it is
a quasi-isomorphism of complexes in~$\sA$.

 Let $\sC_\sA^{\ge0}(\sJ)$ denote the full subcategory in $\sC^{\ge0}(\sJ)$
consisting of all the complexes $0\rarrow J^0\rarrow J^1\rarrow\dotsb$
in $\sJ$ for which there exists an object $A\in\sA$ together with
a morphism $A\rarrow J^0$ such that the sequence $0\rarrow A\rarrow
J^0\rarrow J^1\rarrow\dotsb$ is exact in the category~$\sA$.
 By the definition, one has $\sC_\sA^{\ge0}(\sJ)=\sC^{\ge0}(\sJ)\cap
\sC_\sA^{\ge0}(\sA)\subset\sC^{\ge0}(\sA)$.
 The full subcategory $\sC_\sA^{\ge0}(\sJ)$ is closed under extensions
and the passages to the cokernels of admissible monomorphisms in
$\sC^{\ge0}(\sJ)$, so it inherits the exact category structure.

 Let $\sB$ be another exact category; suppose that it contains
the images of idempotent endomorphisms of its objects.
 Let $d\ge0$ be an integer.
 Denote by $\sC^{\ge0}(\sB)^{\le d}\subset\sC^{\ge0}(\sB)$ the full
subcategory consisting of all the complexes $0\rarrow B^0\rarrow B^1
\rarrow\dotsb$ in $\sB$ such that the sequence $B^d\rarrow B^{d+1}
\rarrow B^{d+2}\rarrow\dotsb$ is exact in~$\sB$.
 In particular, one has $\sC^{\ge0}(\sB)^{\le0}=\sC_\sB^{\ge0}(\sB)$.
 Being closed under extensions and the passages to the cokernels
of admissible monomorphisms, the full subcategory
$\sC^{\ge0}(\sB)^{\le d}$ inherits the exact category structure of
the category $\sC^{\ge0}(\sB)$.
 Denote also by $\sC^{[0,d]}(\sB)$ the exact category of finite
complexes $B^0\rarrow\dotsb\rarrow B^d$ in~$\sB$.
 Then there is an exact functor of canonical truncation
$\tau_{\le d}\:\sC^{\ge0}(\sB)^{\le d}\rarrow\sC^{[0,d]}(\sB)$.

 Suppose that we are given a DG\+functor $\Psi\:\sC^+(\sJ)\rarrow
\sC^+(\sB)$ taking quasi-isomorphisms of complexes belonging to
$\sC_\sA^{\ge0}(\sJ)$ to quasi-isomorphisms of complexes in
the exact category~$\sB$.
 Suppose further that the restriction of the functor $\Psi$ to
the subcategory $\sC_\sA^{\ge0}(\sJ)\subset\sC^+(\sJ)$ is
an exact functor between exact categories $\Psi\:\sC^{\ge0}_\sA(\sJ)
\rarrow\sC^{\ge0}(\sB)^{\le d}$.
 Composing the restriction of the functor $\Psi$ with the functor
of canonical truncation, we obtain an exact functor
$$
 \tau_{\le d}\Psi\:\sC^{\ge0}_\sA(\sJ)\lrarrow\sC^{[0,d]}(\sB).
$$
 Our aim is to construct the right derived functor
$$
 \boR\Psi\:\sD^\st(\sA)\lrarrow\sD^\st(\sB)
$$
acting between any bounded or unbounded, conventional or absolute
derived categories $\sD^\st$ with the symbol $\st=\b$, $+$, $-$,
$\varnothing$, $\abs+$, $\abs-$, or~$\abs$
(see Appendix~\ref{exotic-derived} or~\cite[Section~A.1]{Pcosh} for
the definitions and~\cite[Section~3]{Pkoszul} for a discussion).

 When the exact categories $\sA$ and $\sB$ have exact functors of
infinite direct sum, the full subcategory $\sJ\subset\sA$ is closed
under infinite direct sums, and the functor $\Psi$ preserves infinite
direct sums, there will be also the derived functor $\boR\Psi$
acting between the coderived categories $\sD^\co$ of the exact
categories $\sA$ and~$\sB$.
 Similarly, when the exact categories $\sB$ and $\sA$ have exact
functors of infinite product, the full subcategory $\sJ\subset\sA$
is closed under infinite products, and the functor $\Psi$ preserves
infinite products, we will also have the derived functor $\boR\Psi$
acting between the contraderived categories $\sD^\ctr$ of the exact
categories $\sA$ and~$\sB$.

 We start with a couple of somewhat tedious but rather
straightforward lemmas.

\begin{lem} \label{object-resolutions}
\textup{(a)} For any object $E$ in the category\/ $\sA$ there exists
an exact sequence\/ $0\rarrow E\rarrow J^0\rarrow J^1\rarrow J^2
\rarrow\dotsb$ in the category\/ $\sA$ with the objects $J^j$ belonging
to the full subcategory\/ $\sJ\subset\sA$. \par
\textup{(b)} For any two morphisms $E\rarrow G$ and $F\rarrow G$ and
two exact sequences\/ $0\rarrow E\rarrow K^0\rarrow K^1\rarrow\dotsb$
and\/ $0\rarrow F\rarrow L^0\rarrow L^1\rarrow\dotsb$ in the category\/
$\sA$ there exists an exact sequence\/ $0\rarrow G\rarrow J^0\rarrow
J^1\rarrow\dotsb$ in the category\/ $\sA$ together with two morphisms
of complexes $K^\bu\rarrow J^\bu$ and $L^\bu\rarrow J^\bu$ forming
a commutative diagram with the morphisms $E\rarrow G$, \ $F\rarrow G$,
\ $E\rarrow K^\bu$, and $F\rarrow L^\bu$ and such that the objects $J^j$
belong to the full subcategory\/ $\sJ\subset\sA$. \par
\textup{(c)} For any exact sequence\/ $0\rarrow E\rarrow K^0\rarrow K^1
\rarrow\dotsb$ and any complex\/ $0\rarrow L^0\rarrow L^1\rarrow L^2
\rarrow\dotsb$ in the category\/ $\sA$, and for any morphism of
complexes $K^\bu\rarrow L^\bu$ with a vanishing composition $E\rarrow
K^\bu\rarrow L^\bu$, there exists a complex\/ $0\rarrow J^0\rarrow J^1
\rarrow J^2\rarrow\dotsb$ in the category\/ $\sJ$ together with
a quasi-isomorphism $L^\bu\rarrow J^\bu$ of complexes in the category\/
$\sA$ such that the composition $K^\bu\rarrow L^\bu\rarrow J^\bu$ is
a contractible morphism of complexes. \par
\textup{(d)} For any short exact sequence\/ $0\rarrow E^1\rarrow E^2
\rarrow E^3\rarrow0$ in the category\/ $\sA$ there exists
an exact sequence of short exact sequences\/ $0\rarrow E^\bu\rarrow
J^{\bu,0}\rarrow J^{\bu,1}\rarrow J^{\bu,2}\rarrow\dotsb$ in
the category\/ $\sA$ with the objects $J^{k,j}$ belonging to
the full subcategory\/ $\sJ\subset\sA$.
\end{lem}

\begin{proof}
 Part~(a) follows immediately by induction from the condition that
any object of $\sA$ is the source of an admissible monomorphism into
an object of~$\sJ$.
 To prove part~(b), one also proceeds step by step, starting from
choosing an admissible monomorphism from the fibered coproduct of
the objects $G$ and $K^0\oplus L^0$ over $E\oplus F$ into an object
$J^0\in\sJ$, then applying the same construction to the cokernels
of the morphisms $E\rarrow K^0$, \ $F\rarrow L^0$, and $G\rarrow J^0$
and the admissible monomorphisms from the former two to the objects
$K^1$ and $L^1$, etc.

 The proof of part~(c) is also a step-by-step construction procedure.
 Let $Z^j$ denote the images of the morphisms $K^{j-1}\rarrow K^j$ in
the exact sequence $0\rarrow E\rarrow K^0\rarrow K^1\rarrow\dotsb$.
 Then the morphism $K^0\rarrow L^0$ factorizes through the admissible
epimorphism $K^0\rarrow Z^1$, so there is a morphism $Z^1\rarrow L^0$.
 Consider the fibered coproduct of the objects $K^1$ and $L^0$ over
the object $Z^1$ in the category $\sA$ and choose an admissible
monomorphism from it into an object $J^0\in\sJ$.
 Then there are a natural admissible monomorphism $L^0\rarrow J^0$
and a natural morphism $K^1\rarrow J^0$.
 Denote by $M^1$ the fibered coproduct of the objects $J^0$ and
$L^1$ over the object $L^0$ in the category~$\sA$.
 Then there are a natural admissible monomorphism $L^1\rarrow M^1$
and natural morphisms $J^0\rarrow M^1\rarrow L^2$ with a vanishing
composition in the category~$\sA$.
 The difference of the compositions $K^1\rarrow L^1\rarrow M^1$
and $K^1\rarrow J^0\rarrow M^1$ is annihilated by the composition
with the morphism $K^0\rarrow K^1$, so it factorizes through
the admissible epimorphism $K^1\rarrow Z^2$, providing a morphism
$Z^2\rarrow M^1$.
 Consider the fibered coproduct of the objects $K^2$ and $M^1$
over the object $Z^2$ in the category $\sA$ and choose an admissible
monomorphism from it into an object $J^1\in\sJ$.
 Then there are a natural admissible monomorphism $M^1\rarrow J^1$
and a morphism $K^2\rarrow J^1$ in the category~$\sA$.
 The composition $J^0\rarrow M^1\rarrow J^1$ provides a morphism
$J^0\rarrow J^1$ in the category~$\sJ$.
 Denote by $M^2$ the fibered coproduct of the objects $J^1$ and
$L^2$ over the object $M^1$ in the category~$\sA$, etc.

 To check part~(d), one applies part~(a) to the exact category of
short exact sequences in $\sA$ and its full subcategory of short
exact sequences in~$\sJ$.
 It suffices to show that the short exact sequence $0\rarrow E^1
\rarrow E^2\rarrow E^3\rarrow0$ is the source of an admissible
monomorphism into a short exact sequence of objects from~$\sJ$.
 For this purpose, one picks admissible monomorphisms $E^2\rarrow J^1$
and $E^3\rarrow J^3$ into objects $J^1$, $J^3\in\sJ$; then there is
an admissible monomorphism from our short exact sequence to
the split short exact sequence $0\rarrow J^1\rarrow J^1\oplus J^3
\rarrow J^3\rarrow0$.
\end{proof}

\begin{lem} \label{complex-resolutions}
\textup{(a)} For any complex $E^\bu$ in the category\/ $\sA$ there
exists an exact sequence of complexes
$$
 0\lrarrow E^\bu\lrarrow J^{0,\bu}\lrarrow J^{1,\bu}\lrarrow
 J^{2,\bu}\lrarrow\dotsb
$$
in the category\/ $\sA$ with the objects $J^{j,i}$ belonging to
the full subcategory\/ $\sJ\subset\sA$. \par
\textup{(b)} For any two morphisms of complexes $E^\bu\rarrow G^\bu$
and $F^\bu\rarrow G^\bu$ and two exact sequences of complexes\/
$0\rarrow E^\bu\rarrow K^{0,\bu}\rarrow K^{1,\bu}\rarrow\dotsb$
and\/ $0\rarrow F^\bu\rarrow L^{0,\bu}\rarrow L^{1,\bu}\rarrow\dotsb$
in the category\/ $\sA$ there exists an exact sequence of complexes\/
$0\rarrow G^\bu\rarrow J^{0,\bu}\rarrow J^{1,\bu}\rarrow\dotsb$ in
the category\/ $\sA$ together with two morphisms of bicomplexes
$K^{\bu,\bu}\rarrow J^{\bu,\bu}$ and $L^{\bu,\bu}\rarrow J^{\bu,\bu}$
forming a commutative diagram with the morphisms $E^\bu\rarrow G^\bu$,
\ $F^\bu\rarrow G^\bu$, \ $E^\bu\rarrow K^{\bu,\bu}$, and $F^\bu\rarrow
L^{\bu,\bu}$ and such that the objects $J^{j,i}$ belong to the full
subcategory\/ $\sJ\subset\sA$. \par
\textup{(c)} For any exact sequence of complexes\/ $0\rarrow E^\bu
\rarrow K^{0,\bu}\rarrow K^{1,\bu}\rarrow\dotsb$ and any complex of
complexes\/ $0\rarrow L^{0,\bu}\rarrow L^{1,\bu}\rarrow L^{2,\bu}\rarrow
\dotsb$ in the category\/ $\sA$, and for any morphism of bicomplexes
$K^{\bu,\bu}\rarrow L^{\bu,\bu}$ with a vanishing composition
$E^\bu\rarrow K^{\bu,\bu}\rarrow L^{\bu,\bu}$, there exists a complex
of complexes $0\rarrow J^{0,\bu}\rarrow J^{1,\bu}\rarrow J^{2,\bu}
\rarrow\dotsb$ in the category\/ $\sJ$ together with
a quasi-isomorphism $L^{\bu,\bu}\rarrow J^{\bu,\bu}$ of complexes along
the first grading~$j$ in the exact category of complexes along
the second grading~$i$ in the exact category\/ $\sA$ such that
the composition $K^{\bu,\bu}\rarrow L^{\bu,\bu}\rarrow J^{\bu,\bu}$ is
a contractible morphism of complexes of complexes with a contracting
homotopy $h\:K^{j,i}\rarrow J^{j-1,i}$ defined for all $j\ge0$
and $i\in\Z$. \par
\textup{(d)} For any short exact sequence of complexes\/ $0\rarrow
E^{1,\bu}\rarrow E^{2,\bu}\rarrow E^{3,\bu}\rarrow0$ in the category\/
$\sA$ there exists an exact sequence of short exact sequences of
complexes\/ $0\rarrow E^{\bu,\bu}\rarrow J^{\bu,0,\bu}\rarrow J^{\bu,1,\bu}
\rarrow\dotsb$ in the category\/ $\sA$ with the objects $J^{k,j,i}$
belonging to the full subcategory\/ $\sJ\subset\sA$. \par
\textup{(e)} For any exact complex $E^\bu$ in the category\/ $\sA$
there exists an exact sequence of exact complexes\/ $0\rarrow E^\bu
\rarrow J^{0,\bu}\rarrow J^{1,\bu}\rarrow\dotsb$ in the category\/ $\sA$
with the objects $J^{j,i}$ belonging to the full subcategory\/
$\sJ\subset\sA$ and the complexes $J^{j,\bu}$ exact in the exact
category\/~$\sJ$. \par
\textup{(f)} If in any of the parts~(a\+e) the original data
of a complex, a pair of morphisms of complexes, a morphism of
bicomplexes, or a short exact sequence of complexes consisted of
(bi)complex(es) bounded above, below, or on both sides along
the grading~$i$, then the bicomplex or the short exact sequence
of bicomplexes whose existence is asserted can be also chosen to be
bounded on the respective side(s) along the grading~$i$.
\end{lem}

\begin{proof}
 Parts~(a\+d) are obtained from parts~(a\+d) of
Lemma~\ref{object-resolutions} by substituting the exact category
of complexes in $\sA$ and its full subcategory of complexes in
$\sJ$ in place of the exact category $\sA$ and its full
subcategory~$\sJ$.
 One only has to show that any complex $E^\bu$ in $\sA$ is the source
of an admissible monomorphism into a complex $J^\bu$ with the terms
from~$\sJ$.
 Here it suffices to pick admissible monomorphisms $E^i\rarrow K^i$
into some objects $K^i\in \sJ$ and set $J^\bu$ to be the split exact
complex with the terms $J^i=K^i\oplus K^{i+1}$.
 The same construction of an admissible monomorphism of a complex
in $\sA$ into an exact complex in $\sJ$ allows to prove part~(e).
 Part~(f) is clear from the above argument; one only has to
modify the construction of the admissible monomorphism
$E^\bu\rarrow J^\bu$ slightly in the case of complexes bounded below
or from both sides in order to keep the complex $J^\bu$ within
the boundaries of the complex~$E^\bu$.
\end{proof}

 Now we can construct the derived functor $\boR\Psi$.
 Let $\st$ be one of the symbols $\b$, $+$, $-$, $\varnothing$,
$\abs+$, $\abs-$, $\co$, $\ctr$, or~$\abs$; and let
$\sC^\st(\sA)$ denote the category of (respectively bounded or
unbounded) complexes (and closed morphisms of degree zero between
them) in the category~$\sA$.
 Given a complex $E^\bu\in\sC^\st(\sA)$, we resolve it by a bicomplex
$J^{\bu,\bu}$ with the terms belonging to the subcategory
$\sJ\subset\sA$ as in Lemma~\ref{complex-resolutions}(a), apply
the functor $\tau_{\le d}\Psi$ to every complex $J^{j,\bu}$ obtaining
a bicomplex $\tau_{\le d}\Psi(J^{\bu,\bu})$ concentrated in the segment
of degrees $[0,d]$ along the grading~$j$, and totalize this bicomplex
in order to obtain the complex $\boR\Psi(E^\bu)\in\sD^\st(\sB)$.
 When the symbol~$\st$ presumes bounded complexes, we choose
the bicomplex $J^{\bu,\bu}$ to be bounded on the respective side(s)
along the grading~$i$, as in Lemma~\ref{complex-resolutions}(f).
 This defines the action of the functor $\boR\Psi$ on the objects
$E^\bu$ of the category of complexes $\sC^\st(\sA)$.

 To construct the image of a morphism $E^\bu\rarrow F^\bu$ in
the category of complexes $\sC^\st(\sA)$ under the functor $\boR\Psi$,
one applies Lemma~\ref{complex-resolutions}(b) to the pair of morphisms
of complexes $E^\bu\rarrow F^\bu$ and $\id: F^\bu\rarrow F^\bu$ together
with the chosen resolutions $E^\bu\rarrow K^{\bu,\bu}$ and $F^\bu
\rarrow L^{\bu,\bu}$ of the complexes $E^\bu$ and $F^\bu$ by bicomplexes
$K^{\bu,\bu}$ and $L^{\bu,\bu}$ in the exact category $\sJ$.
 Then one uses the condition that the functor $\Psi$ takes
quasi-isomorphisms in $\sC_\sA^{\ge0}(\sJ)$ to quasi-isomorphisms in
$\sC^{\ge0}(\sB)$ in order to conclude that the induced morphism between
the totalizations of the bicomplexes $\tau_{\le d}\Psi(L^{\bu,\bu})$
and $\tau_{\le d}\Psi(J^{\bu,\bu})$ has an absolutely acyclic cone.

 Proving that the morphism $\boR\Psi(E^\bu)\rarrow\boR\Psi(F^\bu)$
in the derived category $\sD^\st(\sB)$ obtained in this way depends
only on the original morphism $E^\bu\rarrow F^\bu$ in the category
$\sC^\st(\sA)$ and not on any additional choices, and also that
compositions of morphisms in the category $\sD^\st(\sB)$ are assigned
to compositions of morphisms in the category of complexes
$\sC^\st(\sA)$, reduces to checking that a morphism of resolving
bicomplexes $K^{\bu,\bu}\rarrow L^{\bu,\bu}$ in the category $\sJ$ forming
a commutative diagram with a zero morphism of complexes
$E^\bu\rarrow F^\bu$ in the category $\sA$ induces a vanishing morphism
between the totalizations of bicomplexes
$\tau_{\le d}\Psi(K^{\bu,\bu})\rarrow\tau_{\le d}\Psi(L^{\bu,\bu})$
in the derived category $\sD^\st(\sB)$.
 Here one has to apply Lemma~\ref{complex-resolutions}(c) and
recall that the functor $\Psi\:\sC^+(\sJ)\rarrow\sC^+(\sB)$, having
a DG\+functor structure, transforms contracting homotopies to
contracting homotopies and preserves the equations they may satisfy.

 A morphism of complexes in the category $\sA$ is homotopic to zero
if and only if it factorizes through a cone of an identity morphism of
complexes, i.~e., through a complex of the form $\dotsb\rarrow E^{i-1}
\oplus E^i\rarrow E^i\oplus E^{i+1}\rarrow\dotsb$.
 Complexes of this form are easily seen to be taken to zero by
the functor $\boR\Psi$, which is therefore a well-defined functor
from the homotopy category $\Hot^\st(\sA)$ to the derived category
$\sD^\st(\sB)$.

 Checking that the functor $\boR\Psi\:\Hot^\st(\sA)\rarrow\sD^\st(\sB)$
so obtained is triangulated reduces to another application of
Lemma~\ref{complex-resolutions}(b), guaranteeing any closed morphism
of complexes $E^\bu\rarrow F^\bu$ in the category $\sA$ can be included
into a commutative diagram with resolution morphisms
$E^\bu\rarrow K^{\bu,\bu}$ and $F^\bu\rarrow L^{\bu,\bu}$, and
a morphism of resolving bicomplexes $K^{\bu,\bu}\rarrow L^{\bu,\bu}$
in the category~$\sJ$.
 Since the cone of the closed morphism of bicomplexes $K^{\bu,\bu}
\rarrow L^{\bu,\bu}$ (taken along the grading~$i$) is a resolution
of the cone of the closed morphism $E^\bu\rarrow F^\bu$, and
the functor $\Psi$, being a DG\+functor, preserves the cones of
closed morphisms, the functor $\boR\Psi$ takes cones to cones.

 It remains to show that the functor $\boR\Psi$ factorizes through
the derived category $\sD^\st(\sA)$.
 Here one has to consider derived categories of the first and of
the second kind separately.
 One uses Lemma~\ref{complex-resolutions}(e) in the former case
and Lemma~\ref{complex-resolutions}(d) in the latter one, together with
the assumption that $\Psi\:\sC_\sA^{\ge0}(\sJ)\rarrow\sC^{\ge0}(\sB)$ is
an exact functor between exact categories.

 Notice that the derived functor $\boR\Psi\:\sD^\st(\sA)\rarrow
\sD^\st(\sB)$ that we have constructed does not depend on the choice
of an integer~$d$ satisfying the condition that
$\Psi(\sC_\sA^{\ge0}(\sJ))$ is contained in $\sC^{\ge0}(\sB)^{\le d}$.
 Indeed, for any bicomplex $J^{\bu,\bu}$ with $J^{\bu,i}\in
\sC_\sA^{\ge0}(\sJ)$ for all $i\in\Z$ and any integer $d'\ge d$
a cone of the natural morphism between the totalizations of bicomplexes
$\tau_{\le d}\Psi(J^{\bu,\bu})\rarrow\tau_{\le d'}\Psi(J^{\bu,\bu})$ is
absolutely acyclic.

\bigskip

 In the rest of this appendix, our aim is to establish an adjunction
property of the left and right derived functors provided by the above
construction.
 For any additive category $\sE$, we denote by $\sC^-(\sE)$
the DG\+category of bounded above complexes in~$\sE$.
 When $\sE$ is an exact category, the subcategory $\sC^{\le0}(\sE)
\subset\sC^-(\sE)$ of all the complexes $\dotsb\rarrow E^{-2}\rarrow
E^{-1}\rarrow E^0\rarrow0$ and closed morphisms of degree zero between
them is endowed with a natural exact category structure where
a short sequence of complexes is exact if and only if it is
termwise exact in~$\sE$.

 Let $\sA$ and $\sB$ be two exact categories containing the images
of idempotent endomorphisms of their objects.
 Let $\sJ\subset\sA$ be a full subcategory, closed under extensions
and the passages to the cokernels of admissible monomorphisms in $\sA$
and such that any object of $\sA$ is the source of an admissible
monomorphism into an object of~$\sJ$.
 Similarly, let $\sP\subset\sB$ be a full subcategory, closed under
extensions and the passages to the kernels of admissible epimorphisms
and such that any object of $\sB$ is the target of an admissible
epimorphism from an object of~$\sP$.
 The full subcategories $\sJ$ and $\sP$ inherit the exact category
structures of the ambient categories $\sA$ and~$\sB$.

 Denote by $\sC_\sP^{\le 0}(\sB)$ the full subcategory in
$\sC^{\le0}(\sP)$ consisting of all the complexes $\dotsb\rarrow P^{-2}
\rarrow P^{-1}\rarrow P^0\rarrow0$ in $\sP$ for which there exits
an object $B\in\sB$ together with a morphism $P^0\rarrow B$ such that
the sequence $\dotsb\rarrow P^{-1}\rarrow P^0\rarrow B\rarrow0$ is
exact in~$\sB$.
 The full subcategory $\sC_\sB^{\le0}(\sP)$ is closed under extensions
and the passages to the kernels of admissible epimorphisms in
the exact category $\sC^{\le0}(\sP)$, so it inherits the exact category
structure.

 Let $\sC^{\le0}(\sA)^{\ge -d}\subset\sC^{\le0}(\sA)$ denote the full
subcategory consisting of all the complexes $\dotsb\rarrow A^{-1}
\rarrow A^0\rarrow 0$ in $\sA$ such that the sequence $\dotsb\rarrow
A^{-d-2}\rarrow A^{-d-1}\rarrow A^{-d}$ is exact in~$\sA$.
 Being closed under extensions and the passages to the kernels of
admissible epimorphisms, the full subcategory $\sC^{\le0}(\sA)^{\ge -d}$
inherits the exact category structure of the category $\sC^{\le0}(\sA)$.
 Denote by $\sC^{[-d,0]}(\sA)$ the exact category of finite complexes
$A^{-d}\rarrow\dotsb\rarrow A^0$ in~$\sA$.
 Then there is an exact functor of canonical truncation
$\tau_{\ge-d}\:\sC^{\le0}(\sA)^{\ge-d}\rarrow\sC^{[-d,0]}(\sA)$.

 As above, let $\Psi\:\sC^+(\sJ)\rarrow\sC^+(\sB)$ be a DG\+functor
taking quasi-isomorphisms of complexes belonging to
$\sC_\sA^{\ge0}(\sJ)$ to quasi-isomorphisms of complexes in the exact
category~$\sB$.
 Assume that the restriction of the functor $\Psi$ to the subcategory
$\sC_\sA^{\ge0}(\sJ)\subset\sC^+(\sJ)$ is an exact functor between
exact categories $\Psi\:\sC_\sA^{\ge0}(\sJ)\rarrow
\sC^{\ge0}(\sB)^{\le d}$.
 Then the above construction provides the right derived functors
$\boR\Psi\:\sD^\st(\sA)\rarrow\sD^\st(\sB)$.

 Similarly, let $\Phi\:\sC^-(\sP)\rarrow\sC^-(\sA)$ be a DG\+functor
taking quasi-isomorphisms of complexes belonging to
$\sC_\sB^{\le0}(\sP)$ to quasi-isomorphisms of complexes in the exact
category~$\sA$.
 Assume that the restriction of the functor $\Phi$ to the subcategory
$\sC_\sB^{\le0}(\sP)\subset\sC^-(\sP)$ is an exact functor between
exact categories $\Phi\:\sC_\sB^{\le0}(\sP)\rarrow
\sC^{\le0}(\sA)^{\ge-d}$.
 Composing the restriction of the functor $\Phi$ with the functor
of canonical truncation, we obtain an exact functor
$$
 \tau_{\ge-d}\Phi\:\sC^{\le0}_\sB(\sP)\lrarrow\sC^{[-d,0]}(\sA).
$$
 The construction dual to the above provides a left derived functor
$$
 \boL\Phi\:\sD^\st(\sB)\lrarrow\sD^\st(\sA)
$$
for any symbol $\st=\b$, $+$, $-$, $\varnothing$, $\abs+$, $\abs-$,
or~$\abs$.
 When the exact categories $\sA$ and $\sB$ have exact functors of
infinite direct sum (resp., product), the full subcategory
$\sP\subset\sB$ is closed under infinite direct sums (resp.,
products), and the functor $\Phi$ preserves infinite direct sums
(resp., products), there is also the derived functor $\boL\Phi$
acting between the coderived categories $\sD^\co$ (resp.,
the contraderived categories $\sD^\ctr$) of the exact categories
$\sB$ and~$\sA$.

 Denote by $\sC(\sA)$ and $\sC(\sB)$ the DG\+categories of unbounded
complexes in the additive categories $\sA$ and~$\sB$.
 Suppose that the DG\+functors $\Psi\:\sC^+(\sJ)\rarrow\sC^+(\sB)$
and $\Phi\:\sC^-(\sP)\rarrow\sC^-(\sA)$, viewed as partially defined
DG\+functors between the DG\+categories $\sC(\sA)$ and $\sC(\sB)$,
are partially adjoint to each other.
 In other words, this means that for any complexes $J^\bu\in\sC^+(\sJ)$
and $P^\bu\in\sC^-(\sP)$ there is an isomorphism
$$
 \Hom^\bu_{\sC(\sA)}(\Phi(P^\bu),J^\bu)\simeq
 \Hom^\bu_{\sC(\sB)}(P^\bu,\Psi(J^\bu))
$$
of complexes of morphisms in the DG\+categories $\sC(\sA)$ and
$\sC(\sB)$, functorial in the complexes $J^\bu$ and~$P^\bu$.
 Let us show that in this case the derived functor $\boL\Phi$ is
left adjoint to the derived functor $\boR\Psi$.

 Our approach consists in constructing the adjunction morphisms
$$
 \boL\Phi\circ \boR\Psi\rarrow\Id
 \quad\text{and}\quad
 \Id\rarrow\boR\Psi\circ\boL\Phi.
$$
 Since the functors $\Phi$ and $\Psi$, being DG\+functors, commute
with totalizations of finite complexes of complexes in $\sC^-(\sP)$
and $\sC^+(\sJ)$, the question essentially reduces to the case
of single objects of the categories $\sA$ and $\sB$ viewed as
one-term complexes along the grading~$i$.
 So in the reasoning below we ignore the (generally speaking, doubly
unbounded) grading~$i$ altogether and concentrate on the (bounded
on one side or finite) gradings created by the resolution
procedures.

 Let $0\rarrow J^0\rarrow J^1\rarrow\dotsb$ be a right resolution of
an object $A\in\sA$ by objects $J^j\in\sJ$.
 Denote by $B^1\rarrow\dotsb\rarrow B^d$ the finite complex
$\tau_{\le d}\Psi(J^\bu)$ in the category~$\sB$.
 Let $\dotsb\rarrow P^{-1,\bu}\rarrow P^{0,\bu}\rarrow B^\bu\rarrow0$
be an exact sequence of complexes in the category $\sB$ with
the terms $P^{k,j}\in\sP$ concentrated in the interval of gradings
$0\le j\le d$.
 Then the total complex of the bicomplex
$\tau_{\ge -d}\Phi(P^{\bu,\bu})$, where both the functor $\Phi$ and
the canonical truncation are applied to complexes along
the grading~$k$, represents the object $\boL\Phi(\boR\Psi(A))$
in the derived category $\sD^\st(\sA)$.

 The composition $P^{\bu,\bu}\rarrow B^\bu\rarrow\Psi(J^\bu)$ is
a closed morphism from the total complex of the bicomplex
$P^{\bu,\bu}$ to the complex $\Psi(J^\bu)$ in the category~$\sB$.
 Consider the corresponding closed morphism of complexes
$\Phi(P^{\bu,\bu})\rarrow J^\bu$ in the category~$\sA$.
 Since the bicomplex $P^{k,j}$ is concentrated in the cohomological
gradings $0\le j\le d$ and $k\le 0$, while the complex $J^\bu$
is concentrated in the cohomological gradings $j\ge0$, the latter
morphism factorizes through the admissible epimorphism between
the totalizations of bicomplexes $\tau_{\ge -d}\Phi(P^{\bu,\bu})\rarrow
\Phi(P^{\bu,\bu})$ and the admissible monomorphism of complexes
$\tau_{\le d}J^\bu\rarrow J^\bu$.
 We have obtained a closed morphism $\tau_{\ge -d}\Phi(P^{\bu,\bu})
\rarrow\tau_{\le d}J^\bu$ of complexes in the category $\sA$
representing the desired adjunction morphism
$\boL\Phi(\boR\Psi(A))\rarrow A$.

 The adjunction morphism $\id\rarrow\boR\Psi\circ\boL\Phi$ of
functors on the category $\sD^\st(\sB)$ is constructed similarly.
 Checking that the compositions $\boR\Psi\rarrow\boR\Psi\circ\boL\Phi
\circ\boR\Psi\rarrow\boR\Psi$ and $\boL\Phi\rarrow
\boL\Phi\circ\boR\Psi\circ\boL\Phi\rarrow\boL\Phi$
are the identity morphisms is straightforward.

\end{document}